\documentclass[10pt,a4paper]{article}

\usepackage[dvips]{graphicx}
\usepackage{psfrag}
\usepackage{amsmath}
\usepackage{fancybox}
\usepackage{amsfonts}
\usepackage{amsthm}
\usepackage{amssymb}
\usepackage{enumerate}
\theoremstyle{plain}
   \newtheorem{theorem}{Theorem}[section]
   
   \newtheorem{lemma}[theorem]{Lemma}
   \newtheorem{proposition}[theorem]{Proposition}
\theoremstyle{definition}
   \newtheorem{definition}[theorem]{Definition}
   \newtheorem{remark}[theorem]{Remark}
\theoremstyle{remark}
   
   \newtheorem{example}[theorem]{Example}

\usepackage[%
   pdfpagemode=UseNone,
   bookmarks=true,
   pagebackref=false, % referencias a citas
   colorlinks,
   linkcolor=blue,
   anchorcolor=blue,
   citecolor=blue,
   filecolor=blue,
   urlcolor=blue
   ]{hyperref}

\usepackage{color}

\newcommand{\h}{H}

\usepackage{xspace}

\newcommand{\definedas}{:=}
\newcommand{\asdefined}{=:}
\newcommand{\D}{\displaystyle}

\newcommand{\normT}[1]{\left\| {#1} \right\|_{T}}

\newcommand{\normS}[1]{\left\| {#1} \right\|_{S}}
\newcommand{\normbT}[1]{\left\| {#1} \right\|_{\partial T}}

\DeclareMathOperator{\diam}{diam}

\DeclareMathOperator{\e}{e}

\newcommand\FF{\mathcal F}
\newcommand\JJ{\mathcal J}

\newcommand\NN{\mathbb N}

\newcommand\AAA{\mathcal A}

\newcommand\MM{\mathcal M}
\newcommand\Tau{\mathcal T}
\newcommand\PP{\mathcal P}
\newcommand\TT{\mathbb T}
\newcommand\VV{\mathbb V}

\newcommand\RRR{\mathbf R}

\newcommand\CCC{\mathcal C}
\newcommand{\RR}{\mathbb R}

\begin{document}

\title{Convergence of an adaptive Ka\v canov FEM for quasi-linear problems}
%A convergent adaptive algorithm of Ka\v{c}anov type for quasi-linear problems}

\author{Eduardo M. Garau \and Pedro Morin \and Carlos Zuppa}

\maketitle

\begin{abstract}
We design an adaptive finite element method to approximate the solutions of quasi-linear elliptic problems. The algorithm is based on a Ka\v canov iteration and a mesh adaptation step is performed after each linear solve. The method is thus \emph{inexact} because we do not solve the discrete nonlinear problems exactly, but rather perform one iteration of a fixed point method (Ka\v canov), using the approximation of the previous mesh as an initial guess. The convergence of the method is proved for any \emph{reasonable} marking strategy and starting from any initial mesh. We conclude with some numerical experiments that illustrate the theory.
\end{abstract}

\begin{quote}\small
\end{quote}

%%===============================
\section{Introduction}
%%===============================

In this paper we consider quasi-linear elliptic partial differential equations over a polygonal/polyhedral domain $\Omega\subset \RR^d$ ($d=2,3$)  of the form
\begin{equation}\label{E:ley de conservacion}
\left\{
\begin{aligned}
-\nabla\cdot \big[\alpha(\,\cdot\,,|\nabla u|^2)\nabla u\big]&= f\qquad & & \text{in}\,\Omega\\
u&= 0\qquad & &\text{on}\,\partial \Omega,
\end{aligned}
\right.
\end{equation}
where $\alpha:\Omega\times \RR_+\to \RR_+$ is a function whose properties will be stated in Section~\ref{S:setting} below, and $f\in L^2(\Omega)$ is given. %\footnote{If we define $\beta(t)\definedas\frac12\int_0^{t^2}\alpha(s)~ds$, we have that $\beta''(t)=2\alpha'(t^2)t^2+\alpha(t^2)$, and one of the hypotheses become $\beta''(t)\ge cte>0$, for all $t>0$, wich is reasonable from a physical point of view.} 
These equations describe stationary conservation laws which frequently arise in problems of mathematical physics~\cite{Zeidler}. 
%In the particular case where $\alpha(x,t)=1$ for all $x\in\Omega$ and $t>0$, the problem~\eqref{E:ley de conservacion} is the classical Dirichlet's problem. Thus, from a mathematical point of view, the problem~\eqref{E:ley de conservacion} is a first nonlinear generalization of this classical problem; and from a physical point of view, this model covers more general cases.
%Equations like~\eqref{E:ley de conservacion} include several physical problems which are completely different. 
For example, in hydrodynamics and gas dynamics (subsonic and supersonic flows), electrostatics, magnetostatics, heat conduction, elasticity and plasticity (e.g., plastic torsion of bars), etc. Some of these examples are better modeled by variational inequalities (see~\cite{HJS} and the references therein), but these fall beyond the scope of this article, which attempts to set a first step towards understanding the convergence and optimality of inexact Ka\v canov-type iterations, in the context of adaptive finite element methods.

For a summary of convergence and optimality results of AFEM we refer the  reader to the survey~\cite{NSV-survey}, and the references therein. We restrict ourselves to those references strictly related to our work.

\emph{Inexact} adaptive finite element methods have been considered for Stokes problem in~\cite{BMN,KS} using Uzawa's algorithm. Briefly, a Richardson iteration is applied to the infinite-dimensional Schur complement operator and in each iteration, the elliptic problem is solved up to a decreasing tolerance. In~\cite{BMN} linear convergence is proved and in~\cite{KS} the optimality of the method in terms of degrees of freedom is proved, after adding some new refinement steps to the algorithm.

In~\cite{DK}, a contraction property is proved for an adaptive algorithm based on D\"orfler's marking strategy~\cite{Doerfler} for nonlinear problems of the type~\eqref{E:ley de conservacion}, using Orlicz norms to cope with the very mild assumptions on the nonlinear term $\alpha$.

In this work we impose stronger assumptions on $\alpha$, which guarantee the convergence of Ka\v canov's iteration~\cite{HJS}. More precisely, we assume that $\alpha(\cdot,\cdot)$ is decreasing with respect to its second variable (cf.~\eqref{alpha-decreasing}), and fulfills condition~\eqref{E:betasegunda acotada} below; which are the same assumptions stated in~\cite{HJS}, where they consider the iteration on a fixed space, either finite- or infinite-dimensional. Our focus is the convergence analysis of the adaptive method that results from performing one mesh adaptation in each iteration of the non-linear solver. This turns out to be a very realistic and efficient method, based on the sole assumption that a \emph{linear} system is exactly solved in each iteration step.

This paper is organized as follows. In Section~\ref{S:setting} we present the  class of specific nonlinear problems that we study, and some of its properties. In Section~\ref{S:adaptive loop}, we present the inexact adaptive Ka\v canov algorithm and in Section~\ref{S:convergence} we state and prove the main result of this article, namely the convergence of the discrete solutions produced by the algorithm to the exact solution of the nonlinear problem. Finally, in Section~\ref{S:numerical experiments}, we present some numerical experiments which illustrate the theory, and explore the practical boundaries of applicability of the algorithm.

\section{Setting}\label{S:setting}

Let $\Omega\subset\RR^d$ be a bounded polygonal ($d = 2$) or polyhedral ($d = 3$) domain with Lipschitz boundary. A weak formulation of~\eqref{E:ley de conservacion} consists in finding $u\in H^1_0(\Omega)$ such that
\begin{equation}\label{E:cont_prob_nolineal}
a(u;u,v)=L(v),\qquad\forall\,v\in H^1_0(\Omega),
\end{equation}
where
\begin{equation*}%\label{E:forma_a}
a(w;u,v)=\int_\Omega \alpha(\,\cdot\,,|\nabla w|^2)\nabla u\cdot \nabla v,\qquad\forall\,w,u,v\in H^1_0(\Omega),
\end{equation*}
and
\begin{equation*}%\label{E:Lf_fuente}
L(v)=\int_\Omega fv,\qquad\forall\,v\in H^1_0(\Omega).
\end{equation*}
%para alguna $f\in L^2(\Omega)$ dada.
We require that $\alpha$ is locally Lipschitz in its first variable, uniformly with respect to its second variable (cf.~\eqref{E:alpha localmente lipschitz} below). On the other hand, we assume that $\alpha$ is $\mathcal{C}^1$ as a function of its second variable and there exist positive constants $c_a$ and $C_a$ such that
\begin{equation}\label{E:betasegunda acotada} 
c_a\le \alpha(x,t^2)+2t^2D_2\alpha(x,t^2)\le C_a,\qquad\forall\,x\in\Omega,\,t>0,
\end{equation} 
where $D_2\alpha$ denotes the partial derivative of $\alpha$ with respect to its second variable.
The last condition is related to the well-posedness of problem~\eqref{E:cont_prob_nolineal} as we will show below. Also, it is possible to prove that~\eqref{E:betasegunda acotada} implies that
\begin{equation}\label{E:acotacion de alfa}
c_a\le\alpha(x,t)\le C_a,\qquad\forall\,x\in\Omega,\,t>0.
\end{equation}
These same assumptions are stated in~\cite{HJS}, where several interesting applied problems are shown to satisfy them.

%Definamos $\VV\definedas H^1_0(\Omega)$ y consideremos  $\|v\|_\VV\definedas\|\nabla v\|_{L^2(\Omega)}$, para toda $v\in\VV$. 
Additionally, it is easy to check that the form $a$ is linear and symmetric in its second and third variables. Also, from~\eqref{E:acotacion de alfa} it follows that $a$ is bounded,
\begin{equation}\label{E:a_acotada}
|a(w;u,v)|\le C_a \|\nabla u\|_\Omega\|\nabla v\|_\Omega,\qquad\forall\,w,u,v\in H^1_0(\Omega),
\end{equation}
and coercive,
\begin{equation}\label{E:a_coercitiva}
 c_a\|\nabla u\|_\Omega^2\le a(w;u,u),\qquad\forall\,w,u\in H^1_0(\Omega).
\end{equation}

If we define $A:H^1_0(\Omega)\to H^{-1}(\Omega)$ as the nonlinear operator given by
\begin{equation*}
\langle Au,v\rangle \definedas a(u;u,v),\qquad\forall\,u,v\in H^1_0(\Omega),
\end{equation*}
then problem~\eqref{E:cont_prob_nolineal} is equivalent to the equation
$$Au=L,$$
where $L \in H^{-1}(\Omega)$ is given. Assumption~\eqref{E:betasegunda acotada} implies that $A$ is Lipschitz and strongly monotone (see~\cite{Zeidler}), i.e., there exist positive constants $C_A$ and $c_A$ such that
\begin{equation}\label{E:A Lipschitz}
\|Au-Av\|_{H^{-1}(\Omega)}\le C_A \|\nabla (u -v)\|_\Omega,\qquad \forall\,u,v\in H^1_0(\Omega),
\end{equation}
and
\begin{equation}\label{E:A fuertemente monotono}
\langle Au-Av,u-v\rangle\ge c_A\|\nabla (u -v)\|_\Omega^2,\qquad \forall\,u,v\in H^1_0(\Omega).
\end{equation}
As a consequence of~\eqref{E:A Lipschitz} and~\eqref{E:A fuertemente monotono}, problem~\eqref{E:cont_prob_nolineal} has a unique solution and is stable~\cite{Zarantonello,Zeidler}. % (Zarantonello (1960)).

% El siguiente teorema da un control para el error de la sucesi\'on definida en~\eqref{E:kachanov} en t\'erminos de dos aproximaciones sucesivas.
% 
% \begin{theorem}[Estimaci\'on del error]\label{T:kachanov estimacion del error}
% Sea $u\in\VV$ la soluci\'on del Problema~\ref{E:cont_prob_nolineal}. Si $\{U_k\}_{k\in\NN_0}\subset\VV$ es la sucesi\'on definida por~\eqref{E:kachanov}, entonces
% $$\|U_k-u\|_\VV\le\frac{C_a}{c_A}\|U_k-U_{k+1}\|_\VV,\qquad\forall\,k\in\NN_0.$$
% \end{theorem}
% 
% \begin{proof}
% Sea $u\in\VV$ la soluci\'on del Problema~\ref{E:cont_prob_nolineal} y sea $\{U_k\}_{k\in\NN_0}\subset\VV$ la sucesi\'on definida por~\eqref{E:kachanov}. Puesto que el operador $A$ asociado al Problema~\ref{E:cont_prob_nolineal} es fuertemente mon\'otono, usando~\eqref{E:kachanov} y que $a$ es acotada se tiene que
% \begin{align*}
% c_A\|U_k-u\|_\VV^2&\le \langle AU_k -Au, U_k-u\rangle= a(U_k;U_k,U_k-u)-a(u;u,U_k-u)\\
% &=a(U_k;U_k,U_k-u)-a(U_k;U_{k+1},U_k-u)=a(U_k;U_k-U_{k+1},U_k-u)\\
% &\le C_a\|U_k-U_{k+1}\|_\VV\|U_k-u\|_\VV,
% \end{align*}
% de donde sigue la afirmaci\'on del teorema.
% \end{proof}
% 
% \begin{remark}[Convergencia del M\'etodo de Ka\v{c}anov]\label{R:kakanov converge}
% Para probar la convergencia de $\{U_k\}_{k\in\NN_0}$ a $u$, del Teorema~\ref{T:kachanov estimacion del error} se sigue que es suficiente mostrar que
% $$\|U_k-U_{k+1}\|_\VV\to 0,$$
% cuando $k$ tiende a infinito. \'Esto puede demostrarse como en el Lema~\ref{L:sucesivas a cero} de la secci\'on siguiente, suponiendo la Hip\'otesis~\ref{A:key property} enunciada debajo.
% \end{remark}

\section{%A posteriori error estimators and 
Adaptive algorithm}\label{S:adaptive loop}

In order to define discrete adaptive approximations to problem~\eqref{E:cont_prob_nolineal} we will consider \emph{triangulations} of the domain $\Omega$. Let $\Tau_0$ be an
initial conforming triangulation of $\Omega$, that is, a partition
of $\Omega$ into $d$-simplices such that if two elements intersect,
they do so at a full vertex/edge/face of both elements. Let
$\TT$ denote the set of all conforming triangulations of
$\Omega$ obtained from $\Tau_0$ by refinement using the bisection procedures presented by Stevenson~\cite{Stevenson-refine}. These coincide (after some re-labeling) with the so-called \textit{newest vertex} bisection procedure in two dimensions and the bisection procedure of Kossaczk\'y in three
dimensions~\cite{Alberta}.

Due to the processes of refinement used, the family $\TT$ is shape regular, i.e., 
$$\sup_{\Tau\in\TT}\,  \sup_{T \in \Tau} \frac{\diam(T)}{\rho_T} \asdefined \kappa_\TT <\infty,$$
where $\diam(T)$ is the diameter of $T$, and $\rho_T$ is the radius of the largest ball contained in it. Throughout this article, we only consider meshes $\Tau$ that belong to the family $\TT$, so the shape regularity of all of them is bounded by the uniform constant $\kappa_\TT$ which only depends on the initial triangulation $\Tau_0$~\cite{Alberta}. Also, the diameter of any element $T\in\Tau$ is equivalent to the local mesh-size $H_T\definedas|T|^{1/d}$, which in turn defines the global mesh-size $H_\Tau \definedas \D\max_{T\in\Tau} H_T$.

For the discretization we consider the Lagrange finite element spaces consisting of continuous functions vanishing on $\partial \Omega$ which are piecewise linear over a mesh $\Tau\in\TT$, i.e.,
\begin{equation}\label{E:V_Tau}
\VV_{\Tau}\definedas\{v\in H^1_0(\Omega)\mid\quad v_{|_T}\in \PP_1(T),\quad \forall~T\in\Tau\}.
\end{equation}
We are now in position to state the adaptive loop to approximate the solution $u$ of the problem~\eqref{E:cont_prob_nolineal}.

\medskip
\begin{center}
\doublebox{ % otras opciones: doublebox, ovalbox, Ovalbox
\begin{minipage}{.9\textwidth}
\textbf{Adaptive Algorithm.}
\tt
Let $\Tau_0$ be an initial conforming mesh of $\Omega$ and $u_0\in\VV_{\Tau_0}$ be an initial approximation. Let $\Tau_1=\Tau_0$ and $k=1$.
\begin{enumerate}
\vspace{0.15cm}
 \item [1.] $u_k\definedas \textsf{SOLVE}(u_{k-1},\Tau_k)$.
\vspace{0.15cm}
\item [2.] $\{\eta_k(T)\}_{T\in\Tau_k}\definedas \textsf{ESTIMATE}(u_{k-1},u_k,\Tau_k)$.
\vspace{0.15cm}
\item [3.] $\MM_k\definedas \textsf{MARK}(\{\eta_k(T)\}_{T\in\Tau_k},\Tau_k)$.
\vspace{0.15cm}
\item [4.] $\Tau_{k+1}\definedas \textsf{REFINE}(\Tau_k,\MM_k,n)$.
\vspace{0.15cm}
\item [5.] Increment $k$ and go back to step 1.
\end{enumerate}
\end{minipage}
}
\end{center}

We now explain each module of the last algorithm in detail.

\begin{description}
\item[\textbf{The module \textsf{SOLVE}.}]
 Given the conforming triangulation $\Tau_k$ of $\Omega$, and the solution $u_{k-1}$ from the previous iteration, the module \textsf{SOLVE} outputs the solution $u_k\in\VV_k\definedas \VV_{\Tau_k}$ of the \emph{linear} problem
\begin{equation}\label{E:kdisc-problem-nolineal inexacto}
a(u_{k-1};u_k,v_k)=L(v_k),\qquad \forall~v_k\in  \VV_k.
\end{equation}
\end{description}

\begin{remark}[Linear vs.\ nonlinear] 
Notice that while \textsf{SOLVE} requires the solution of a linear system, the usual discretization of~\eqref{E:cont_prob_nolineal} in $\VV_k$ consists in finding $u_k\in \VV_k$ such that
\begin{equation}\label{E:kdisc-problem-nolineal}
a(u_k;u_k,v_k)=L(v_k),\qquad \forall~v_k\in  \VV_k,
\end{equation}
which is nonlinear. We propose here to make only one step of a fixed point iterative method to solve the nonlinear problem, and proceed to the mesh adaptation, whereas the usual approach would be to approximate the discrete nonlinear problem up to a very fine accuracy (pretending to have it exactly solved) before proceeding to mesh adaptation~\cite{GMZ-optimality-nonlinear}. Each iteration of the adaptive loop is thus much cheaper in our approach. Our theory guarantees convergence of this algorithm, and the numerical experiments of Section~\ref{S:numerical experiments} show that the convergence is quasi-optimal, although the latter is not yet rigorously proved.

\end{remark}

%Note that solving~\eqref{E:kdisc-problem-nolineal} implies to solve a \emph{nonlinear} discrete system (see~\cite{GMZ-optimality-nonlinear}). Instead, we assume that $u_k$ is, in fact, the solution of a \emph{linear} discrete system. More precisely, the module \textsf{SOLVE} takes the space $\VV_k$ and the discrete solution $u_{k-1}$ computed in the last step as input arguments and computes the solution $u_k\in\VV_k$ of the \emph{linear} problem
%\begin{equation}\label{E:kdisc-problem-nolineal inexacto}
%a(u_{k-1};u_k,v_k)=L(v_k),\qquad \forall~v_k\in  \VV_k.
%\end{equation}

\begin{remark}[Ka\v{c}anov's Method]
Let us consider problem~\eqref{E:cont_prob_nolineal} (resp.\ problem~\eqref{E:kdisc-problem-nolineal} with $k\in\NN$ fixed). We denote the space $H^1_0(\Omega)$ (resp.\ $\VV_k$) by $\VV$, and the solution $u$ (resp.\  $u_k$) by $U$. Given an initial approximation $U_0\in\VV$ of the solution $U$, we consider the sequence $\{U_i\}_{i\in\NN_0}$ where $U_i\in\VV$ is the solution of the \emph{linear} problem
\begin{equation*}%\label{E:kachanov}
a(U_{i-1};U_i,v)=L(v),\qquad\,\forall\,v\in\VV,\quad i \in \NN.
\end{equation*}

This is known as Ka\v{c}anov's Method, and it follows~\cite{HJS} that the sequence $\{U_i\}_{i\in\NN_0}$ converges to the solution $U$, provided $D_2\alpha(x,t)\le 0$ for all $x\in\Omega$ and $t>0$.

Notice that our algorithm consists in performing \emph{only one} step of Ka\v{c}anov iteration in each step of the adaptive loop.
\end{remark}

\begin{description}
\item[\textbf{The module \textsf{ESTIMATE}.}] 
Given $\Tau_k$, $u_{k-1}$ and the corresponding output $u_k$ of \textsf{SOLVE}, the module \textsf{ESTIMATE} computes and outputs the local error estimators $\{\eta_k(T)\}_{T\in\Tau_k}$ given by
\begin{equation}\label{E:estimadores kakanov def}
\eta_k^2(T)\definedas \h_T^2\normT{R_k}^2 + \h_T\normbT{J_k}^2,
\end{equation}
for all $T\in\Tau_k$.
Here $R_k$ denotes the \emph{element residual} given by \begin{equation*}%\label{E:element-residual}
{R_k}_{|_{T}}\definedas -\nabla \cdot \big[\alpha(\,\cdot \,,|\nabla u_{k-1}|^2)\nabla u_k\big]-f,\qquad\forall\,T\in\Tau_k,
\end{equation*}
and $J_k$ denotes the \emph{jump residual} given by
\begin{equation*}%\label{E:jump-residual}
{J_k}_{|_{S}}\definedas\frac12\left[(\alpha(\,\cdot \,,|\nabla u_{k-1}|^2)\nabla u_k)_{|_{T}}\cdot \vec{n}+(\alpha(\,\cdot \,,|\nabla u_{k-1}|^2)\nabla u_k)_{|_{T'}} \cdot\vec{n}' \right],
\end{equation*}
for each interior side $S$, and ${J_k}_{|_{S}}\definedas 0$, if $S$ is a side lying on the boundary of $\Omega$. Here, $T$ and $T'$ denote the elements of $\Tau_k$ sharing $S$, and $\vec{n}$, $\vec{n}'$ are the outward unit normals of $T$, $T'$ on $S$, respectively.
\end{description}

The squared sum of the a posteriori error estimators is an upper bound for the \emph{residual} $\RRR(u_k) \in H^{-1}(\Omega)$ of $u_k$ which is defined as
\begin{equation*}
\langle \RRR(u_k),v\rangle\definedas a(u_{k-1};u_k,v)-L(v)=\int_\Omega \left(\alpha(\,\cdot\,,|\nabla u_{k-1}|^2)\nabla u_k\cdot \nabla v- fv\right),
\end{equation*}
for $v\in H^1_0(\Omega)$.
In fact, integrating by parts on each $T\in\Tau_k$ we have that
\begin{equation}\label{E:relacion fundamental}
\langle\RRR(u_k),v\rangle=\sum_{T\in\Tau_k} \left(\int_T R_kv +\int_{\partial T} J_kv\right),
\end{equation}
% \begin{equation*}
% \eta_k^2\definedas \sum_{T\in\Tau} \eta_k^2(T).
% \end{equation*}
and since $\langle\RRR(u_k),v_k\rangle =0$ for $v_k\in\VV_k$, using~\eqref{E:relacion fundamental} and interpolation estimates, we arrive at the following upper bound:% \emph{reliability} property:
\footnote{From now on, we will write $a \lesssim b$ to indicate that $a \le C b$ with $C>0$ a constant depending on the data of the problem and possibly on shape regularity $\kappa_\TT$ of the meshes. %Also $a \simeq b$ will indicate that $a \lesssim b$ and $b \lesssim a$.
}
\begin{equation}\label{E:local-residual-upperbound}
|\langle\RRR(u_k),v\rangle|\lesssim\sum_{T\in\Tau_k} \eta_k(T)\|\nabla v\|_{\omega_k(T)},\quad\forall\,v\in H^1_0(\Omega),
\end{equation} 
where $\omega_k(T)$ is the union of $T$ and its neighbors in $\Tau_k$. 

The next result is some kind of \emph{stability} result for the estimators, which is a property somewhat weaker than the usual \emph{efficiency}, but sufficient to guarantee convergence of our adaptive algorithm (see also~\cite{Siebert-convergence,GM-09}). In order to prove it, we assume that $\alpha$ is locally Lipschitz in its first variable, uniformly with respect to its second variable, as we mentioned before. More precisely, we assume that there exists a constant $C_\alpha>0$ such that
\begin{equation}\label{E:alpha localmente lipschitz}
|\alpha(x,t)-\alpha(y,t)|\le C_\alpha |x-y|,\qquad\forall\,x,y\in T,\,t>0,
\end{equation}
for each $T\in\Tau_0$.

%Now we prove the following

\begin{proposition}[Stability of the local error estimators]\label{P:estabilidad kakanov}
Let $\{u_k\}_{k\in\NN}$ be the sequence of discrete solutions computed with the Adaptive Algorithm. Then, the local error estimators given by~\eqref{E:estimadores kakanov def} are stable. More precisely, there holds $$\eta_k(T)\lesssim\|\nabla u_k\|_{\omega_k(T)}+\|f\|_{T},\qquad\forall\, T\in\Tau_k,$$
for all $k\in\NN$.
%para alguna $C=C(\kappa,d,\ell,\underline{a},\overline{a},C_\AAA,C_\alpha,C_A)>0$.
\end{proposition}

\begin{proof}
Let $\{u_k\}_{k\in\NN}$ be the sequence computed with the Adaptive Algorithm. Let $k\in\NN$ and $T\in\Tau_k$ be fixed. On the one hand, using that ${u_k}_{|T}$ is linear, and thus $\Delta u_k = 0$ inside $T$, we have that
\begin{align*}
\normT{R_k} &= \normT{-\nabla \cdot \big[\alpha(\cdot,|\nabla u_{k-1}|^2)\nabla u_k\big]-f } 
\\
&\le \normT{\nabla \big[\alpha(\cdot,|\nabla u_{k-1}|^2)\big]\cdot\nabla u_k}+\normT{f}.
\end{align*}
Since $\alpha$ is locally Lipschitz in its first variable (cf.~\eqref{E:alpha localmente lipschitz}), it follows that if $\xi\definedas\nabla u_{k-1}$ (constant over $T$),
\begin{align*}
\left|\frac{\partial}{\partial x_i}\alpha(x,|\xi|^2)\right|&=\left|\frac{\partial\alpha}{\partial x_i}(x,|\xi|^2)\right|\le C_\alpha,\qquad \forall\, x\in T,\quad 1\le i\le d,
%&\le C_\alpha+C_\AAA \underline{a}^{-1}\left|D_2\alpha(x,|\xi_{k-1}|_{\AAA(x)}^2)\right||\xi_{k-1}|_{\AAA(x)}^2\\
%&\le C_\alpha+C_\AAA \underline{a}^{-1}\frac12(\tilde{C}_A-\tilde{c}_A),
\end{align*}
and thus,
\begin{align*}
\normT{R_k} &\lesssim  \normT{\nabla u_k}+\normT{f}.
\end{align*}

On the other hand, if $S$ is a side of $\Tau_k$ shared by the elements $T,T' \in \Tau_k$,
% and $T'$ of $\Tau_k$, %, elements interior of $\Omega$, and if $T'$ is the element of $\Tau_k$ which shares~$S$, 
\begin{align*}
\normS{J_k} &\le 
\normS{ (\alpha(\,\cdot \,,|\nabla u_{k-1}|^2)\nabla u_k)_{|_{T}}}
+
\normS{(\alpha(\,\cdot \,,|\nabla u_{k-1}|^2)\nabla u_k)_{|_{T'}} }
\\
&\lesssim \normS{{\nabla u_k}_{|_{T}}}+\normS{{\nabla u_k}_{|_{T'}}}
\\
&\lesssim  \h_T^{-1/2}\|\nabla u_k\|_{T}+\h_{T'}^{-1/2}\|\nabla u_k\|_{T'}\lesssim \h_T^{-1/2}\|\nabla u_k\|_{T\cup T'},
\end{align*}
where we have used~\eqref{E:acotacion de alfa} and a scaled trace theorem. Therefore,
\begin{equation*}
\h_T^{1/2}\normbT{J_k}\lesssim\|\nabla u_k\|_{\omega_k(T)},
\end{equation*}
which completes the proof.
\end{proof}

\begin{description}
\item[\textbf{The module \textsf{MARK}.}] Based on the local error indicators, the module \textsf{MARK} selects a subset $\MM_k$ of $\Tau_k$, using any of the so-called \emph{reasonable} marking strategies, such as the \emph{maximum strategy}, the \emph{equidistribution strategy}, or \emph{D\"orfler's strategy}~\cite{Alberta}. More precisely, we only require that the set of marked elements $\MM_k$ has at least one element of $\Tau_k$ holding the largest local estimator. That is, there exists an element $T_k^{\max} \in \MM_k$ such that
\begin{equation}\label{E:estrategia razonable}
 \eta_k(T_k^{\max}) = \max_{T \in \Tau_k} \eta_k(T).
\end{equation}
This is what practitioners usually do in order to maximize the error reduction with a minimum computational effort.

\item[\textbf{The module \textsf{REFINE}.}] Finally, the module \textsf{REFINE} takes the mesh $\Tau_k$ and the subset $\MM_k\subset \Tau_k$ as inputs. By using the bisection rule described by Stevenson in~\cite{Stevenson-refine}, this module refines (bisects) each element in $\MM_k$ at least $n$ times (where $n \ge 1$ is fixed), in order to obtain a new conforming triangulation $\Tau_{k+1}$ of $\Omega$, which is a refinement of $\Tau_k$ and the output of this module. By definition, $\Tau_{k} \in \TT$ for all $k \in \NN$ and the family of meshes obtained by the Adaptive Algorithm is shape-regular. Finally, it is worth observing that the resulting spaces are nested, i.e., $\VV_{k} \subset \VV_{k+1}$; this fact will be used in the proof of Lemma~\ref{L:sucesivas a cero} below.
\end{description}

\begin{remark}[The adaptive sequence $\{u_k\}_{k\in\NN_0}$ is bounded]\label{R:uk acotada no lineal}
By the coercivity of $a$ given by~\eqref{E:a_coercitiva}, and using the definition~\eqref{E:kdisc-problem-nolineal inexacto} of $u_k$ we have that
$$\|\nabla u_k\|_\Omega^2\le\frac{1}{c_a}a(u_{k-1};u_k,u_k)=\frac{1}{c_a}L(u_k)\le \frac{\|L\|_{H^{-1}(\Omega)}}{c_a}\|\nabla u_k\|_\Omega,$$
for all $k\in\NN$, and then
\begin{equation*}
\|\nabla u_k\|_\Omega\le\frac{\|L\|_{H^{-1}(\Omega)}}{c_a}.
\end{equation*}
Therefore, $\{u_k\}_{k\in\NN_0}$ is bounded in $H^1_0(\Omega)$.
\end{remark}

\section{Convergence Analysis}\label{S:convergence}

In this section we prove the convergence of the sequence computed with the Adaptive Algorithm described in the previous section. We combine the ideas of the proof of convergence of adaptive algorithms for linear problems given in~\cite{MSV-convergence} and~\cite{Siebert-convergence} with new techniques needed to overcome the difficulties arisen by the nonlinear nature of the problem treated in this paper, adapting some ideas from~\cite{GMZ-08,GM-09}. 

As a first step to prove the convergence of the $\{u_k\}_{k\in\NN_0}$ to the exact solution $u$, we prove that $\|\nabla (u_k-u_{k+1})\|_\Omega\to 0$, as $k$ tends to infinity, for which we need the following auxiliary lemma.

From now on we assume that 
\begin{equation}\label{alpha-decreasing}
\alpha(x,t_1) \ge \alpha(x, t_2)  \text{ whenever } 0 \le t_1 \le t_2 \text{ and } x\in\Omega.
\end{equation}

\begin{lemma}\label{L:key property}
Let us assume that $\alpha(x,\cdot)$ is a monotone decreasing function for all $x \in \Omega$, i.e.,~\eqref{alpha-decreasing} holds. Then
$$\JJ(v)-\JJ(w)\le\frac12 \big(a(w;v,v)-a(w;w,w)\big),\qquad\forall\,v,w\in H^1_0(\Omega),$$
where $\JJ(v)= \int_0^1 a(sv;sv,v)~ds$.
\end{lemma}

\begin{proof}
Let $v,w\in H^1_0(\Omega)$. Then, a change of variables in the integral defining $\JJ(\cdot)$ yields %Usando~\eqref{E:J en terminos de beta}, la definici\'on de $\gamma$ dada en~\eqref{E:gamma} y la definici\'on de $\beta$ dada en~\eqref{E:beta} tenemos que
\begin{align*}
\JJ(v)-\JJ(w)&=\frac12\int_\Omega \left(\int_0^{|\nabla v|^2} \alpha(x,t)~dt-\int_0^{|\nabla w|^2}\alpha(x,t)~dt\right)~dx
\\
&=\frac12\int_\Omega \int_{|\nabla w|^2}^{|\nabla v|^2} \alpha(x,t)~dt~dx.
\end{align*}
Since $\alpha(x,\cdot)$ is a decreasing function for all $x \in \Omega$,
\begin{align*}
\frac12\int_\Omega \int_{|\nabla w|^2}^{|\nabla v|^2} \alpha(x,t)~dt~dx
&\le\frac12\int_\Omega \alpha(x,|\nabla w|^2)\left(|\nabla v|^2-|\nabla w|^2\right)~dx
\\
&=\frac12\big(a(w;v,v)-a(w;w,w)\big),
\end{align*}
and the assertion of the lemma follows. %implies that $\int_{s_1}^{s_2}\alpha(\cdot,t)~dt\le \alpha(\cdot,s_1)(s_2-s_1)$, for $s_1,s_2\ge 0$.
\end{proof}

%\note{$\alpha$ monotona decreciente. ?`comentario?}
%El siguiente lema es fundamental para probar la convergencia del Algoritmo~\ref{Al:nolineal inexacto}.

\begin{lemma}\label{L:sucesivas a cero}
Let $\{u_k\}_{k\in\NN_0}$ denote the sequence of discrete solutions computed with the Adaptive Algorithm. Then,
$$\lim_{k\to\infty} \|\nabla (u_k-u_{k+1})\|_\Omega=0.$$
\end{lemma}

\begin{proof}
Let $\{u_k\}_{k\in\NN_0}$ be the sequence obtained with the Adaptive Algorithm. Since $a$ is coercive (cf.~\eqref{E:a_coercitiva}), and linear and symmetric in its second and third variables, we have that
\begin{multline}\label{E:sucesivas a cero aux}
c_a\|\nabla (u_k-u_{k+1})\|_\Omega^2\le a(u_k;u_k-u_{k+1},u_k-u_{k+1}) \\
=a(u_k;u_k,u_k)-2a(u_k;u_{k+1},u_k)+a(u_k;u_{k+1},u_{k+1}).
\end{multline}
From~\eqref{E:kdisc-problem-nolineal inexacto}, since $u_k \in \VV_{k+1}$, it follows that $a(u_k;u_{k+1},u_k-u_{k+1})=L(u_k-u_{k+1})$ and thus
$$a(u_k;u_{k+1},u_k)=L(u_k-u_{k+1})+a(u_k;u_{k+1},u_{k+1}).$$
Replacing this equality in~\eqref{E:sucesivas a cero aux} and taking into account Lemma~\ref{L:key property} we obtain
\begin{equation}\label{E:sucesivas a cero aux 2}
\begin{aligned}
 c_a\|\nabla (u_k-u_{k+1})\|_\Omega^2&\le a(u_k;u_k,u_k)-2L(u_k-u_{k+1})-a(u_k;u_{k+1},u_{k+1})\\
&\le 2\JJ(u_k)-2\JJ(u_{k+1})-2L(u_k)+2L(u_{k+1})\\
&=2(\FF(u_k)-\FF(u_{k+1})),
\end{aligned}
\end{equation}
where $\FF\definedas \JJ-L$, and therefore, $\{\FF(u_k)\}_{k\in\NN_0}$ is a monotone decreasing sequence.

On the other hand, $\{\FF(u_k)\}_{k\in\NN_0}$ is bounded below since %por la Observaci\'on~\ref{R:uk acotada no lineal},
\begin{align*}
\FF(u_k)&=\int_0^1 sa(su_k;u_k,u_k)~ds-L(u_k)
\\
&\ge \frac12 c_a\|\nabla u_k\|_\Omega^2-\|L\|_{H^{-1}(\Omega)}\|\nabla u_k\|_\Omega\ge -\frac{\|L\|_{H^{-1}(\Omega)}^2}{2c_a}.
\end{align*}

Finally, from the last two assertions it follows that $\{\FF(u_k)\}_{k\in\NN_0}$ is convergent. Considering~\eqref{E:sucesivas a cero aux 2}, we conclude the proof of this lemma.
\end{proof}

We show now that the sequence obtained with the Adaptive Algorithm is convergent, and more precisely, that it converges to a function in the limiting space $\VV_\infty\definedas \overline{\cup \VV_k}^{H^1_0(\Omega)}$. Note that $\VV_\infty$ is a Hilbert space with the inner product inherited from $H^1_0(\Omega)$.

\begin{theorem}[The adaptive sequence is convergent]\label{T:convergencia a un limite}
Let $\{u_k\}_{k\in\NN_0}$ be the sequence obtained with the Adaptive Algorithm. Let $u_\infty\in\VV_\infty$ be the only solution to 
\begin{equation}\label{E:cont_prob_vinfty inexacto}
u_\infty \in \VV_\infty : \quad a(u_\infty;u_\infty,v)=L(v),\qquad \forall~v\in  \VV_\infty.
\end{equation}
Then
$$u_k \longrightarrow u_\infty \quad \text{in } H^1_0(\Omega).$$
\end{theorem}

\begin{remark} Notice that~\eqref{E:cont_prob_vinfty inexacto} always has a solution because~\eqref{E:A Lipschitz} and~\eqref{E:A fuertemente monotono} hold on $\VV_\infty$, which is itself a Hilbert space.
\end{remark}

\begin{proof}
Let $\{u_k\}_{k\in\NN_0}$ be the sequence obtained with the Adaptive Algorithm. Let $u_\infty\in\VV_\infty$ denote the solution of~\eqref{E:cont_prob_vinfty inexacto} and $\PP_{k+1}:H^1_0(\Omega)\to\VV_{k+1}$ be the orthogonal projection onto $\VV_{k+1}$. Since $A$ is strongly monotone (cf.~\eqref{E:A fuertemente monotono}), using~\eqref{E:cont_prob_vinfty inexacto} and~\eqref{E:kdisc-problem-nolineal inexacto} we have that
\begin{align*}
c_A\|\nabla (u_k-u_\infty)\|_\Omega^2&\le \langle Au_k-Au_\infty,u_k-u_\infty\rangle\\
&=\langle Au_k,u_k-u_\infty\rangle-L(u_k-u_\infty)\\
&=\langle Au_k,u_k-\PP_{k+1}u_\infty\rangle+\langle Au_k,\PP_{k+1}u_\infty-u_\infty\rangle\\
&\quad-L(u_k-\PP_{k+1} u_\infty)-L(\PP_{k+1} u_\infty-u_\infty)\\
&=a(u_k;u_k-u_{k+1},u_k-\PP_{k+1}u_\infty)+\langle Au_k,\PP_{k+1}u_\infty-u_\infty\rangle
\\
&\quad -L(\PP_{k+1} u_\infty-u_\infty) \\
&\le C_a\|\nabla (u_k-u_{k+1})\|_\Omega(\|\nabla u_k\|_\Omega+\|\nabla u_\infty\|_\Omega) \\
&\quad+(C_a\|\nabla u_k\|_\Omega+\|L\|_{H^{-1}(\Omega)})\|\nabla (\PP_{k+1}u_\infty-u_\infty)\|_\Omega,
\end{align*}
for all $k\in\NN_0$, where in the last inequality we have used~\eqref{E:a_acotada}. From Remark~\ref{R:uk acotada no lineal} it follows that $\{u_k\}_{k\in\NN_0}$ is bounded in $H^1_0(\Omega)$, and using Lemma~\ref{L:sucesivas a cero}, together with the fact that the spaces $\{\VV_k\}_{k\in\NN_0}$ are nested and $\cup_{k\in\NN_0} \VV_k$ is dense in $\VV_\infty$, we conclude that $u_k\to u_\infty$ in $H^1_0(\Omega)$. 
\end{proof}

In order to show that the limiting function $u_\infty$ is, in fact, the solution of the problem~\eqref{E:cont_prob_nolineal} and thereby conclude that the adaptive sequence converges to the solution of this problem, we establish first two auxiliary results (see Lemma~\ref{L:marcados tienden a cero} and Theorem~\ref{T:convergencia debil del residuo}). We need the following

\begin{definition}\label{D:splitting}
Given any sequence of meshes $\{\Tau_k\}_{k\in\NN_0}\subset\TT$, with $\Tau_{k+1}$ a refinement of $\Tau_k$, for each $k \in \NN_0$, we define
$$\Tau_k^+\definedas \{T\in\Tau_k \mid T\in\Tau_m,\quad \forall\,m\ge k\},\qquad\qquad\Tau_k^0\definedas \Tau_k \setminus \Tau_k^+,$$
and
$$\qquad\qquad\Omega_k^+\definedas \bigcup\limits_{T\in\Tau_k^+} \omega_k(T),\qquad\qquad\qquad\qquad\Omega_k^0\definedas \bigcup\limits_{T\in\Tau_k^0} \omega_k(T).$$
In words, $\Tau_k^+$ is the subset of the elements of $\Tau_k$ which are never refined in the adaptive process, and $\Tau_k^0$ consists of the elements which are eventually refined.
\end{definition}

%Puesto que $\Tau_{k+1}$ es siempre un refinamiento de $\Tau_k$, para casi todo $x\in\Omega$ se tiene que $\{h_{\Tau_k}(x)\}_{k\in\NN_0}$ es mon\'otona decreciente y acotada inferiormente por $0$. As\'i, $$h_\infty(x)\definedas\lim\limits_{k\rightarrow\infty} h_{\Tau_k}(x)$$ est\'a bien definida en casi todo $x\in\Omega$ y define una funci\'on en $L^\infty(\Omega)$. M\'as a\'un, la convergencia es uniforme como lo afirma el siguiente lema, cuya demostraci\'on puede encontrarse en el Lema 4.3 y el Corolario 4.1 de~\cite{MSV-convergence}. 

% \begin{lemma}\label{L:hktendstozero}
% La sucesi\'on de funciones de tama\~no de la malla $\{h_{\Tau_k}\}_{k\in\NN_0}$ converge a $h_\infty$ uniformemente, es decir,
% $$\lim_{k\rightarrow\infty} \|h_{\Tau_k}-h_\infty\|_{L^\infty(\Omega)}=0,$$
% y si $\chi_{\Omega_k^0}$ denota la funci\'on caracter\'istica de $\Omega_k^0$ entonces
It can be proved~\cite{MSV-convergence} that if $\chi_{\Omega_k^0}$ denotes the characteristic function of $\Omega_k^0$, then
\begin{equation}\label{E:hktendstozero}
 \lim_{k\rightarrow\infty} \|h_k \chi_{\Omega_k^0}\|_{L^\infty(\Omega)}=0,
\end{equation}
where $h_k\in L^\infty(\Omega)$ denotes the piecewise constant mesh-size function satisfying ${h_k}_{|_T}\definedas\h_T$, for all $T\in\Tau_k$.
%\end{lemma}

Since the error estimators are stable (cf. Proposition~\ref{P:estabilidad kakanov}), using the convergence proved in the last theorem and~\eqref{E:hktendstozero} we can establish the following

\begin{lemma}[Estimator on marked elements]\label{L:marcados tienden a cero} 
Let $\left\{\{\eta_k(T)\}_{T\in\Tau_k}\right\}_{k\in\NN_0}$ be the sequence of local error estimators computed with the Adaptive Algorithm, and let $\{\MM_k\}_{k\in\NN_0}$ be the sequence of subsets of marked elements over each mesh. Then, $$\lim_{k\to\infty} \max_{T\in\MM_k} \eta_k(T)=0.$$
\end{lemma}

\begin{proof}
Let $\left\{\{\eta_k(T)\}_{T\in\Tau_k}\right\}_{k\in\NN_0}$ and $\{\MM_k\}_{k\in\NN_0}$ be as in the assumptions. For each $k\in\NN_0$, we select $T_k\in\MM_k$ such that $\eta_k(T_k)=\max_{T\in\MM_k} \eta_k(T)$. Using Proposition~\ref{P:estabilidad kakanov} we have that
\begin{equation}\label{E:marcados tienden a cero aux1}
\eta_k(T_k)\lesssim \|\nabla u_k\|_{\omega_k(T_k)}+\|f\|_{T_k}\lesssim \|\nabla u_k-\nabla u_\infty\|_{\Omega}+\|\nabla u_\infty\|_{\omega_k(T_k)}+\|f\|_{T_k},
\end{equation}
where $u_\infty\in\VV_\infty$ is the solution of~\eqref{E:cont_prob_vinfty inexacto}. On the one hand, the first term in the right hand side of~\eqref{E:marcados tienden a cero aux1} tends to zero due to Theorem~\ref{T:convergencia a un limite}. On the other hand, since $T_k\in\MM_k\subset \Tau_k^0$, from~\eqref{E:hktendstozero} it follows that
\begin{equation*}
|T_k|\le|\omega_k(T_k)|\lesssim \h_{T_k}^d\le \|h_k\chi_{\Omega_k^0}\|_{L^\infty(\Omega)}^d\rightarrow 0, \quad\text{as } k\to\infty,
\end{equation*}
and the last two terms in the right hand side of~\eqref{E:marcados tienden a cero aux1} also tend to zero.
\end{proof}

Using the upper bound~\eqref{E:local-residual-upperbound}, the stability of the estimators (Proposition~\ref{P:estabilidad kakanov}), the facts that $\{u_k\}_{k\in\NN_0}$ is bounded (cf. Remark~\ref{R:uk acotada no lineal}) and the marking strategy is reasonable (cf.~\eqref{E:estrategia razonable}), and Lemma~\ref{L:marcados tienden a cero}, we now prove the following important result.

\begin{theorem}[Weak convergence of the residual]\label{T:convergencia debil del residuo}
%Supongamos que se cumple la Hip\'otesis~\ref{A:aproximacion local no lineales} de aproximaci\'on local. 
If $\{u_k\}_{k\in\NN_0}$ denotes the sequence of discrete solutions computed with the Adaptive Algorithm, then
$$
\lim_{k\to\infty} \langle\RRR(u_k),v\rangle=0,
\quad\text{for all $v\in H^1_0(\Omega)$}.
$$
\end{theorem}

\begin{proof}
We prove first the result for $v\in H^2(\Omega)\cap H^1_0(\Omega)$, and then extend it to $H^1_0(\Omega)$ by density. Let $p\in\NN$ and $k>p$. By Definition~\ref{D:splitting} we have that $\Tau_p^+\subset\Tau_k^+\subset\Tau_k$.  Let $v_k\in \VV_k$ be the Lagrange's interpolant of $v$. Since $\langle \RRR(u_k),v_k\rangle=0$, using~\eqref{E:local-residual-upperbound}, and Cauchy-Schwartz's inequality we have that
\begin{align*}
|\langle\RRR(u_k),v\rangle|&=|\langle\RRR(u_k),v-v_k\rangle|
\lesssim \sum_{T\in\Tau_k} \eta_k(T)\|\nabla (v-v_k)\|_{\omega_k(T)}\\
&= \sum_{T\in\Tau_p^+} \eta_k(T)\|\nabla(v-v_k)\|_{\omega_k(T)}+\sum_{T\in\Tau_k\setminus\Tau_p^+} \eta_k(T)\|\nabla(v-v_k)\|_{\omega_k(T)}\\
&\lesssim \eta_k(\Tau_p^+)\|\nabla(v-v_k)\|_{\Omega_p^+}+\eta_k(\Tau_k\setminus\Tau_p^+)\|\nabla(v-v_k)\|_{\Omega_p^0},
\end{align*}
where, for any $\Xi\subset\Tau_k$ we hereafter denote $\left(\sum_{T\in\Xi}\eta_k^2(T)\right)^{\frac12}$ by $\eta_k(\Xi)$.
Taking into account Proposition~\ref{P:estabilidad kakanov} and the boundedness of the discrete solutions (cf. Remark~\ref{R:uk acotada no lineal}) we have that $\eta_k(\Tau_k\setminus\Tau_p^+)\le\eta_k(\Tau_k)\lesssim 1$, and therefore,
\[
|\langle\RRR(u_k),v\rangle|\lesssim \left(\eta_k(\Tau_p^+)+\|h_p \chi_{\Omega_p^0}\|_{L^\infty(\Omega)}\right)|v|_{H^2(\Omega)},
\]
due to interpolation estimates.

In order to prove that $\langle\RRR(u_k),v\rangle \to 0$ as $k\to\infty$ we now let $\varepsilon>0$ be arbitrary. Due to~\eqref{E:hktendstozero}, there exists $p\in\NN$ such that
$$ \|h_p \chi_{\Omega_p^0}\|_{L^\infty(\Omega)}<\varepsilon.$$
On the other hand, since $\Tau_p^+\subset\Tau_k^+\subset\Tau_k$ and the marking strategy is reasonable (cf.~\eqref{E:estrategia razonable}), 
$$\eta_k(\Tau_p^+)\le (\# \Tau_p^+)^{1/2} \max_{T\in\Tau_p^+}\eta_k(T)\le (\# \Tau_p^+)^{1/2} \max_{T\in\MM_k}\eta_k(T).$$
Now, by Lemma~\ref{L:marcados tienden a cero}, we can select $K>p$ such that
$\eta_k(\Tau_p^+)<\varepsilon$,
for all $k>K$.

Summarizing, we have proved that
\[
\lim_{k\to\infty} \langle\RRR(u_k),v\rangle=0,\qquad \text{for all } v\in H^2(\Omega)\cap H^1_0(\Omega).
\]
Finally, since $H^{2}(\Omega)\cap H^1_0(\Omega)$ is dense in $H^1_0(\Omega)$, this limit is also zero for all $v\in H^1_0(\Omega)$.
\end{proof}

As a consequence of Theorem~\ref{T:convergencia debil del residuo} we now prove that $u_\infty$ is the solution of problem~\eqref{E:cont_prob_nolineal}.

\begin{theorem}[The limiting function is the solution]\label{T:u_infty es solucion inexacto}
If $u_\infty$ denotes the solution of~\eqref{E:cont_prob_vinfty inexacto}, then $u_\infty$ is the solution of problem~\eqref{E:cont_prob_nolineal}, i.e.,
\begin{equation*}
a(u_\infty;u_\infty,v)=L(v),\qquad\forall\,v\in H^1_0(\Omega).
\end{equation*}
\end{theorem}

\begin{proof}
Let $u_\infty$ be the solution of~\eqref{E:cont_prob_vinfty inexacto}. If $v\in H^1_0(\Omega)$, and $\{u_k\}_{k\in\NN_0}$ denotes the sequence of discrete solutions computed with the Adaptive Algorithm, then
\begin{align*}
|a(u_\infty;u_\infty,v)-L(v)|&=|a(u_\infty;u_\infty,v)-L(v)-a(u_k;u_{k+1},v)+a(u_k;u_{k+1},v)|\\
&\le |a(u_\infty;u_\infty,v)-a(u_k;u_{k+1},v)|+|\langle \RRR(u_{k+1}),v\rangle|\\
&\le|a(u_\infty;u_\infty,v)-a(u_k;u_k,v)|
\\
&\quad +|a(u_k;u_k-u_{k+1},v)|+|\langle \RRR(u_{k+1}),v\rangle|\\
&\le \|Au_\infty-Au_k\|_{H^{-1}(\Omega)}\|\nabla v\|_\Omega
\\
&\quad+C_a\|\nabla (u_k-u_{k+1})\|_\Omega\|\nabla v\|_\Omega+|\langle \RRR(u_k),v\rangle|\\
&\le  C_A\|\nabla (u_\infty-u_k)\|_\Omega\|\nabla v\|_\Omega
\\
&\quad+C_a\|\nabla (u_k-u_{k+1})\|_\Omega\|\nabla v\|_\Omega+|\langle \RRR(u_k),v\rangle|,
\end{align*}
where we have used that $A$ is Lipschitz (cf.~\eqref{E:A Lipschitz}) and $a$ is bounded (cf.~\eqref{E:a_acotada}). Using  Theorem~\ref{T:convergencia a un limite}, Lemma~\ref{L:sucesivas a cero} and Theorem~\ref{T:convergencia debil del residuo} it follows that the right-hand side in the last inequality tends to zero as $k$ tends to infinity.% Since $v\in H^1_0(\Omega)$ is arbitrary, $u_\infty$ is the solution to problem~\eqref{E:cont_prob_nolineal}.
\end{proof}

As an immediate consequence of Theorems~\ref{T:convergencia a un limite} and~\ref{T:u_infty es solucion inexacto} we finally obtain the main result of this article.

\begin{theorem}[Main result]\label{T:main}
Let $\{u_k\}_{k\in\NN_0}$ denote the sequence of discrete solutions computed with the Adaptive Algorithm. If $\alpha(\cdot,\cdot)$ satisfies assumptions~\eqref{E:betasegunda acotada} and~\eqref{alpha-decreasing}, then $\{u_k\}_{k\in\NN_0}$ converges to the solution $u$ of  problem~\eqref{E:cont_prob_nolineal}.
\end{theorem}

We conclude this section with a couple of remarks.

\begin{remark}
The problem given by~\eqref{E:ley de conservacion} is a particular case of the more general problem:
\begin{equation*}%\label{E:ley de conservacion general}
\left\{
\begin{aligned}
-\nabla\cdot \big[\alpha(\,\cdot\,,|\nabla u|_\AAA^2)\AAA\nabla u\big]&= f\qquad & & \text{in}\,\Omega\\
u&= 0\qquad & &\text{on}\,\partial \Omega,
\end{aligned}
\right.
\end{equation*}
where $\alpha:\Omega\times \RR_+\to \RR_+$ and $f\in L^2(\Omega)$ satisfy the properties assumed in the previous sections, and $\AAA:\Omega\to \RR^{d\times d}$ is %such that $\AAA(x)$ is 
symmetric for all $x\in\Omega$, and uniformly positive definite, i.e., there exist constants $\underline{a},\overline{a}>0$ such that
\begin{equation*}%\label{E:matriz A no lineales}
\underline{a}|\xi|^2\leq \AAA(x)\xi\cdot \xi\leq \overline{a}|\xi|^2,\qquad \forall~x\in\Omega,\,\xi\in\RR^d.
\end{equation*}
 If $\AAA$ is piecewise Lipschitz over an initial conforming mesh $\Tau_0$ of $\Omega$, i.e., there exists $C_\AAA>0$ such that
\begin{equation*}%\label{E:A Lipschitz no lineales}
\|\AAA(x)-\AAA(y)\|_2\le C_\AAA|x-y|,\qquad\forall\,x,y\in T,\quad\forall\,T\in\Tau_0,
\end{equation*}
then the convergence results previously presented also hold for this problem.
\end{remark}

\begin{remark}\label{R:polynomial degree}
We have assumed the use of \emph{linear} finite elements for the discretization (see~\eqref{E:V_Tau}). It is important to notice that the only place where we used this is in Proposition~\ref{P:estabilidad kakanov}.  The rest of the steps of the proof hold regardless of the degree of the finite element space. The use of linear finite elements is customary in nonlinear problems, because they greatly simplify the analysis. The numerical experiments of the next section show a competitive performance of the adaptive method for any tested polynomial degree (up to four).
\end{remark}

\section{Numerical Experiments}\label{S:numerical experiments}

We conclude this article reporting on the behavior of the adaptive algorithm for some particular nonlinear problems. 
In the first subsection we study the convergence rate in terms of degrees of freedom for an exact solution and different functions $\alpha(\cdot,\cdot)$. 
In the second subsection we show the performance of the algorithm when approximating an unknown solution of a prescribed curvature equation. 

\subsection{Exact solution}

Let us consider the problem
\begin{equation}\label{E:problema general}
\left\{
\begin{aligned}
-\nabla\cdot \big[\alpha(|\nabla u|^2)\nabla u\big]&= f& &\qquad \text{in}\,\Omega\\
u&= g& &\qquad\text{on}\,\partial \Omega,
\end{aligned}
\right.
\end{equation}
where $\Omega\subset\RR^2$ is the $L$-shaped domain given in Figure~\ref{f:solucion}. In the following examples, in order to study experimentally the behavior of the Adaptive Algorithm, we consider~\eqref{E:problema general} with different choices of the function $\alpha$, defining $f$ and $g$ in each case so that the solution of the problem is the function $u$ depicted in Figure~\ref{f:solucion}, given in polar coordinates by
\begin{equation}\label{E:u polar}
u(r,\varphi)=r^{\frac{2}{3}}\sin\left(\frac{2}{3}\varphi\right).
\end{equation}

\begin{figure}[h!tbp]
\begin{center}
%\psfrag{u_h}{\small $u(r,\varphi)=r^{\frac{2}{3}}\sen(\frac{2}{3}\varphi)$}
\psfrag{W}{\small$\Omega$}
\psfrag{1}{\tiny $1$}
\psfrag{-1}{\tiny $-1$}
\hfil
\includegraphics[width=.30\textwidth]{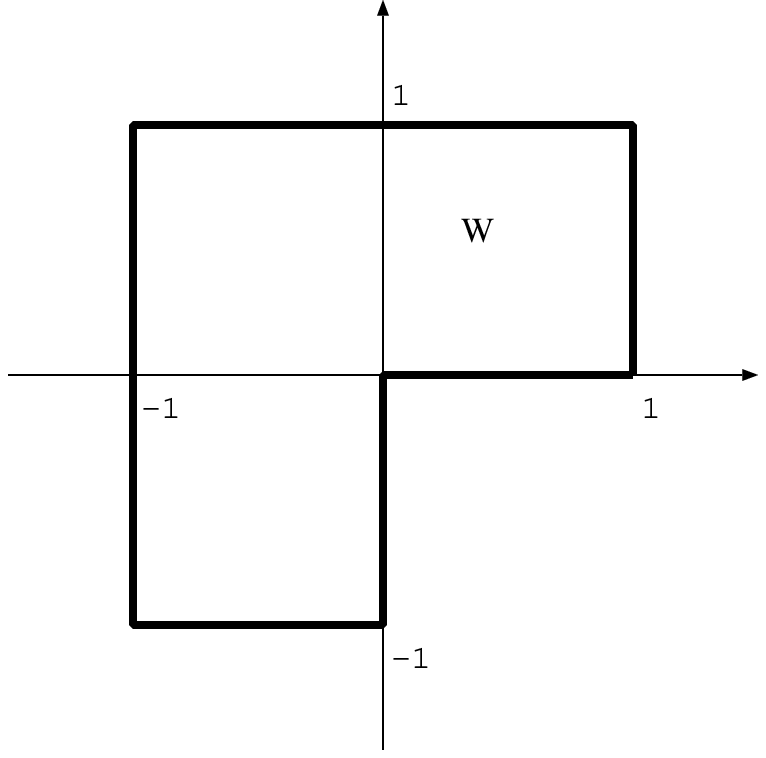}
%\vspace{-0.1cm}
\hfil
\includegraphics[width=.34\textwidth]{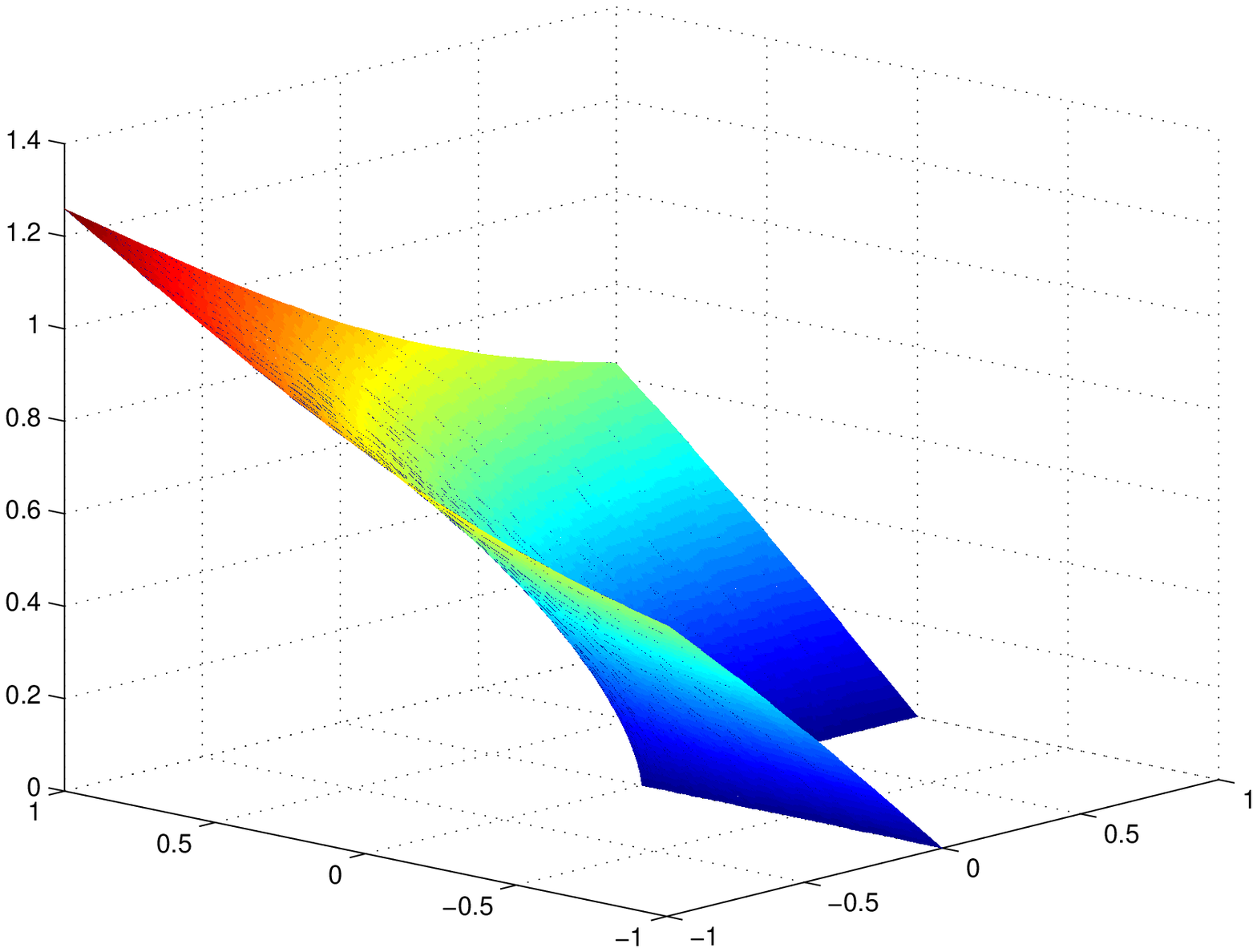}
%\hfil
%\includegraphics[width=.34\textwidth]{figuras/funcionsingular2}
\caption{The domain $\Omega$ where the problem~\eqref{E:problema general} is posed and the function $u$ which is the solution of the problem in each example.}\label{f:solucion}
\end{center}
\end{figure}

We consider the Adaptive Algorithm using different marking strategies, namely, \emph{global refinement}, \emph{maximum strategy} with $\theta=0.7$ and \emph{D\"orfler's strategy} with $\theta=0.5$ (see~\cite{Alberta}).

We implemented the Adaptive Algorithm using the finite element toolbox ALBERTA~\cite{Alberta}. We iterated the algorithm until the global error estimator was below $10^{-6}$ or the number of degrees of freedom exceeded $5\times 10^5$. We tested the limits of our theory by trying with some functions $\alpha$ which did not satisfy all the assumptions of the theoretical results above.

%\clearpage

% Ejemplo 1
\begin{example}[Optimal rate of convergence when $\alpha$ satisfies the hypotheses]\label{Ex:todo_ok}
As a first example, in order to study experimentally the rate of convergence of the Adaptive Algorithm, we consider
$$\alpha(t)=\frac{1}{1+t}+\frac{1}{2},\qquad t>0,$$
which satisfies the hypotheses to guarantee the convergence (see Figure~\ref{f:todo_ok}), i.e., $\alpha$ is a $\CCC^1$-function, and there exist positive constant $c_a$ and $C_a$ such that
\begin{equation}\label{E:Hip beta segunda}
c_a\le\alpha(t^2)+2t^2\alpha'(t^2)\le C_a,\qquad\forall\,t>0,
\end{equation}
and
\begin{equation}\label{E:Hip alpha decreciente}
\text{$\alpha$ is monotone decreasing, i.e., $\alpha'(t)\le 0$ for all $t>0$.}
\end{equation}

\begin{figure}[h!tbp]
\begin{center}
\psfrag{alpha}{\hspace{-3pt}\tiny $\alpha(t)$}
\psfrag{betasegundabetasegundabetasegunda}{\tiny$\alpha(t^2)+2t^2\alpha'(t^2)$}
\psfrag{t}{\tiny $t$}
\includegraphics[width=.40\textwidth]{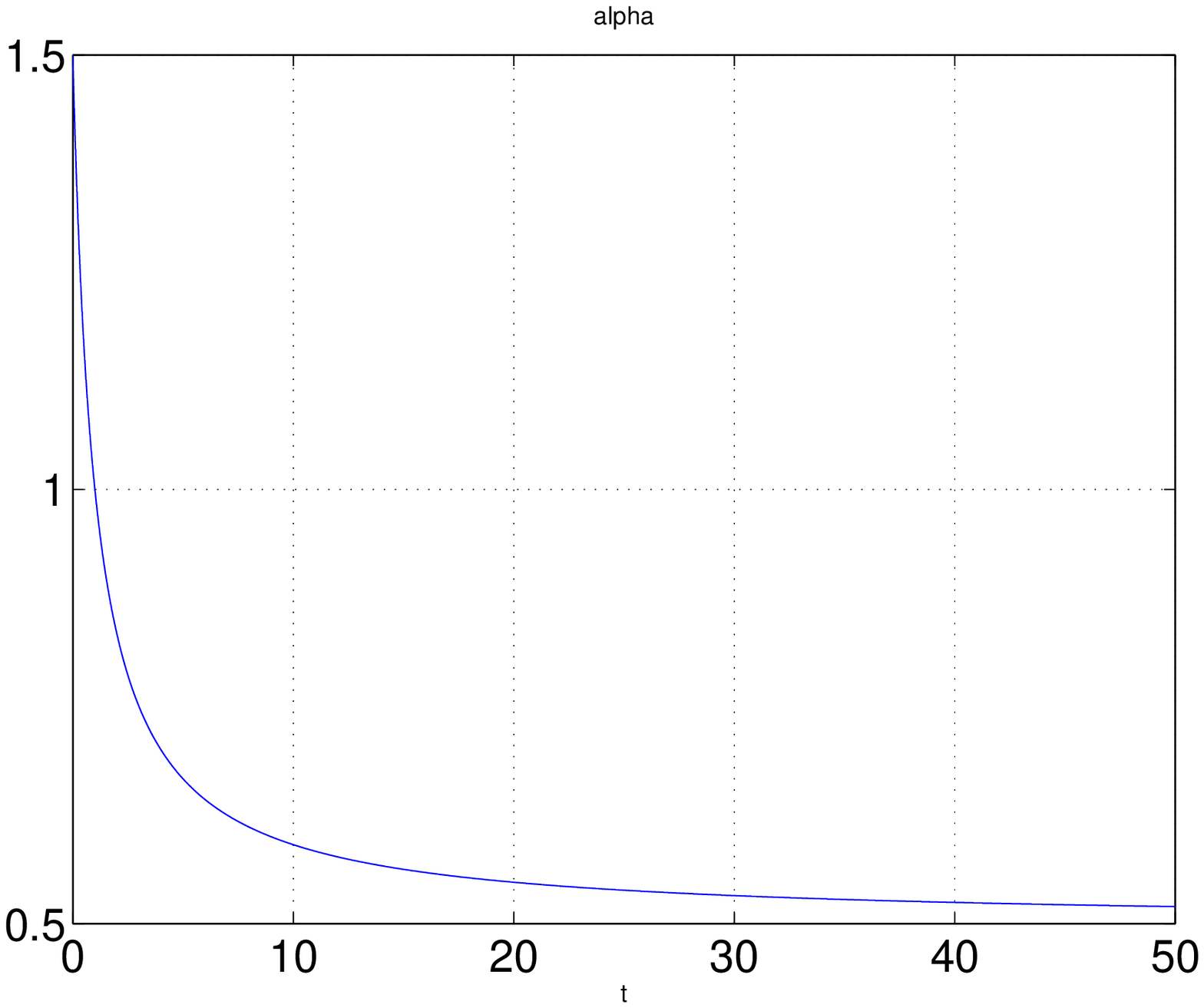}
\hfil
\includegraphics[width=.40\textwidth]{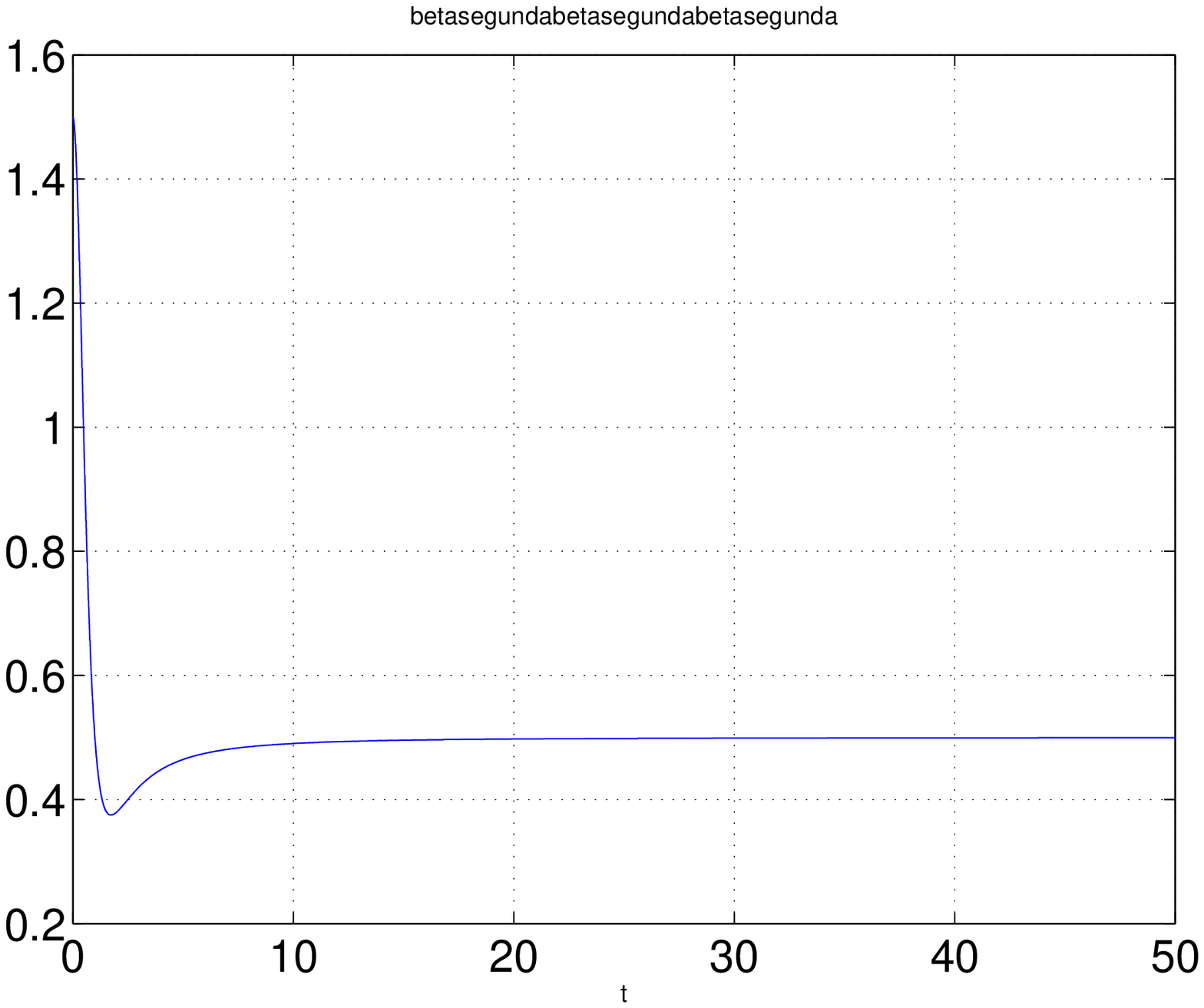}
\caption{The function $\alpha(t)=\frac{1}{1+t}+\frac{1}{2}$, of Example~\ref{Ex:todo_ok} satisfies the properties we require to guarantee the convergence.}\label{f:todo_ok}
\end{center}
\end{figure}

In Figure~\ref{f:todo_ok_error} we plot the $H^1(\Omega)$-error versus the number of degrees of freedom (DOFs), for finite elements of degree $\ell=1,2,3,4$. In this case, the rate of the convergence is optimal for adaptive strategies, that is, $\| u - u_k \|_{H^1(\Omega)} =O(\text{DOFs}_k^{-\ell/2})$. For global refinement, the observed order of convergence is $\text{DOFs}_k^{-1/3}$ for all tested polynomial degrees, due to the fact that the solution $u$ belongs to $H^{1+\delta}(\Omega)$, for all $0<\delta<\frac{2}{3}$, and does not belong to $H^{1+2/3}(\Omega)$. 

Note that, although the theory only guarantees the plain convergence for linear elements (cf. Theorem~\ref{T:main}), the numerical results suggest that the method works for any polynomial degree (see Remark~\ref{R:polynomial degree}), and the convergence rate is optimal.

\begin{figure}[h!tbp]
\begin{center}
\psfrag{Refinamiento Global}{\tiny Global refinement}
\psfrag{Estrategia del Maximo}{\tiny Maximum Strategy}
\psfrag{Estrategia de Dorfler}{\tiny D\"orfler Strategy}
\psfrag{Pendiente Optima}{\tiny Optimal slope}
\psfrag{Error}{\hspace{-10pt}\tiny $H^1$-error}
\psfrag{DOFs}{\tiny DOFs}
\psfrag{Aproximacion con polinomios de grado 1}{ \tiny Polynomial degree $\ell=1$}
\psfrag{Aproximacion con polinomios de grado 2}{ \tiny Polynomial degree $\ell=2$}
\psfrag{Aproximacion con polinomios de grado 3}{ \tiny Polynomial degree $\ell=3$}
\psfrag{Aproximacion con polinomios de grado 4}{ \tiny Polynomial degree $\ell=4$}
\includegraphics[width=.40\textwidth]{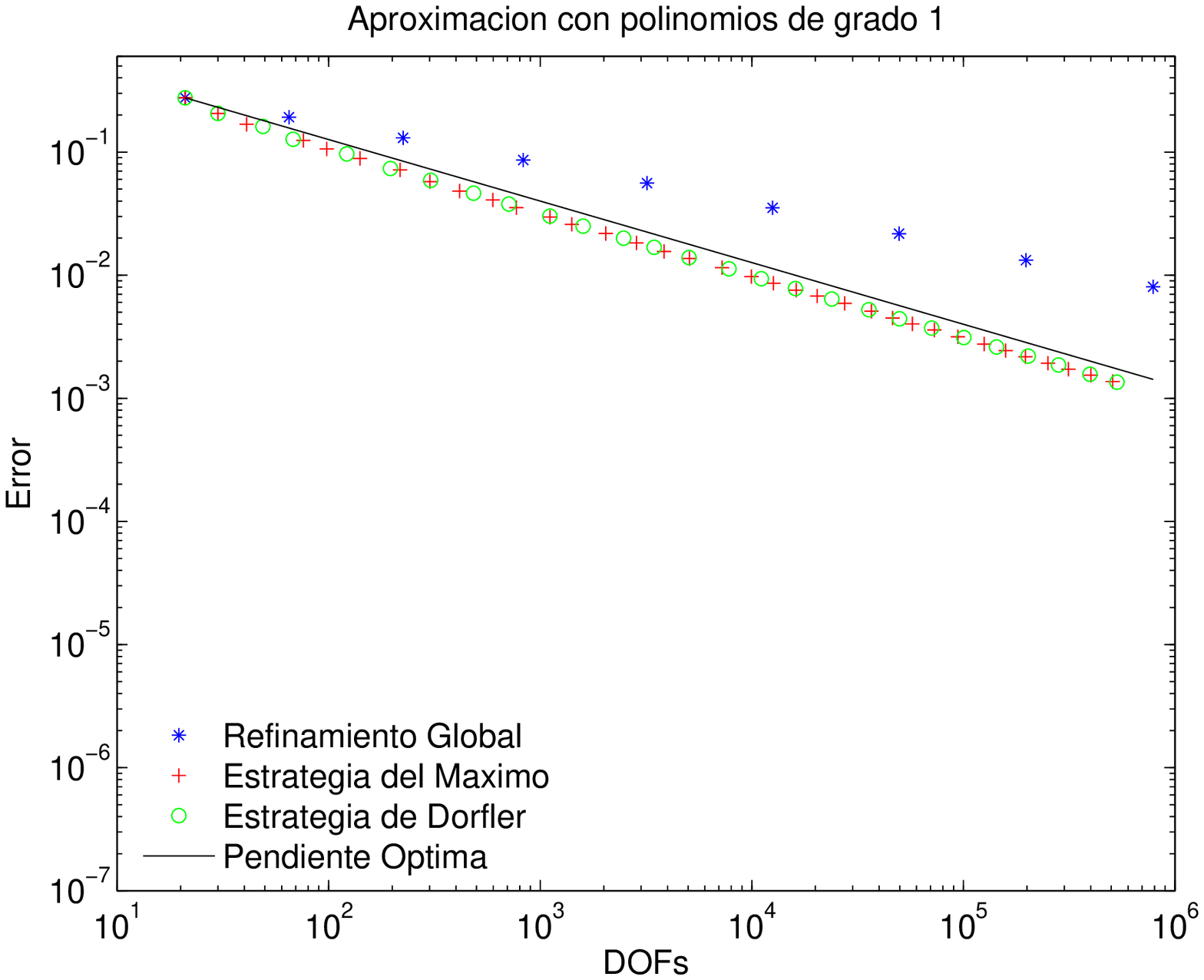}\hfil
\includegraphics[width=.40\textwidth]{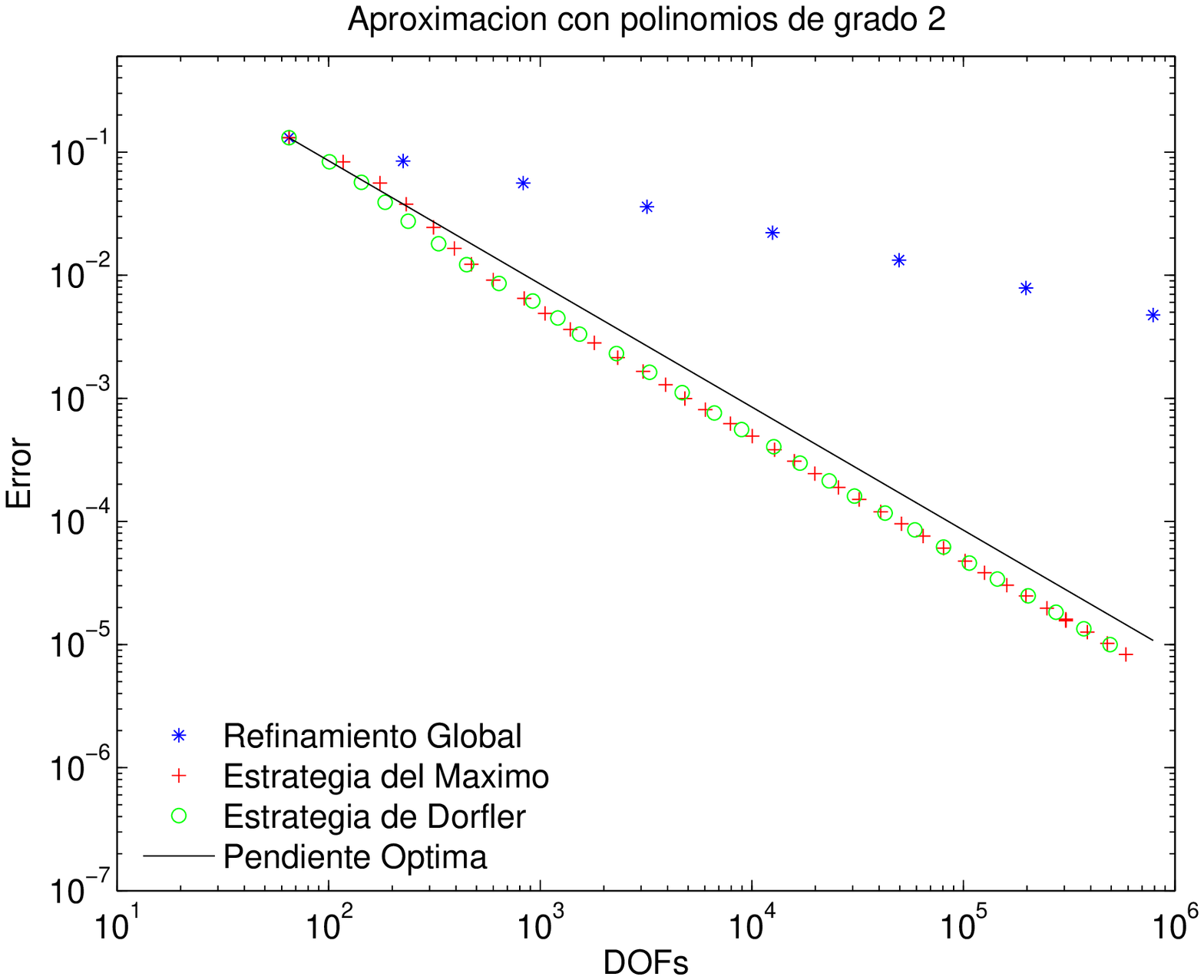}

\bigskip
\includegraphics[width=.40\textwidth]{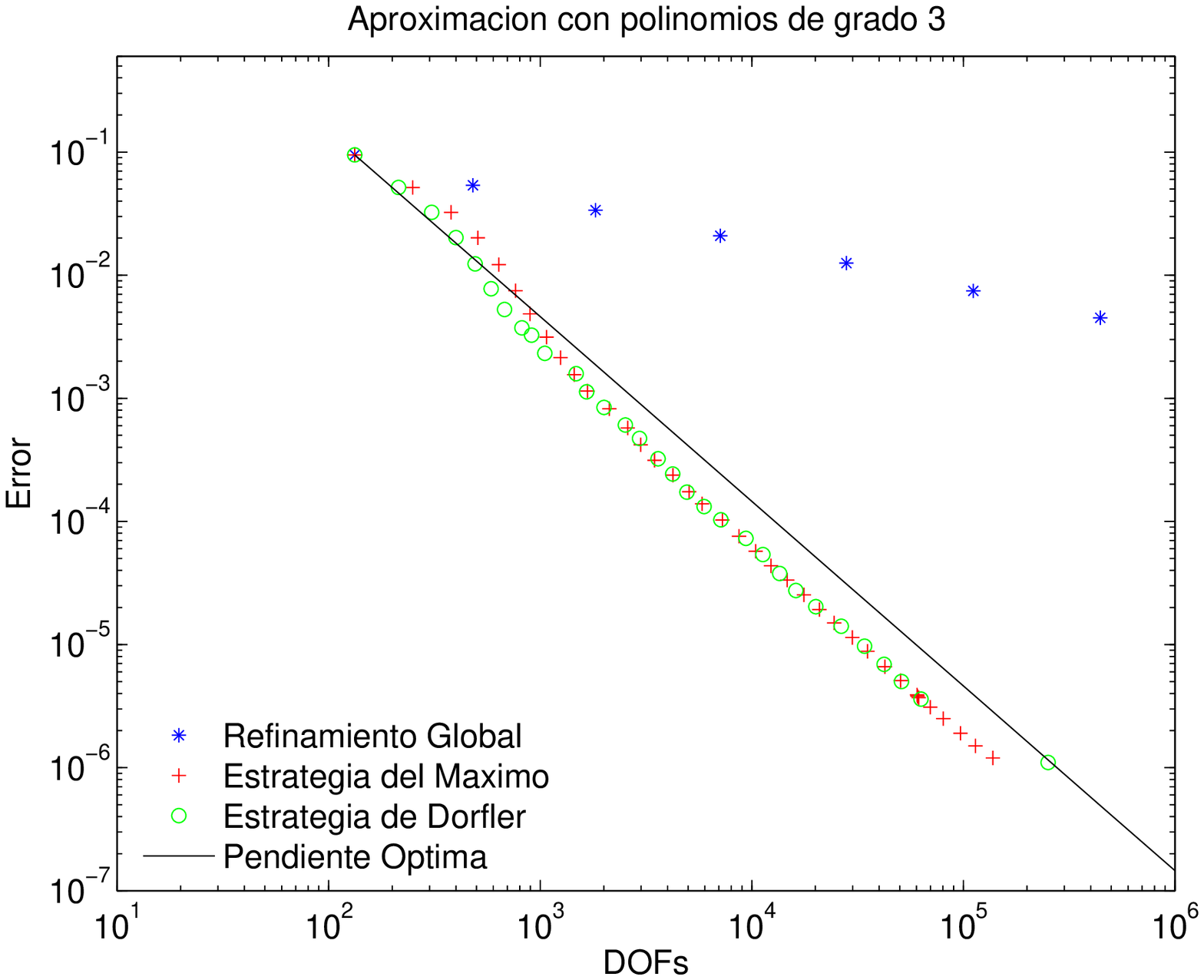}\hfil
\includegraphics[width=.40\textwidth]{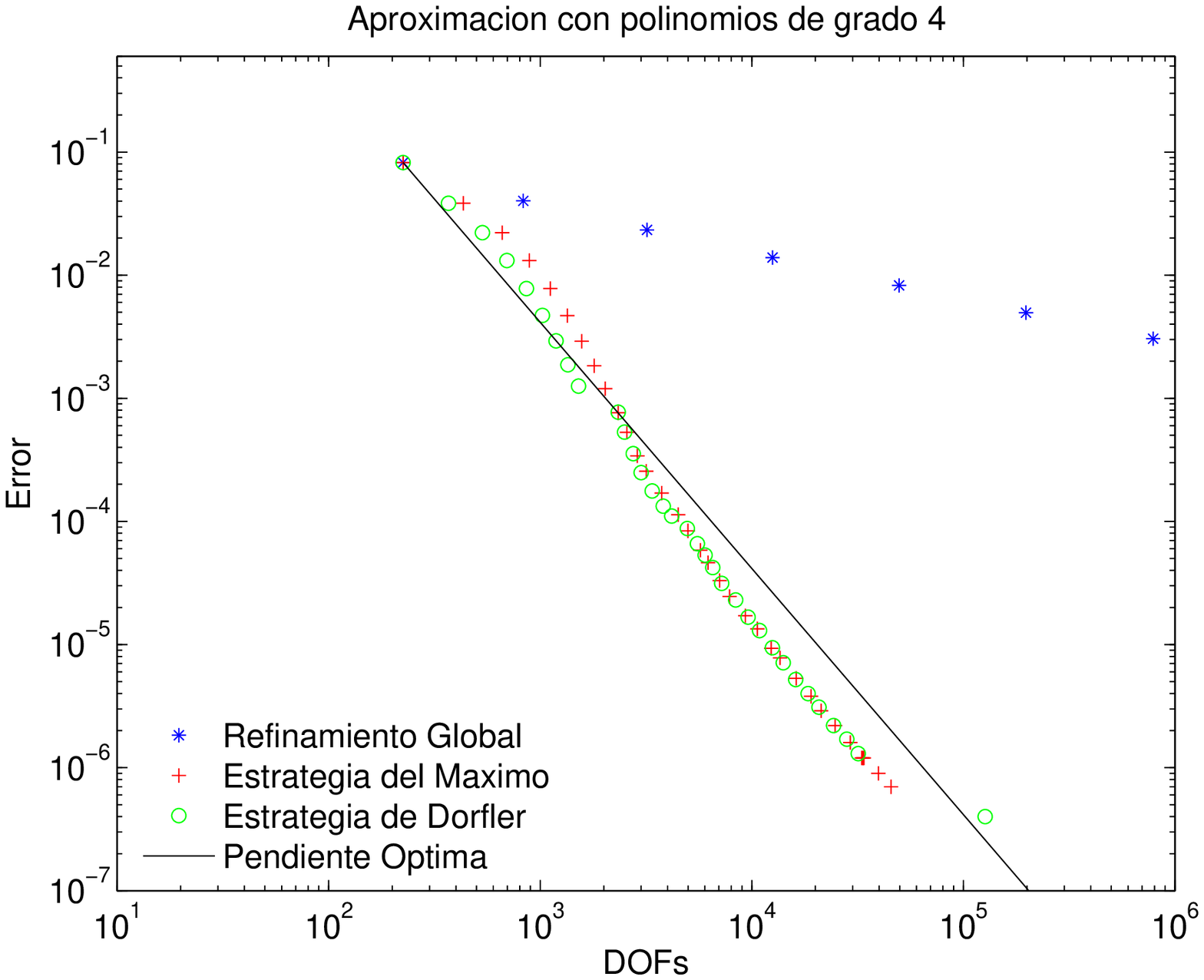}
\caption{
Error versus DOFs for Example~\ref{Ex:todo_ok}.
%Gr\'aficos del error en norma $H^1(\Omega)$ en funci\'on de los grados de libertad (DOFs) utilizados, para la aproximaci\'on con distintos grados polinomiales, considerando el Ejemplo~\ref{Ex:todo_ok}.
We present the $H^1(\Omega)$-error between the exact solution and discrete solutions, versus the number of degrees of freedom (DOFs) used to represent each of them. We note that the convergence rate is optimal for the considered adaptive strategies, but not for global refinement, due to the fact that the solution $u$ is not sufficiently smooth. In this case, $\alpha(t)=\frac{1}{1+t}+\frac{1}{2}$ satisfies all the properties established to guarantee the convergence with linear finite elements. The numerical experiments suggest that the method converges with optimal rate for any polynomial degree.}\label{f:todo_ok_error}
\end{center}
\end{figure}
\end{example}

%\clearpage

% Ejemplo 2
\begin{example}[About the hypothesis~\eqref{E:Hip beta segunda}]\label{Ex:betasegundabad}
We consider the function
$$\alpha(t)=\frac{1}{1+t}+\frac{1}{10},\qquad t>0,$$
which is monotone decreasing, i.e., satisfies~\eqref{E:Hip alpha decreciente}, but not~\eqref{E:Hip beta segunda}, as it is shown in Figure~\ref{f:betasegundabad}. Since~\eqref{E:Hip beta segunda} guarantees the well-posedness of problem~\eqref{E:problema general} (uniqueness and stability), we could be facing an example with multiple solutions. %in this case the problem could have other solutions besides $u$ given by~\eqref{E:u polar}.

\begin{figure}[h!tbp]
\begin{center}
\psfrag{alpha}{\hspace{-3pt}\tiny $\alpha(t)$}
\psfrag{betasegundabetasegundabetasegunda}{\tiny$\alpha(t^2)+2t^2\alpha'(t^2)$}
\psfrag{t}{\tiny $t$}
\includegraphics[width=.40\textwidth]{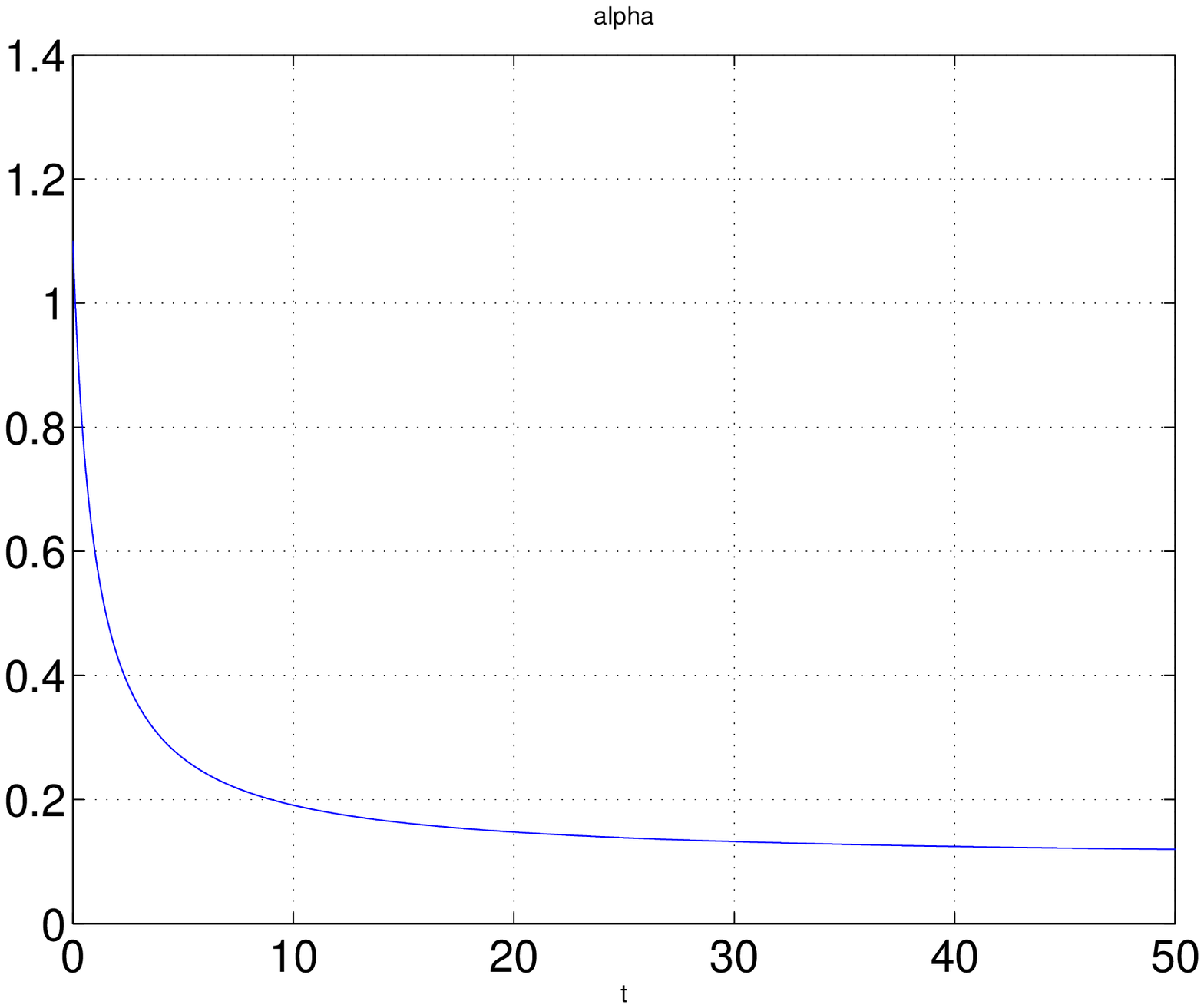}
\hfil
\includegraphics[width=.40\textwidth]{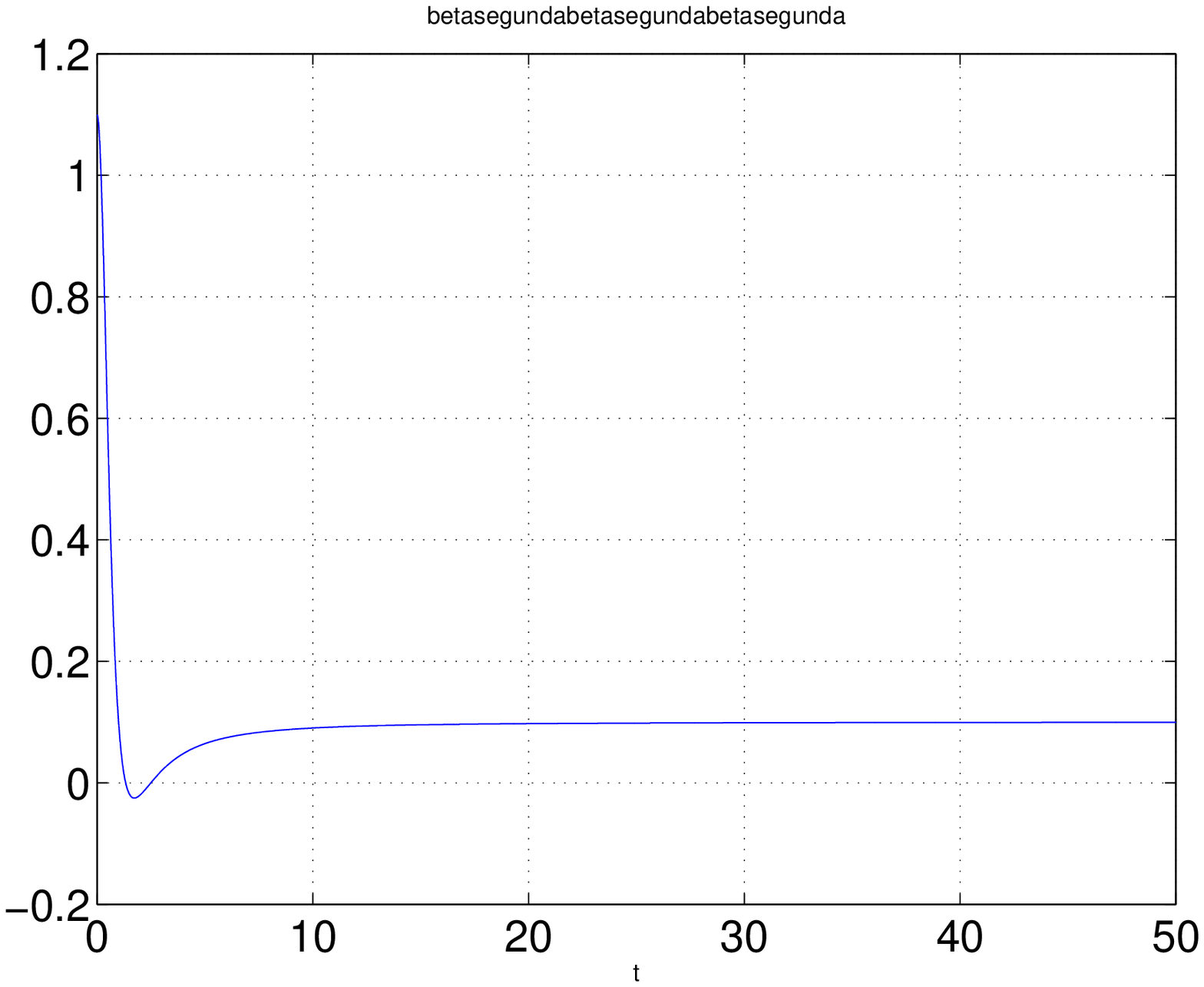}
\caption{The function $\alpha(t)=\frac{1}{1+t}+\frac{1}{10}$ of Example~\ref{Ex:betasegundabad} does not satisfy $\alpha(t^2)+2t^2\alpha'(t^2)>0$ for all $t>0$.}\label{f:betasegundabad}
\end{center}
\end{figure}

In Figure~\ref{f:betasegundabad_error} we plot the $H^1(\Omega)$-error versus the number degrees of freedom, for different polynomial degrees. For $\ell=3$ and $\ell=4$ the algorithm stopped with an estimator below the desired tolerance $10^{-6}$, although the error is around $10^{-2}$ in all cases. On the other hand, as we can see in Figure~\ref{f:betasegundabad_estimador}, the global error estimator decreases with optimal rate for the adaptive strategies, indicating that the adaptive algorithm may be converging to another solution of the nonlinear problem, different from the one given by~\eqref{E:u polar}.

\begin{figure}[h!tbp]
\begin{center}
\psfrag{Refinamiento Global}{\tiny Global refinement}
\psfrag{Estrategia del Maximo}{\tiny Maximum Strategy}
\psfrag{Estrategia de Dorfler}{\tiny D\"orfler Strategy}
\psfrag{Pendiente Optima}{\tiny Optimal slope}
\psfrag{Error}{\hspace{-10pt}\tiny $H^1$-error}
\psfrag{DOFs}{\tiny DOFs}
\psfrag{Aproximacion con polinomios de grado 1}{ \tiny Polynomial degree $\ell=1$}
\psfrag{Aproximacion con polinomios de grado 2}{ \tiny Polynomial degree $\ell=2$}
\psfrag{Aproximacion con polinomios de grado 3}{ \tiny Polynomial degree $\ell=3$}
\psfrag{Aproximacion con polinomios de grado 4}{ \tiny Polynomial degree $\ell=4$}
\includegraphics[width=.40\textwidth]{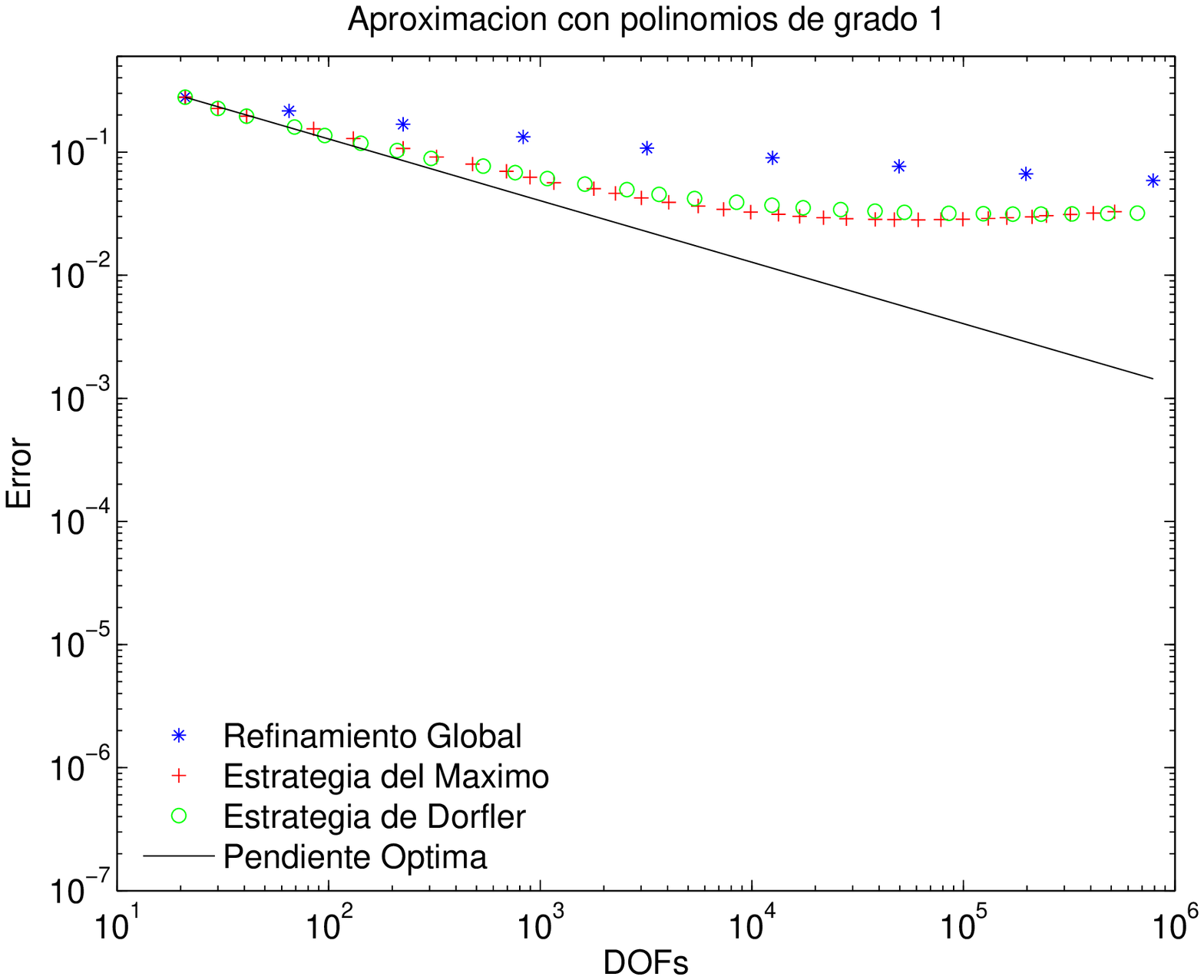}\hfil
\includegraphics[width=.40\textwidth]{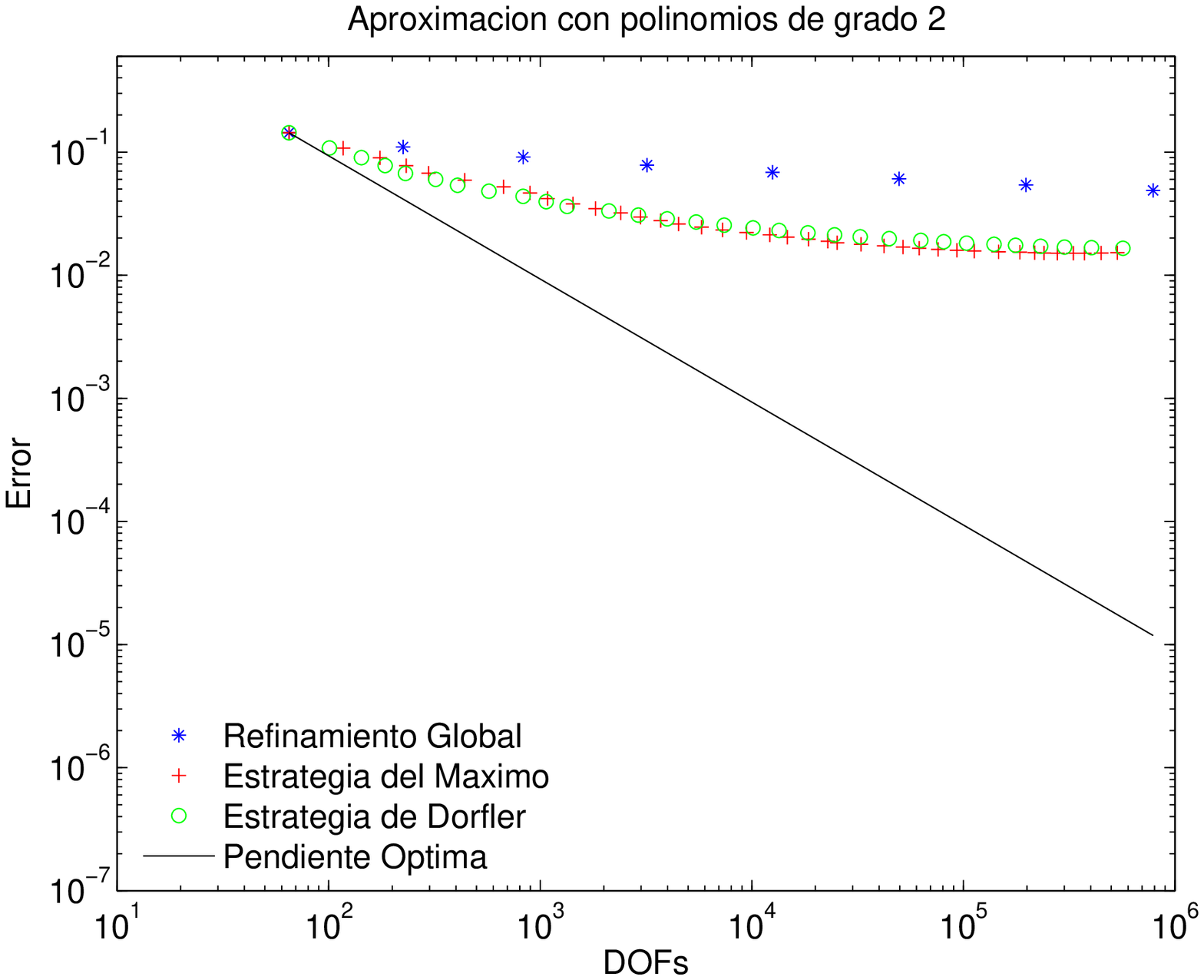}

\bigskip
\includegraphics[width=.40\textwidth]{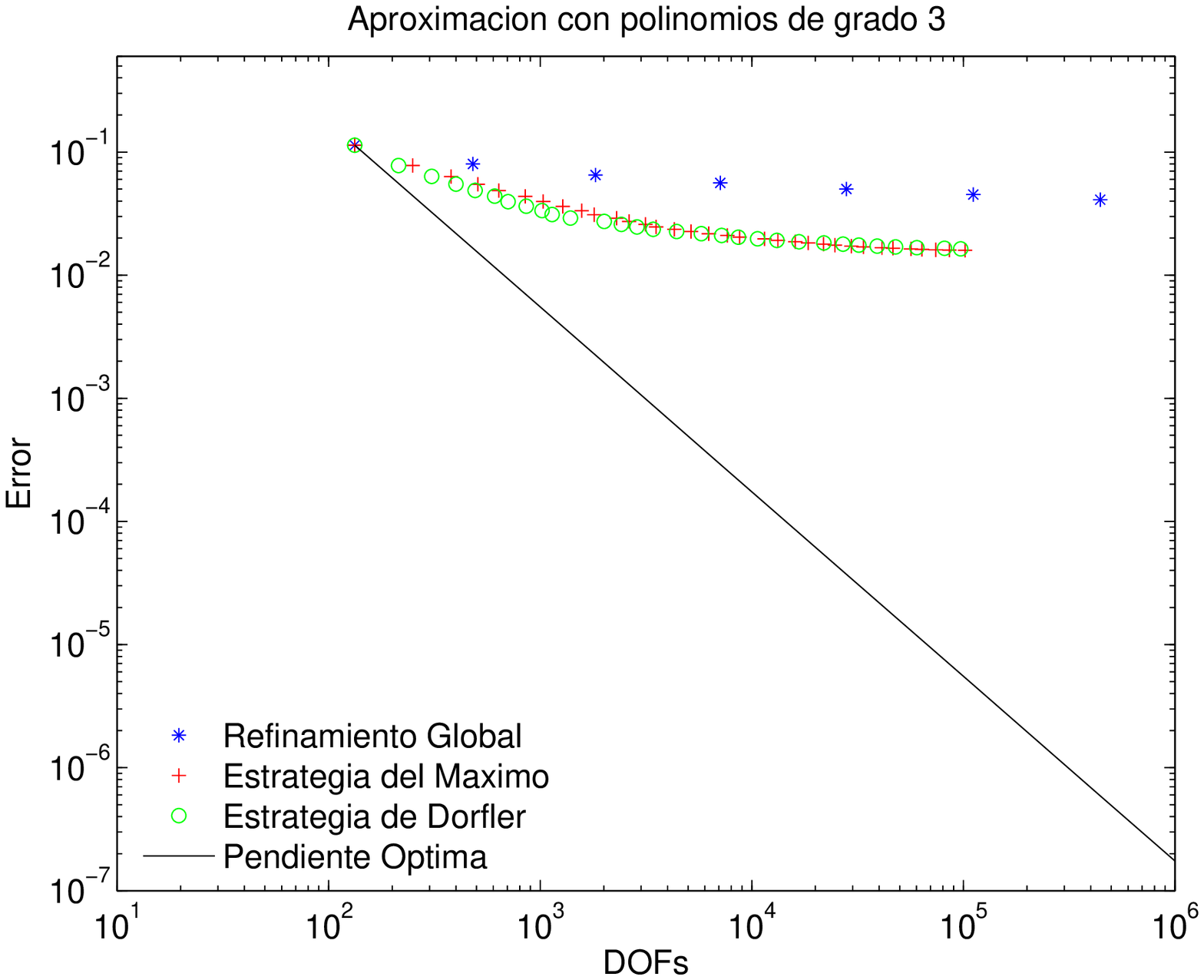}\hfil
\includegraphics[width=.40\textwidth]{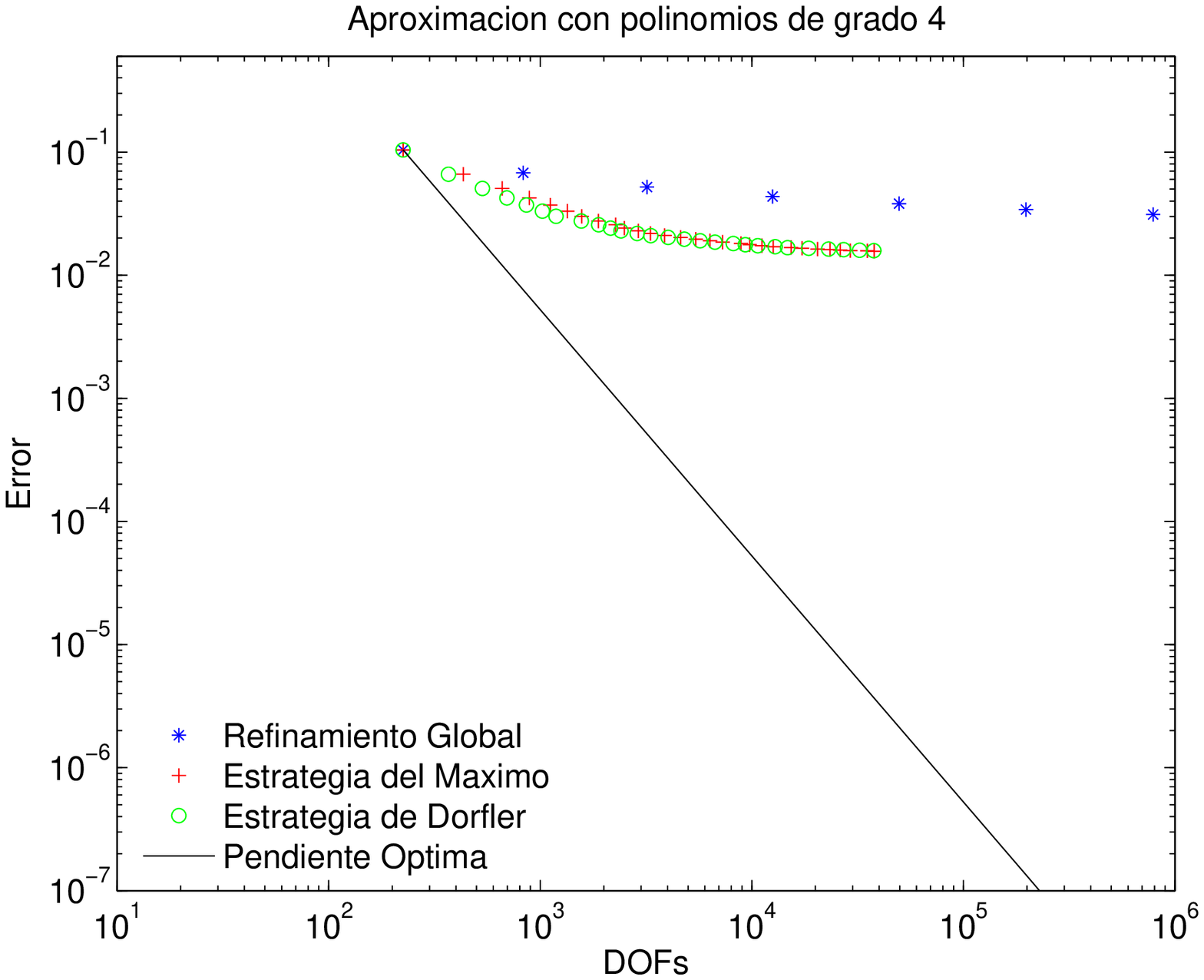}
\caption{
Error versus DOFs for Example~\ref{Ex:betasegundabad}.
%Gr\'aficos del error en norma $H^1(\Omega)$ en funci\'on de los grados de libertad (DOFs) utilizados, para la aproximaci\'on con distintos grados polinomiales, considerando el Ejemplo~\ref{Ex:betasegundabad}. 
We present the $H^1(\Omega)$-error between the exact and discrete solutions, versus DOFs. We observe that the method does not converge, but the error stagnates around $10^{-2}$. % This is an indication that the method could be converging, but not to the solution~\eqref{E:u polar} which we are considering.
Since $\alpha(t)=\frac{1}{1+t}+\frac{1}{10}$ does not satisfy the condition which guarantees uniqueness of solutions, and based on the fact that the a posteriori error estimator does tend to zero (see Figure~\ref{f:betasegundabad_estimador}) we conclude that the method converges to \emph{another solution} of the same problem, different from the one given by~\eqref{E:u polar}.
% This conjecture requires a more careful analysis and it will be study in future researchs.
}\label{f:betasegundabad_error}
\end{center}
\end{figure}

%\note{Aqu\'i viene un comentario}

Based on this remark, it seems that the adaptive algorithm converges to a solution $u_1$ such that $\|u-u_1\|_{H^1(\Omega)}\approx 10^{-2}$. We recall that $\alpha$ does not satisfy the condition~\eqref{E:Hip beta segunda} which guarantees the uniqueness of solutions.

\begin{figure}[h!tbp]
\begin{center}
\psfrag{Refinamiento Global}{\tiny Global refinement}
\psfrag{Estrategia del Maximo}{\tiny Maximum Strategy}
\psfrag{Estrategia de Dorfler}{\tiny D\"orfler Strategy}
\psfrag{Pendiente Optima}{\tiny Optimal slope}
\psfrag{Estimador global}{\hspace{-10pt}\tiny global error estimator}
\psfrag{DOFs}{\tiny DOFs}
\psfrag{Aproximacion con polinomios de grado 1}{ \tiny Polynomial degree $\ell=1$}
\psfrag{Aproximacion con polinomios de grado 2}{ \tiny Polynomial degree $\ell=2$}
\psfrag{Aproximacion con polinomios de grado 3}{ \tiny Polynomial degree $\ell=3$}
\psfrag{Aproximacion con polinomios de grado 4}{ \tiny Polynomial degree $\ell=4$}
\includegraphics[width=.40\textwidth]{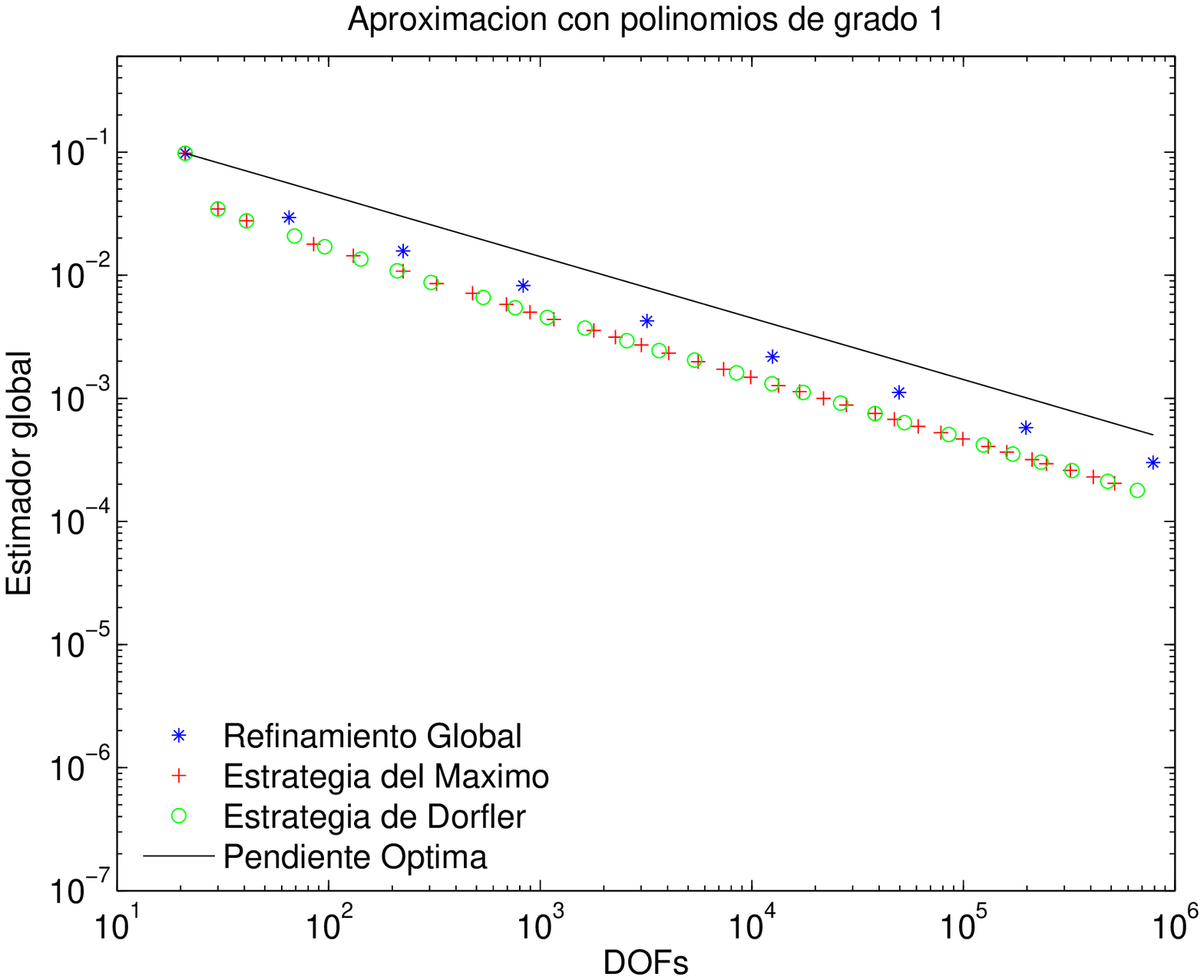}\hfil
\includegraphics[width=.40\textwidth]{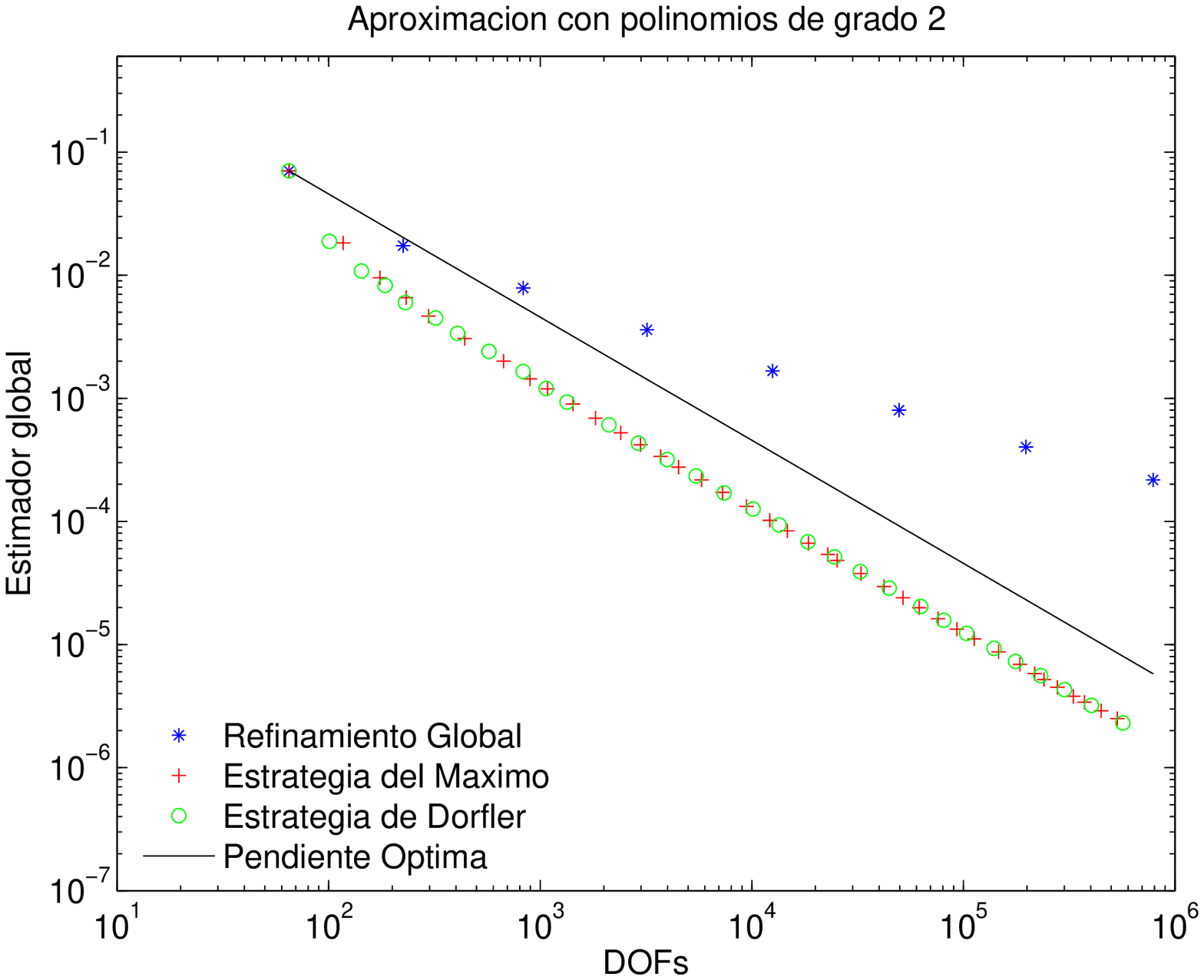}

\bigskip
\includegraphics[width=.40\textwidth]{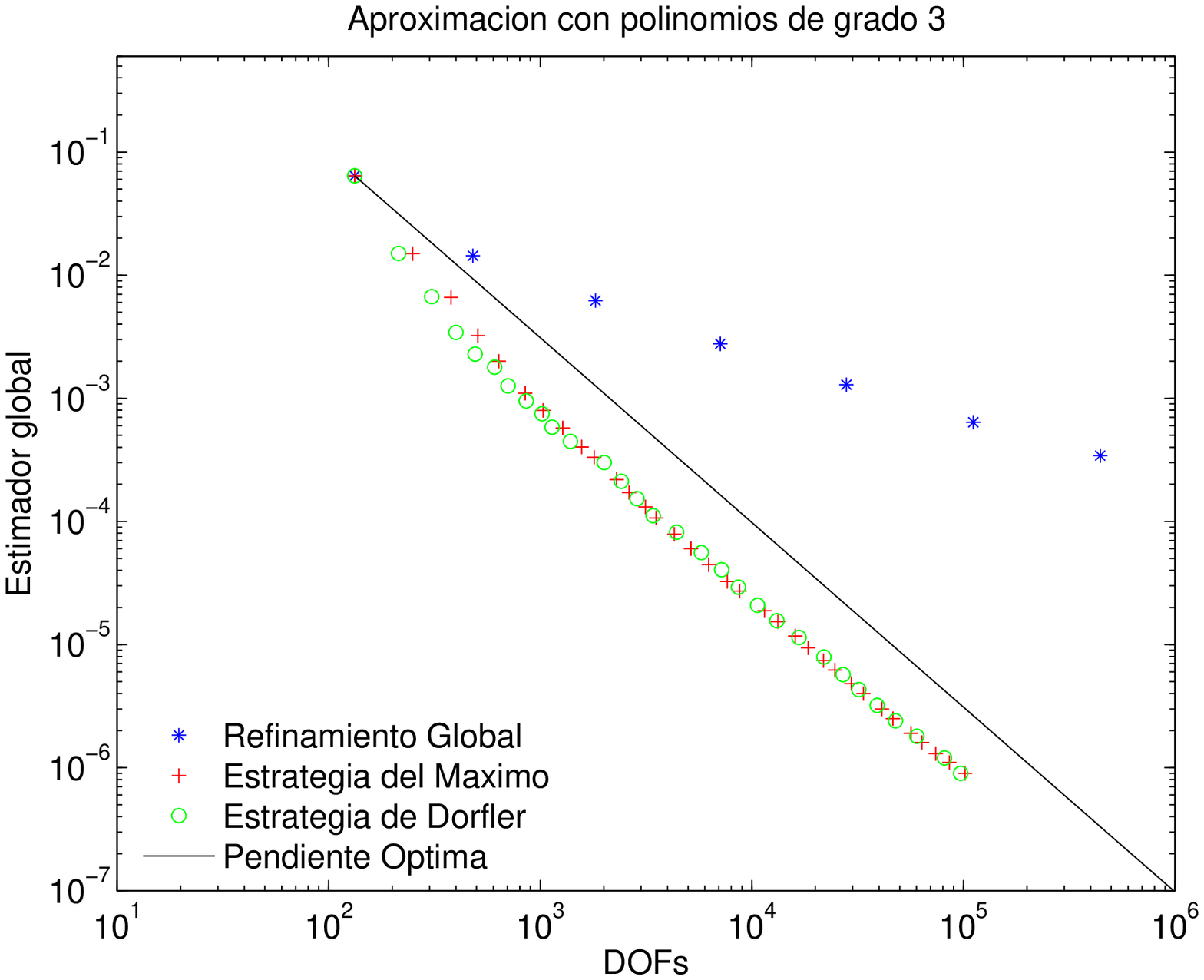}\hfil
\includegraphics[width=.40\textwidth]{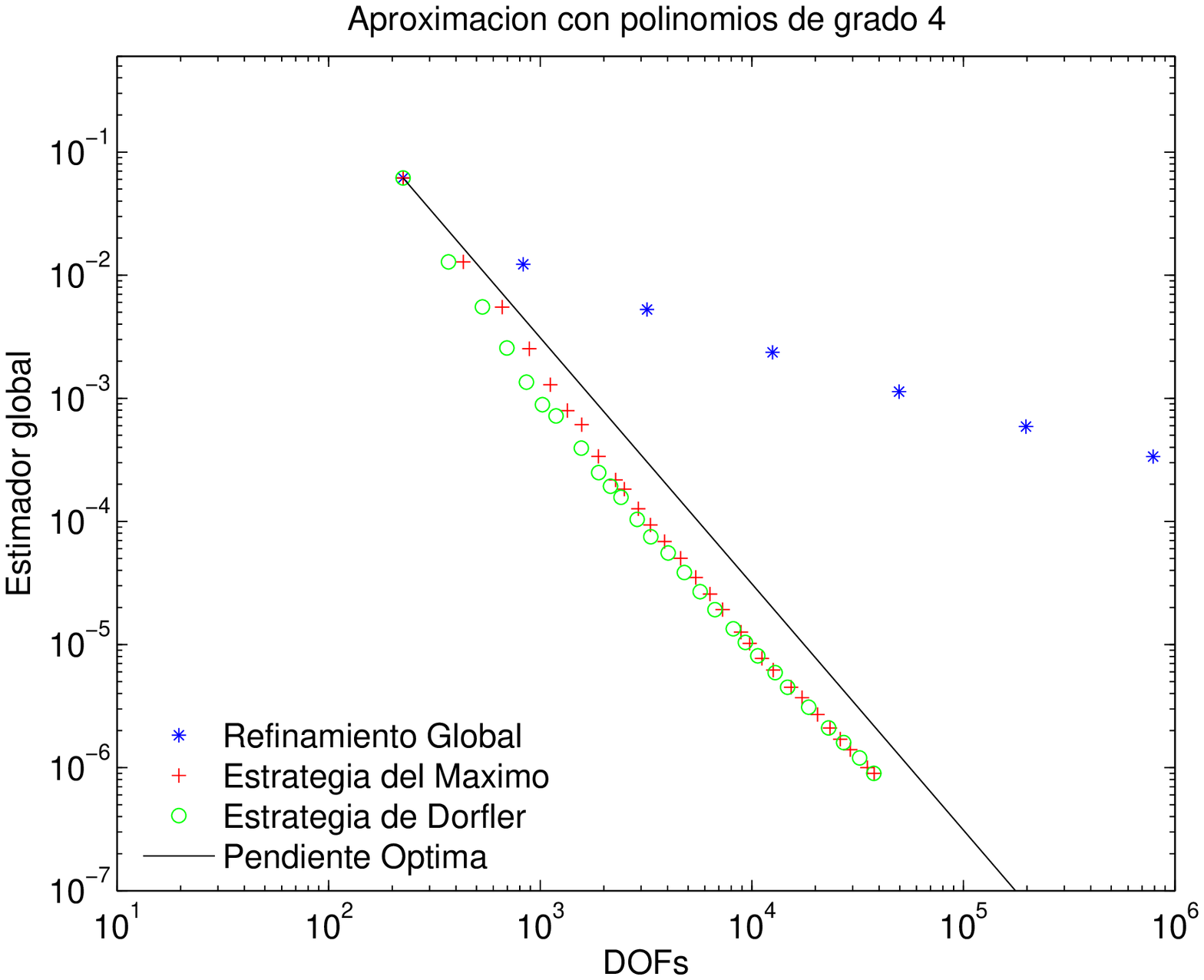}
\caption{
Estimator versus DOFs for Example~\ref{Ex:betasegundabad}.
%Gr\'aficos del error en norma $H^1(\Omega)$ en funci\'on de los grados de libertad (DOFs) utilizados, para la aproximaci\'on con distintos grados polinomiales, considerando el Ejemplo~\ref{Ex:betasegundabad}. 
We present the global error estimator $\eta_k(\Tau_k)$ versus DOFs used to represent each discrete solution. We note that for the adaptive strategies the global error estimator decreases with the optimal rate for the $H^1$-error, although the error does not tend to zero (see Figure~\ref{f:betasegundabad_error}). In this case, $\alpha(t)=\frac{1}{1+t}+\frac{1}{10}$ does not satisfy the condition which guarantees uniqueness of solutions. It seems that the method converges to \emph{another solution} of the problem.% We will study this situation in future researchs.
}\label{f:betasegundabad_estimador}
\end{center}
\end{figure}

\end{example}

%\clearpage

% Ejemplo 3
\begin{example}[About the hypothesis~\eqref{E:Hip alpha decreciente}]\label{Ex:creciente}
We consider
$$\alpha(t)=-\frac12\e^{-\frac{3}{2}t}+1,\qquad t>0,$$
which satisfies the hypothesis~\eqref{E:Hip beta segunda} related to the  well-posedness of the problem but not~\eqref{E:Hip alpha decreciente}, because $\alpha$ is monotone increasing (see Figure~\ref{f:creciente}).

\begin{figure}[h!tbp]
\begin{center}
\psfrag{alpha}{\hspace{-3pt}\tiny $\alpha(t)$}
\psfrag{betasegundabetasegundabetasegunda}{\tiny$\alpha(t^2)+2t^2\alpha'(t^2)$}
\psfrag{t}{\tiny $t$}
\includegraphics[width=.40\textwidth]{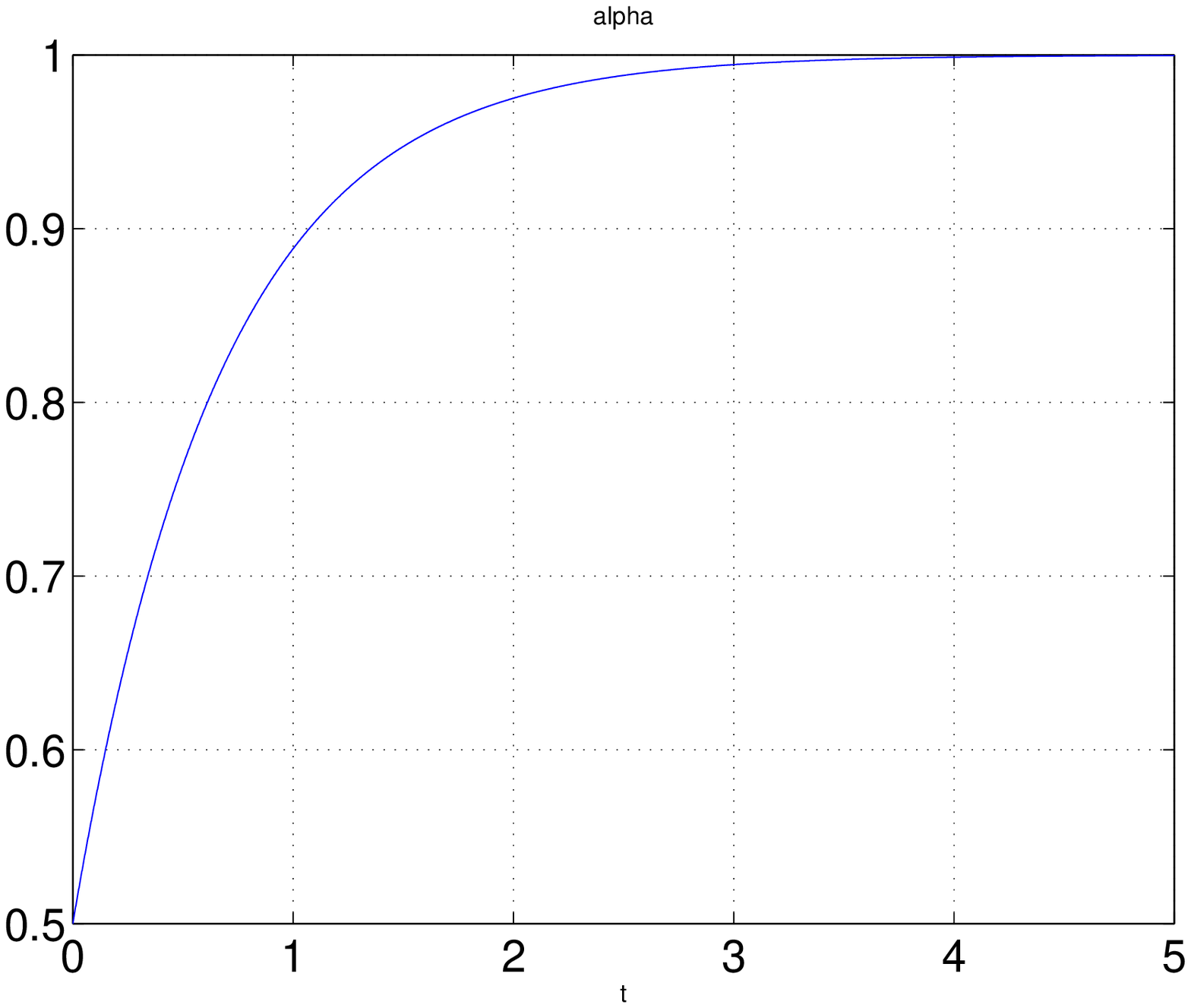}
\hfil
\includegraphics[width=.40\textwidth]{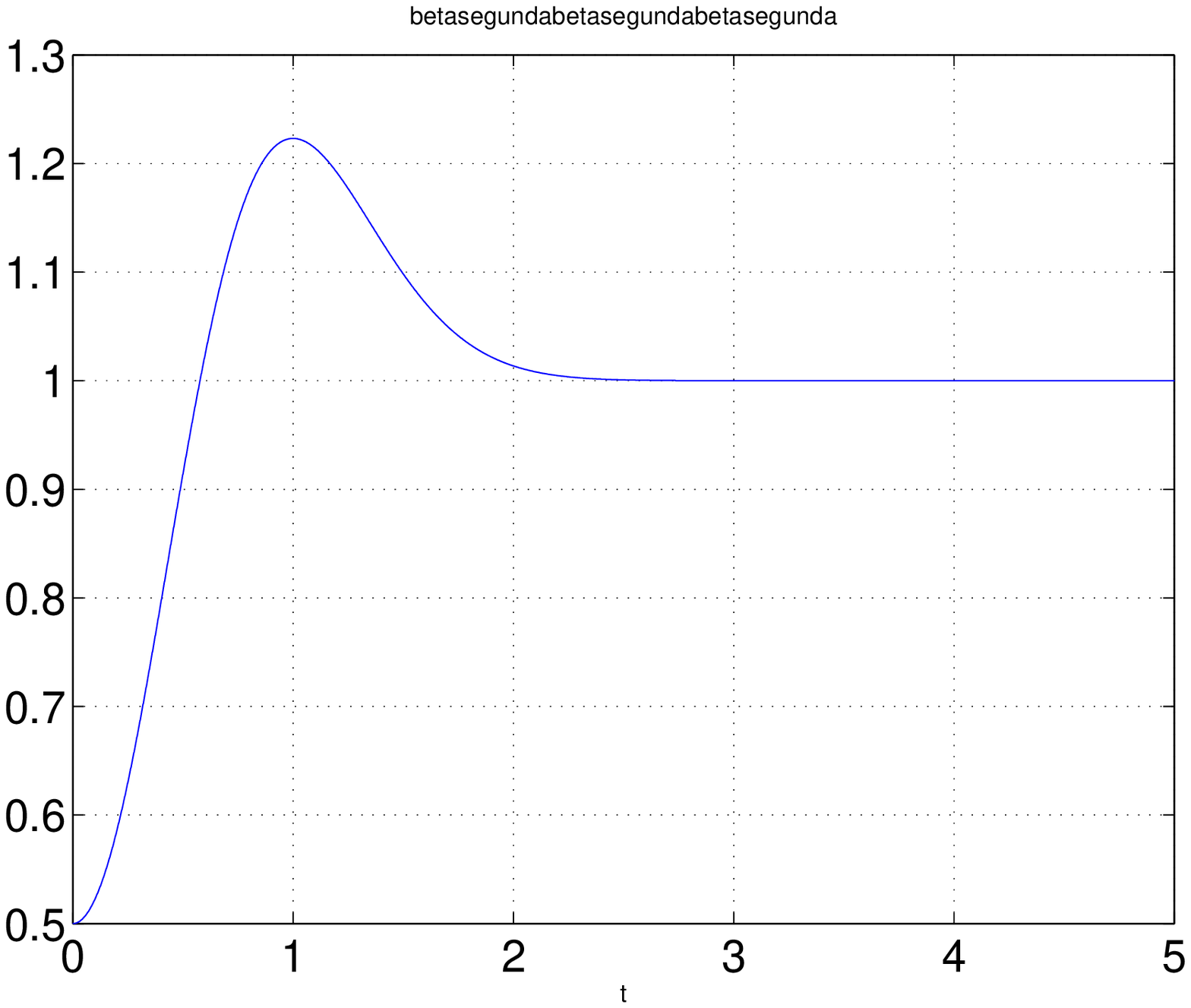}
\caption{The function $\alpha(t)=-\frac12\e^{-\frac{3}{2}t}+1$ of Example~\ref{Ex:creciente} satisfies the hypothesis~\eqref{E:Hip beta segunda} but is monotone increasing.}\label{f:creciente}
\end{center}
\end{figure}

In Figure~\ref{f:creciente_error} we plot $H^1(\Omega)$-error versus the number of degrees of freedom, for finite elements of degrees $\ell = 1, 2, 3, 4$. Note that in this case the optimal convergence rate $\text{DOFs}^{-\ell/2}$ is still observed for the adaptive algorithm. This is an indication that the assumption~\eqref{E:Hip alpha decreciente} about $\alpha$ being monotone decreasing can be superfluous, and only an artificial requirement for the presented proof (see Lemma~\ref{L:key property}). A more detailed study about this hypothesis is subject of future research, and beyond the scope of this article.

\begin{figure}[h!tbp]
\begin{center}
\psfrag{Refinamiento Global}{\tiny Global refinement}
\psfrag{Estrategia del Maximo}{\tiny Maximum Strategy}
\psfrag{Estrategia de Dorfler}{\tiny D\"orfler Strategy}
\psfrag{Pendiente Optima}{\tiny Optimal slope}
\psfrag{Error}{\hspace{-10pt}\tiny $H^1$-error}
\psfrag{DOFs}{\tiny DOFs}
\psfrag{Aproximacion con polinomios de grado 1}{ \tiny Polynomial degree $\ell=1$}
\psfrag{Aproximacion con polinomios de grado 2}{ \tiny Polynomial degree $\ell=2$}
\psfrag{Aproximacion con polinomios de grado 3}{ \tiny Polynomial degree $\ell=3$}
\psfrag{Aproximacion con polinomios de grado 4}{ \tiny Polynomial degree $\ell=4$}
\includegraphics[width=.40\textwidth]{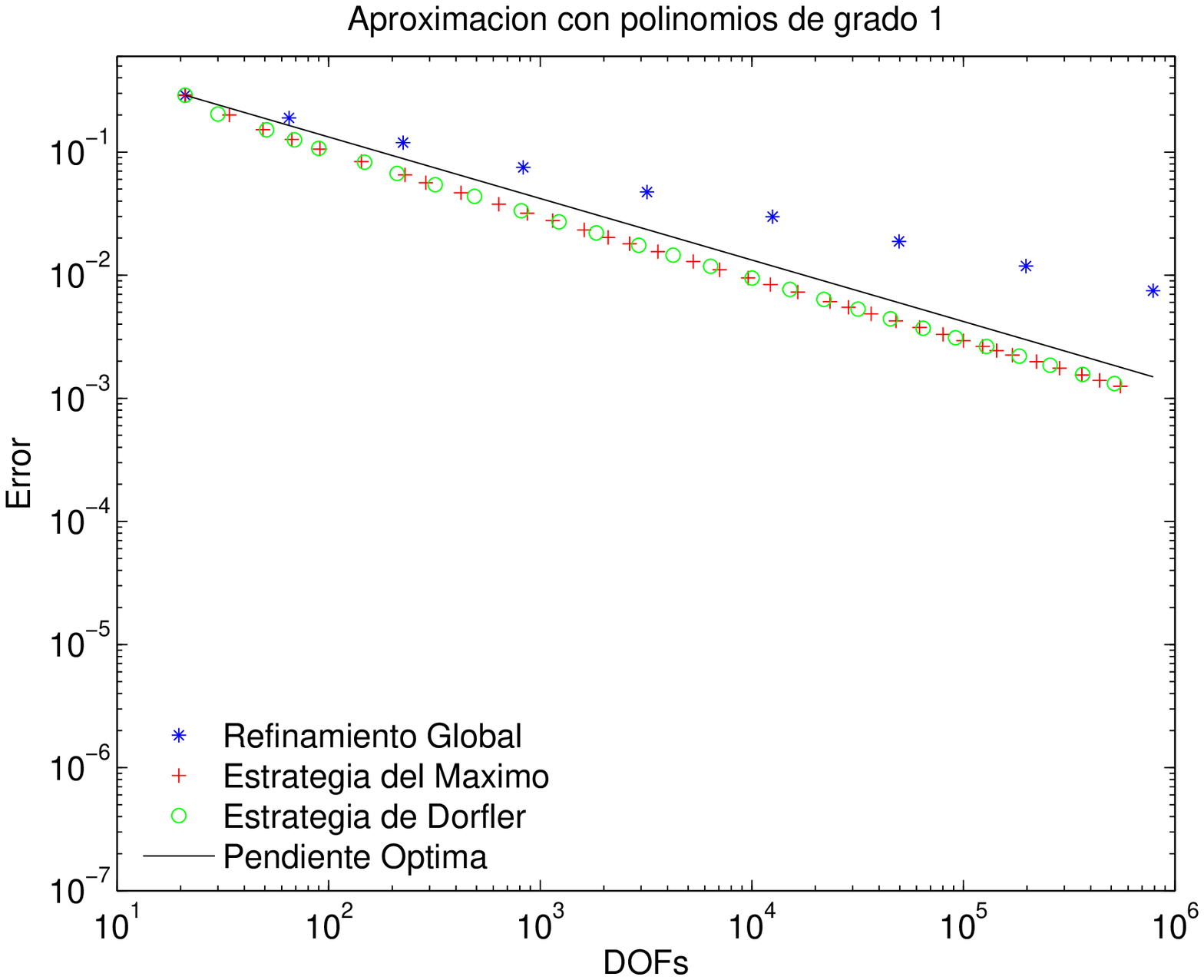}\hfil
\includegraphics[width=.40\textwidth]{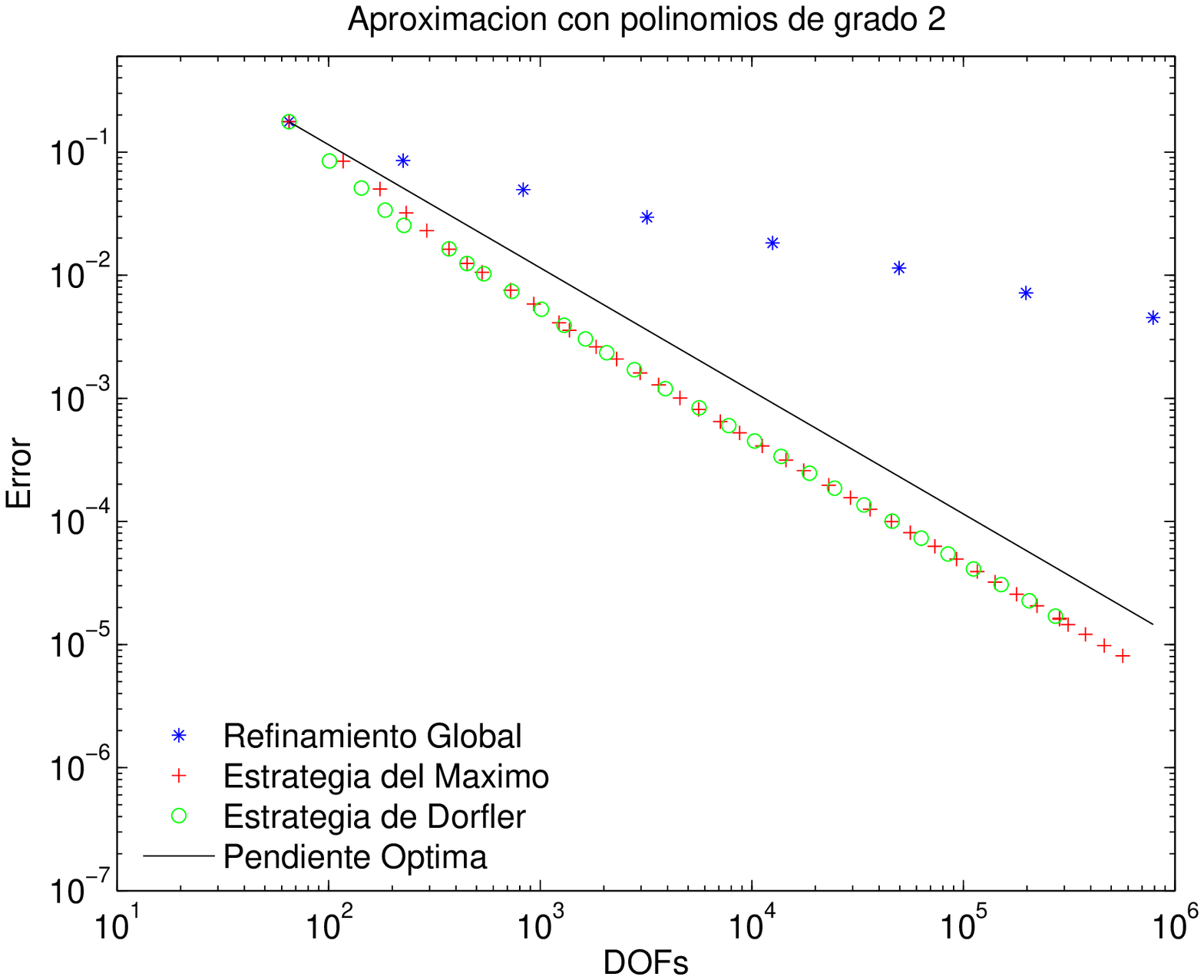}

\bigskip
\includegraphics[width=.40\textwidth]{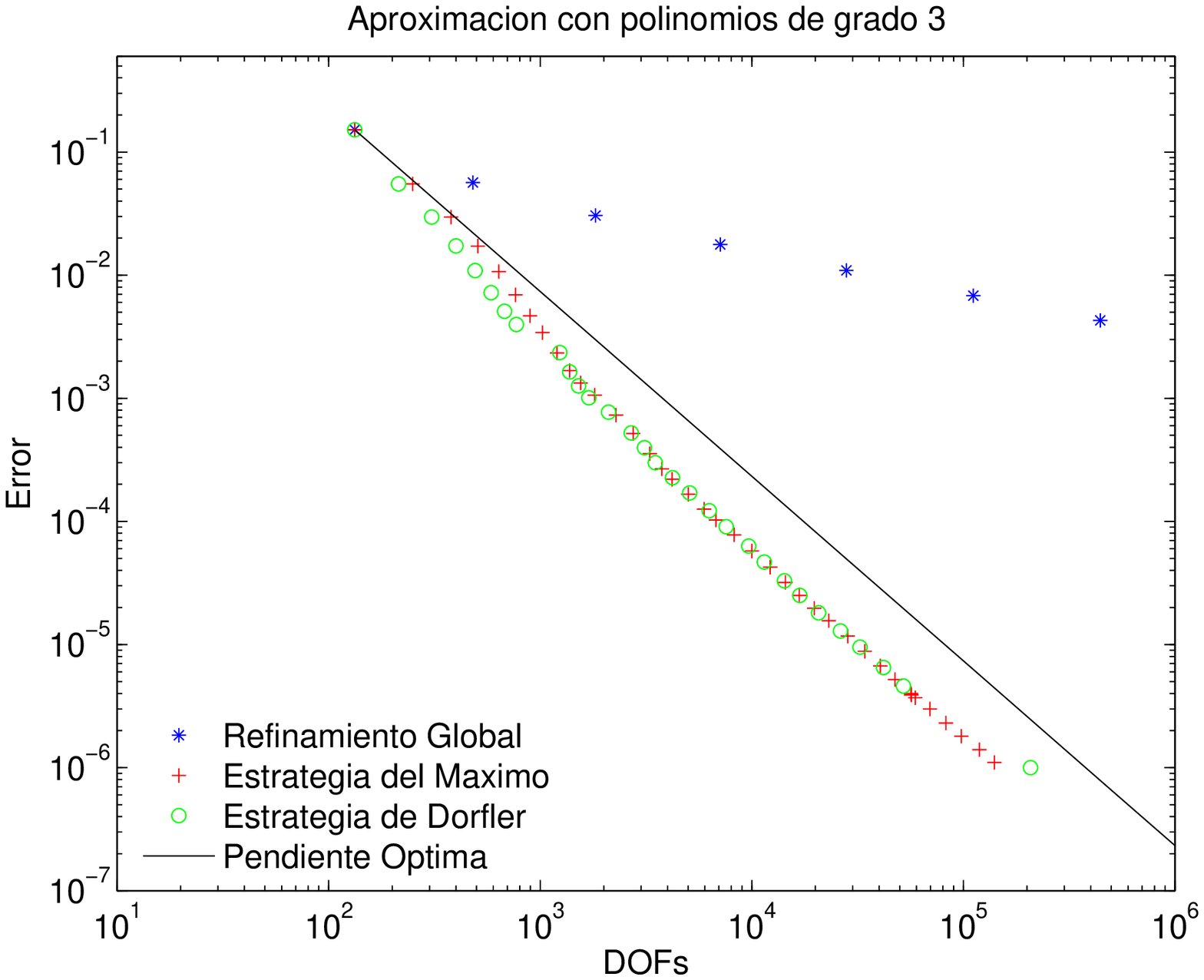}\hfil
\includegraphics[width=.40\textwidth]{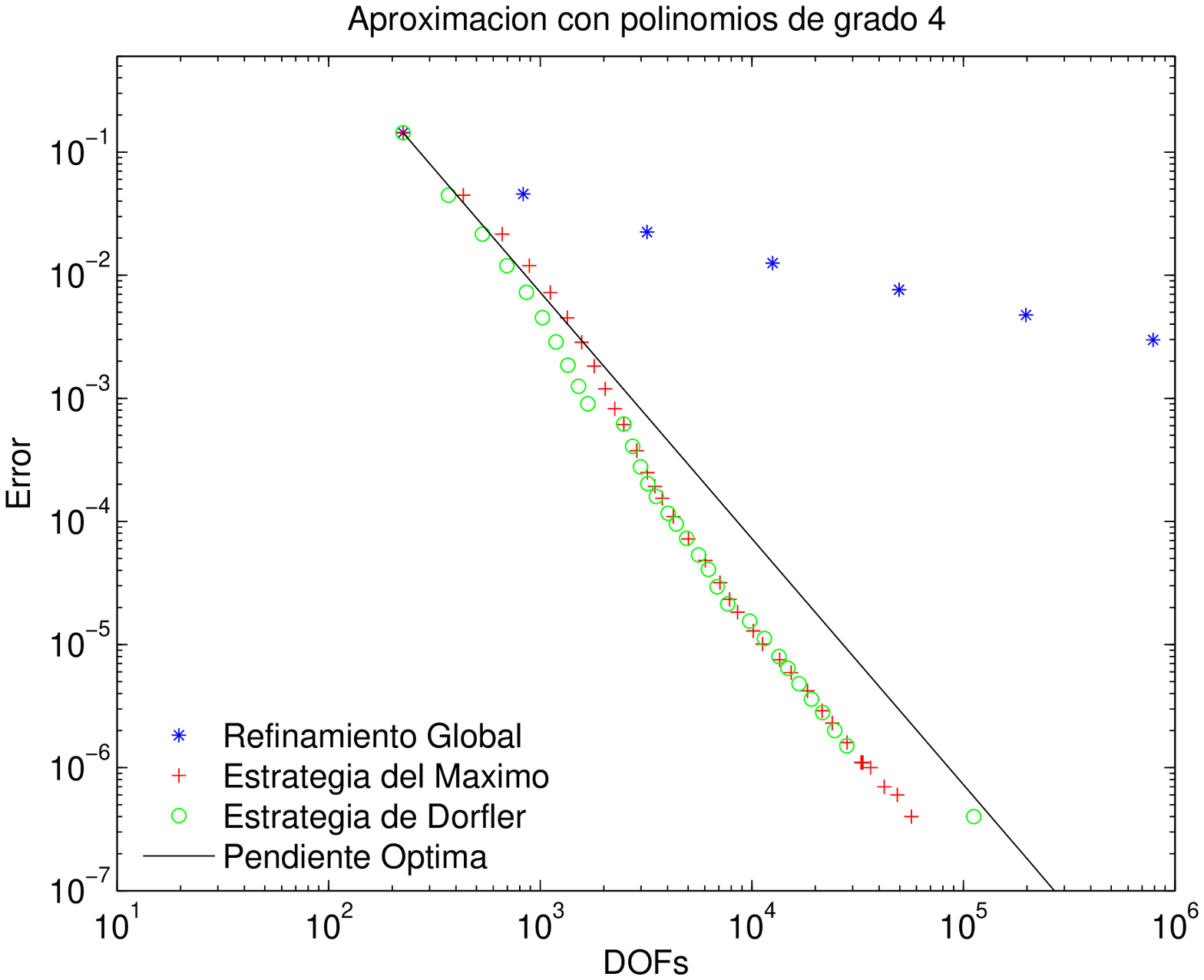}
\caption{
Error versus DOFs for Example~\ref{Ex:creciente}.
%Gr\'aficos del error en norma $H^1(\Omega)$ en funci\'on de los grados de libertad (DOFs) utilizados, para la aproximaci\'on con distintos grados polinomiales, considerando el Ejemplo~\ref{Ex:creciente}.
We present the $H^1(\Omega)$-error between the exact and discrete solutions, versus DOFs. We note that the convergence rate is optimal for the considered adaptive strategies, although the function $\alpha(t)=-\frac12\e^{-\frac{3}{2}t}+1$ is not monotone decreasing. This could mean that this hypothesis is not necessary for the convergence of the Adaptive Algorithm, which performs better than expected by our theory.
}\label{f:creciente_error}
\end{center}
\end{figure}
\end{example}

%\clearpage

% Ejemplo 5
\begin{example}\label{Ex:derivadainfinita}
Finally, we consider an extreme case, with
$$\alpha(t)=2-\frac{\sqrt{t}}{1+\sqrt{t}},\qquad t>0,$$
which satisfies~\eqref{E:Hip beta segunda} and~\eqref{E:Hip alpha decreciente}, but $\lim_{t\to 0^+} \alpha'(t)=-\infty$, as can be observed in Figure~\ref{f:derivadainfinita}. This means that $\alpha$ is not Lipschitz continuous. Since we only require that $|t\alpha'(t)|$ is bounded (cf.~\eqref{E:Hip beta segunda}), it still satisfies the assumptions of the convergence theory, and optimality is observed regardless of the fact that $\lim_{t\to 0^+} \alpha'(t)=-\infty$.

\begin{figure}[h!tbp]
\begin{center}
\psfrag{alpha}{\hspace{-3pt}\tiny $\alpha(t)$}
\psfrag{betasegundabetasegundabetasegunda}{\tiny$\alpha(t^2)+2t^2\alpha'(t^2)$}
\psfrag{t}{\tiny $t$}
\includegraphics[width=.40\textwidth]{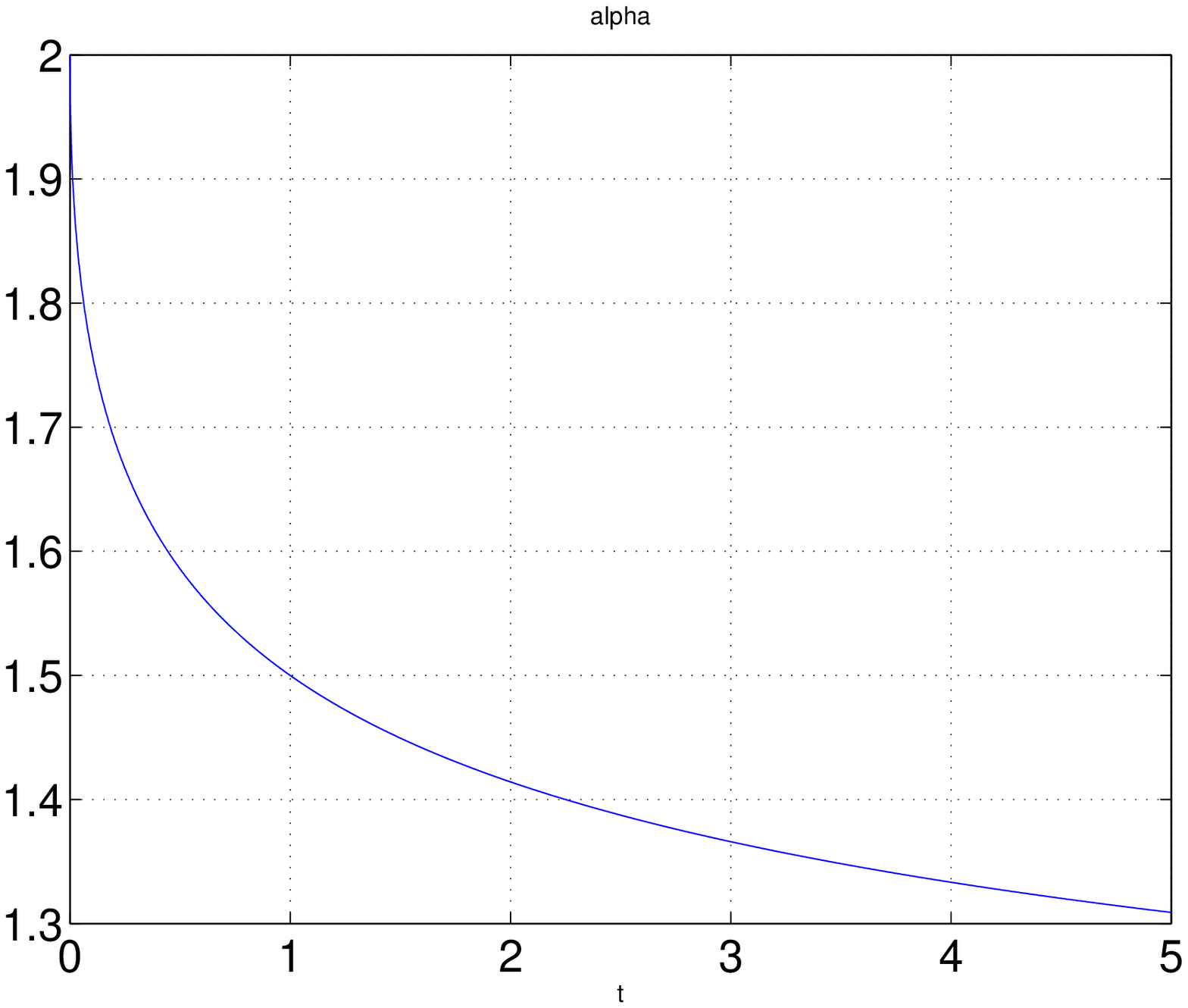}
\hfil
\includegraphics[width=.40\textwidth]{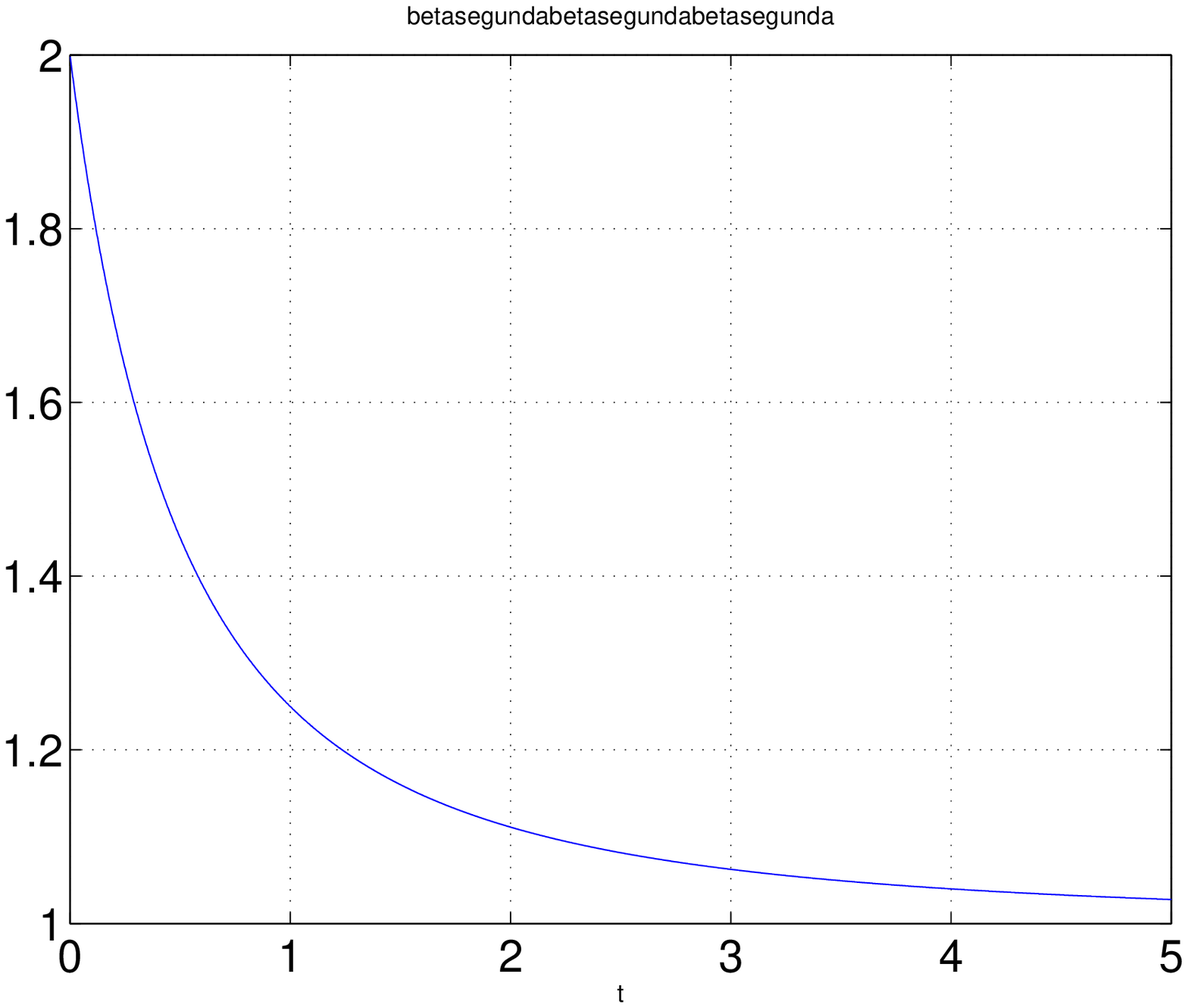}
\caption{The function $\alpha(t)=2-\frac{\sqrt{t}}{1+\sqrt{t}}$ of Example~\ref{Ex:derivadainfinita} satisfies our assumptions which do not require that $\alpha$ is Lipschitz continuous.}\label{f:derivadainfinita}
\end{center}
\end{figure}

Figure~\ref{f:derivadainfinita_error} shows $H^1(\Omega)$-error versus DOFs, for polynomial degrees $\ell = 1, 2, 3, 4$. We obtain optimal convergence rate $\text{DOFs}^{-\ell/2}$ for the adaptive strategies. Thus, even when $\alpha$ is not Lipschitz, the convergence rate is optimal. As a consequence we conclude that it should not be necessary to make additional assumptions in order to prove optimality.

\begin{figure}[h!tbp]
\begin{center}
\psfrag{Refinamiento Global}{\tiny Global refinement}
\psfrag{Estrategia del Maximo}{\tiny Maximum Strategy}
\psfrag{Estrategia de Dorfler}{\tiny D\"orfler Strategy}
\psfrag{Pendiente Optima}{\tiny Optimal slope}
\psfrag{Error}{\hspace{-10pt}\tiny $H^1$-error}
\psfrag{DOFs}{\tiny DOFs}
\psfrag{Aproximacion con polinomios de grado 1}{ \tiny Polynomial degree $\ell=1$}
\psfrag{Aproximacion con polinomios de grado 2}{ \tiny Polynomial degree $\ell=2$}
\psfrag{Aproximacion con polinomios de grado 3}{ \tiny Polynomial degree $\ell=3$}
\psfrag{Aproximacion con polinomios de grado 4}{ \tiny Polynomial degree $\ell=4$}
\includegraphics[width=.40\textwidth]{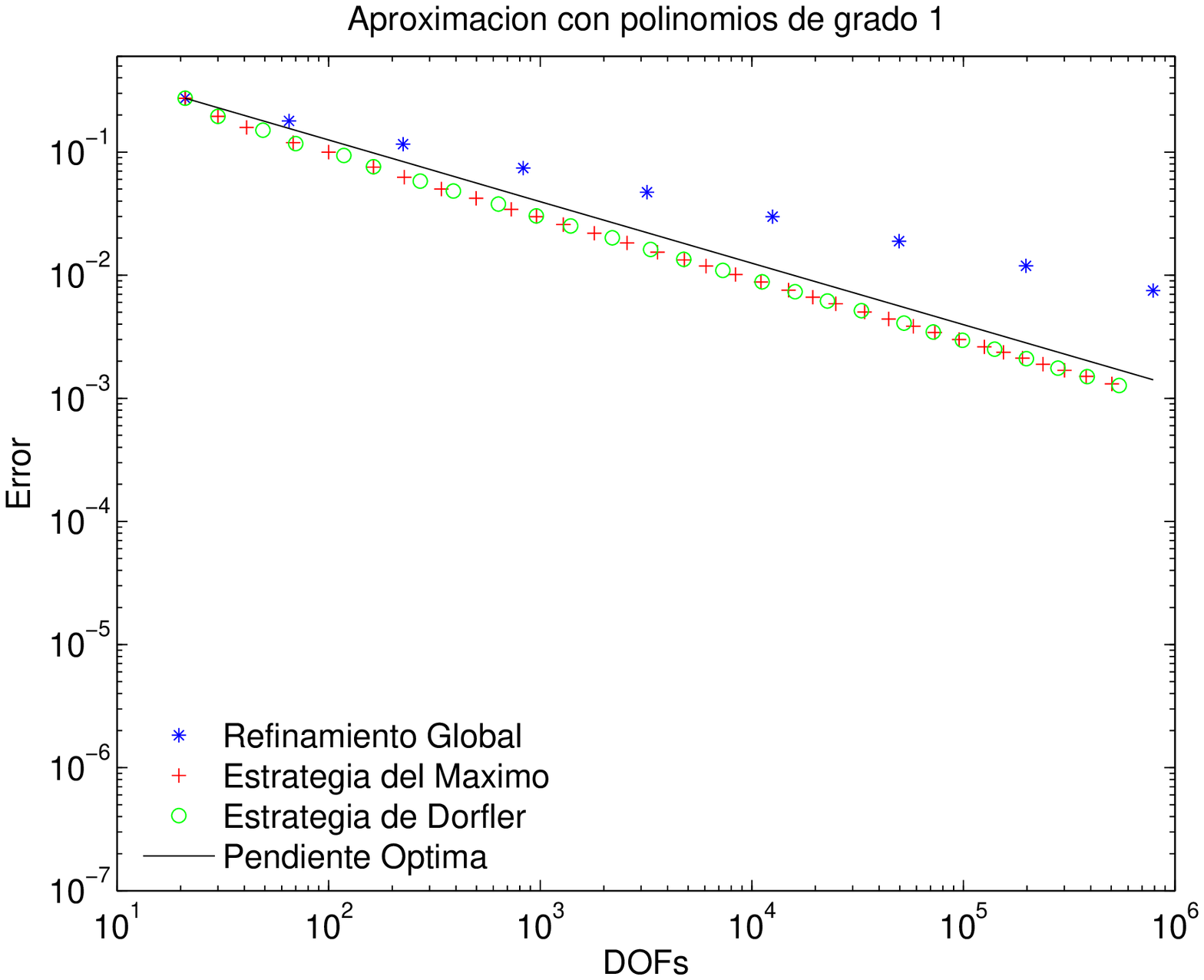}\hfil
\includegraphics[width=.40\textwidth]{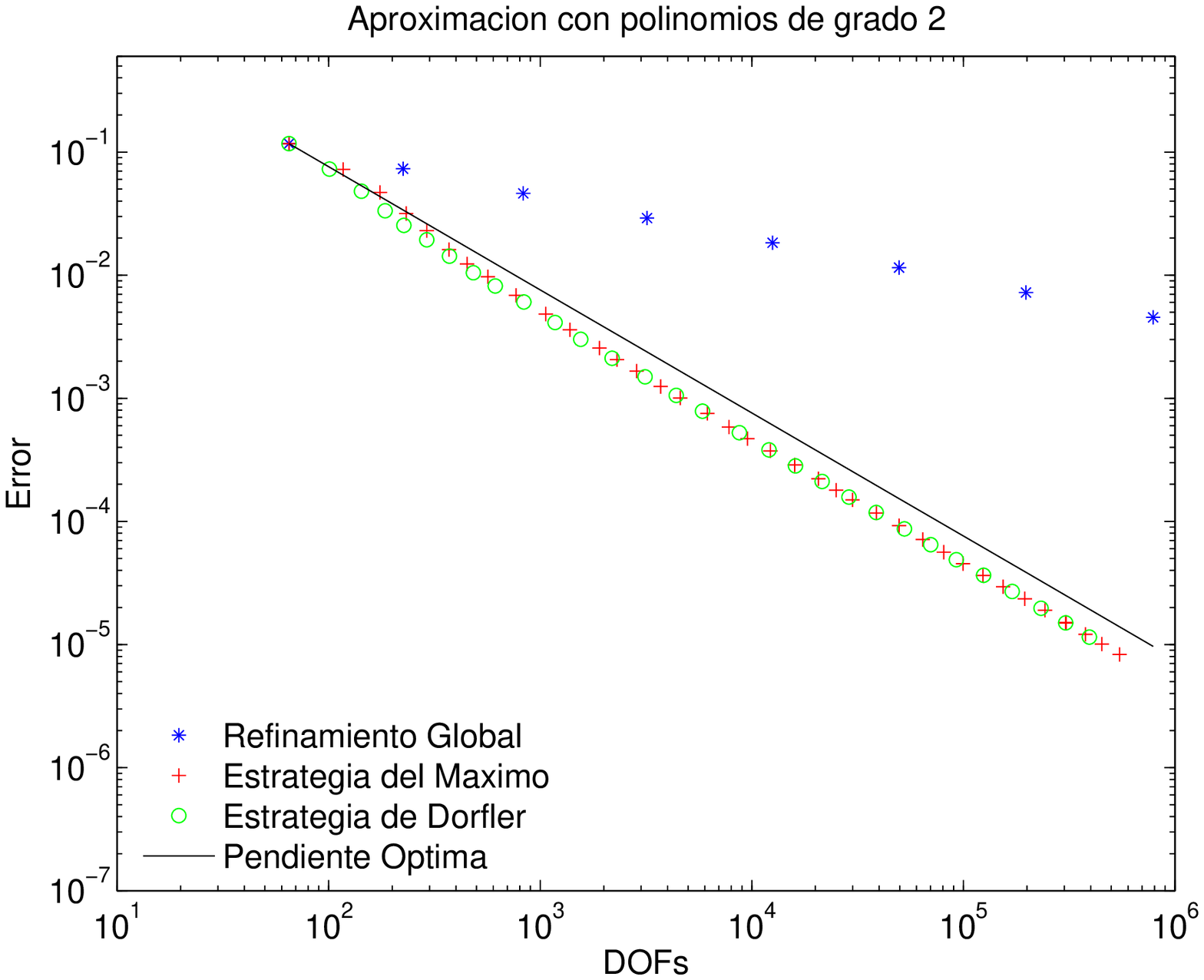}

\bigskip
\includegraphics[width=.40\textwidth]{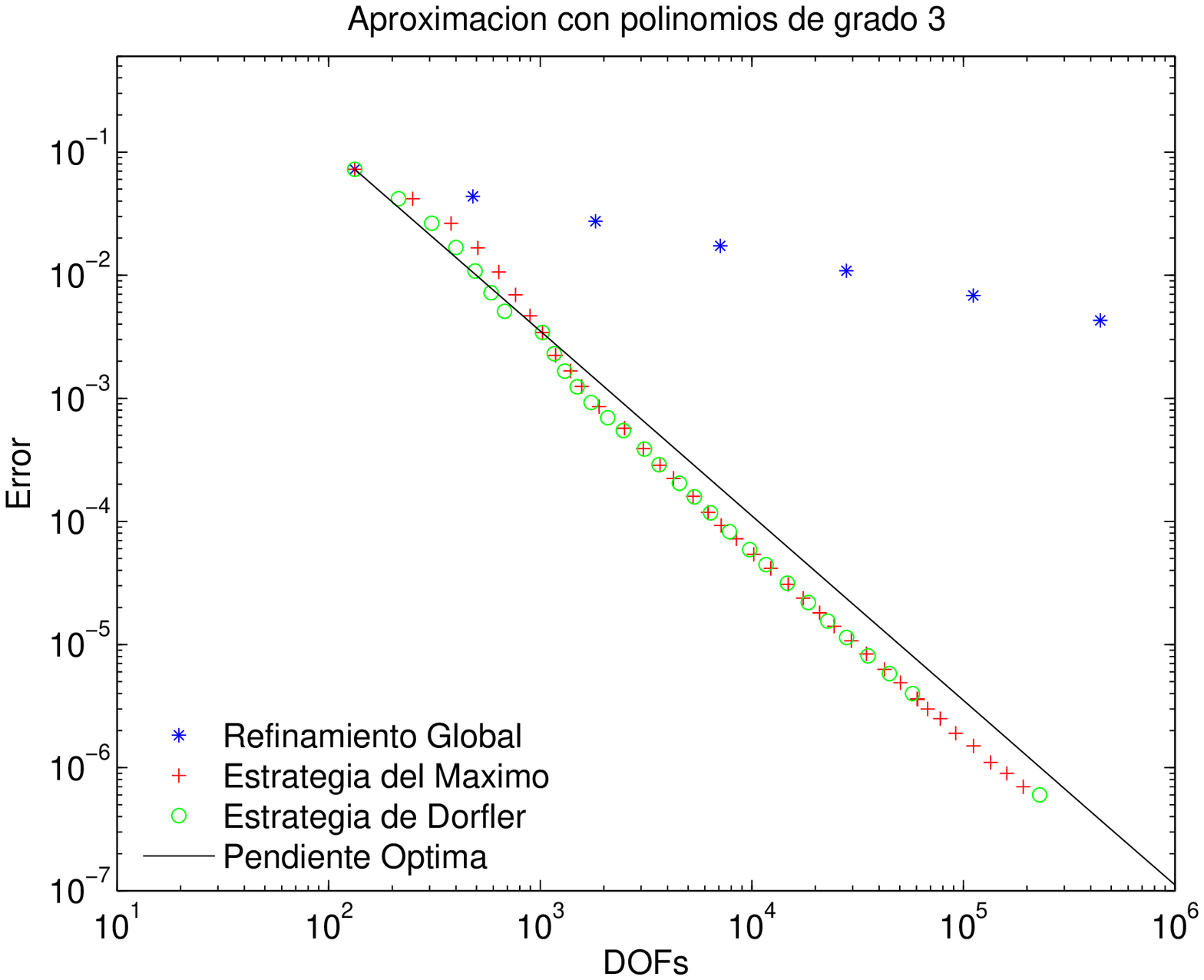}\hfil
\includegraphics[width=.40\textwidth]{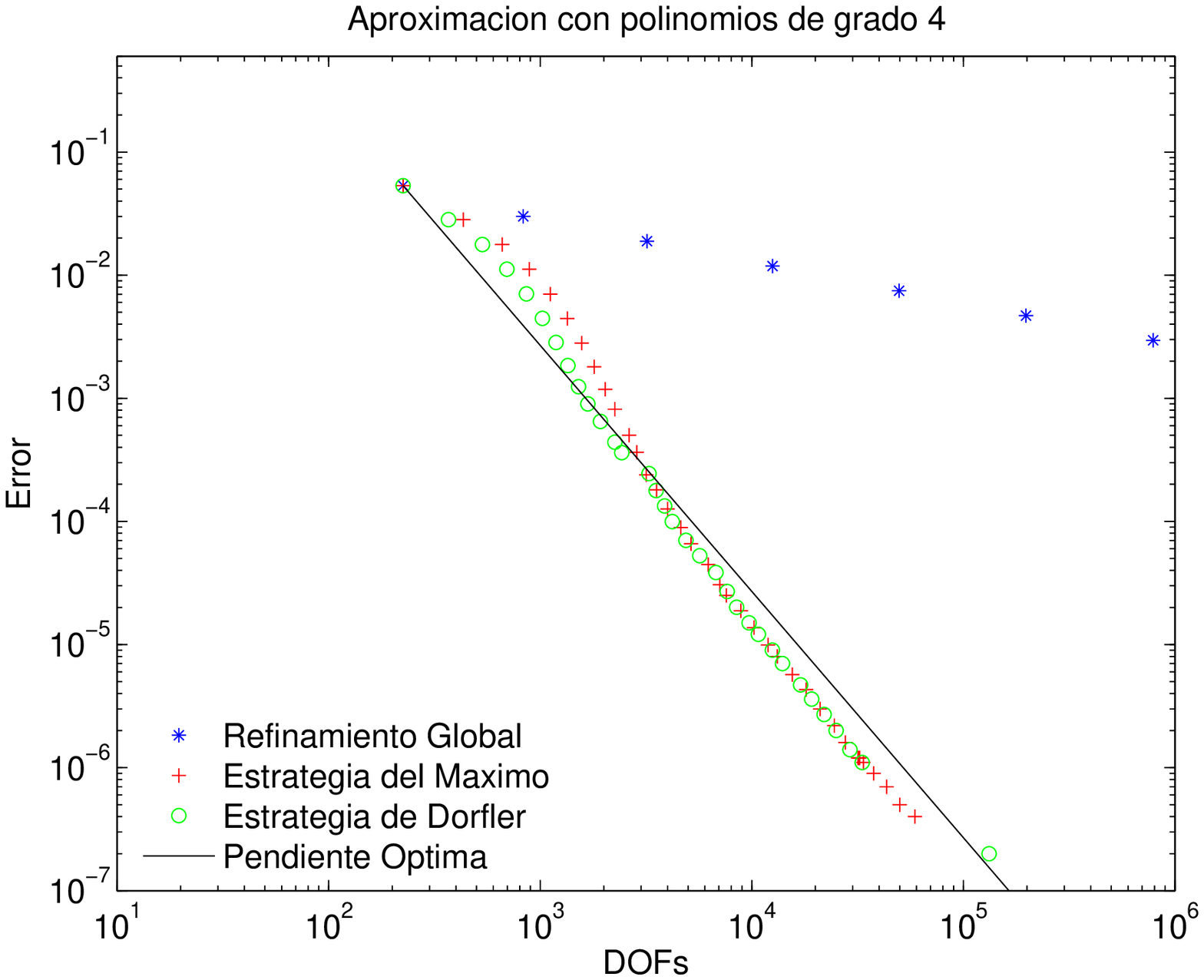}
\caption{
Error versus DOFs for Example~\ref{Ex:derivadainfinita}.
%Gr\'aficos del error en norma $H^1(\Omega)$ en funci\'on de los grados de libertad (DOFs) utilizados, para la aproximaci\'on con distintos grados polinomiales, considerando el Ejemplo~\ref{Ex:derivadainfinita}. 
We present the $H^1(\Omega)$-error between the exact and discrete solutions, versus DOFs. We note that the convergence rate is optimal for the considered adaptive strategies, even though the function $\alpha(t)=2-\frac{\sqrt{t}}{1+\sqrt{t}}$ has an infinite derivative at $0$. This example falls inside the theory, and the fact that $\alpha$ is not Lipschitz continuous does not destroy the optimality of the sequence of discrete solutions.
}\label{f:derivadainfinita_error}
\end{center}
\end{figure}
\end{example}

\subsection{Unknown solution}

In this section we use the Adaptive Algorithm to approximate a solution to a prescribed mean curvature problem. We consider the problem
\begin{equation}\label{E:curvature}
\left\{\begin{aligned}
-\nabla \cdot \left[ \frac{\nabla u}{\sqrt{1+|\nabla u|^2}} \right] &= f \quad & &\text{in }\Omega = (-1,1)\times(-1,1) 
\\
u &= 0 \quad& &\text{on }\partial\Omega,
\end{aligned}\right.
\end{equation}
with 
\[
f(x) = \begin{cases}
5 \quad &\text{if } |x| \le 1/3,
\\
-3 \quad &\text{if } 1/3 < |x| \le 2/3,
\\
0 \quad&\text{otherwise}.
\end{cases}
\]
The function $\alpha(t) = 1/\sqrt{1+t}$ corresponding to this equation does not fulfill~\eqref{E:Hip beta segunda} because $\alpha(t),\alpha'(t) \to 0$ as $t\to\infty$. Nevertheless, for many right-hand side functions $f$, as the one stated above, the solution belongs to $W^{1,\infty}(\Omega)$. This implies that $|\nabla u|$ is bounded and $\alpha$ could be replaced by a function satisfying~\eqref{E:Hip beta segunda} without changing the solution. This is not needed in practice, but is rather a theoretical tool for proving that the assumptions hold in some sense.

 We experimented with several right-hand sides $f$ and observed that the method performs robustly whenever $u \in W^{1,\infty}(\Omega)$. When the solution has an unbounded gradient, the method produces a sequence with a maximum value increasing to infinity. This is a drawback of our method, since it cannot be used to approximate singular solutions (not belonging to $W^{1,\infty}(\Omega)$), if $\alpha$ does not satisfy~\eqref{E:Hip beta segunda}. It is worth recalling that adaptivity is still a good tool for approximating regular solutions, specially taking into account the fact that in this context \emph{regularity} is a concept relative to the polynomial degree. Using adaptivity with higher order finite elements can improve drastically the performance even when the solution is in $H^2(\Omega)\cap W^{1,\infty}(\Omega)$.

We present in Figure~\ref{F:curvature} a picture of the solution obtained with our method and several meshes at different iteration steps, for polynomial degrees 1, 2 and 3. It is worth observing the stronger grading obtained for higher polynomial degrees. This obeys the fact that we mention above, that the singularity of the solution is relative to the polynomial degree of the finite element space. Recall that in practice adaptive methods with polynomials of degree $\ell$ on domains in $\RR^d$ converge with order DOF$^{-\ell/d}$. In order to obtain this rate with uniform meshes, we would need the solution to belong to $H^{\ell+1}(\Omega)$.

\begin{figure}[h!tbp]

\hfil
\includegraphics[width=.6\textwidth]{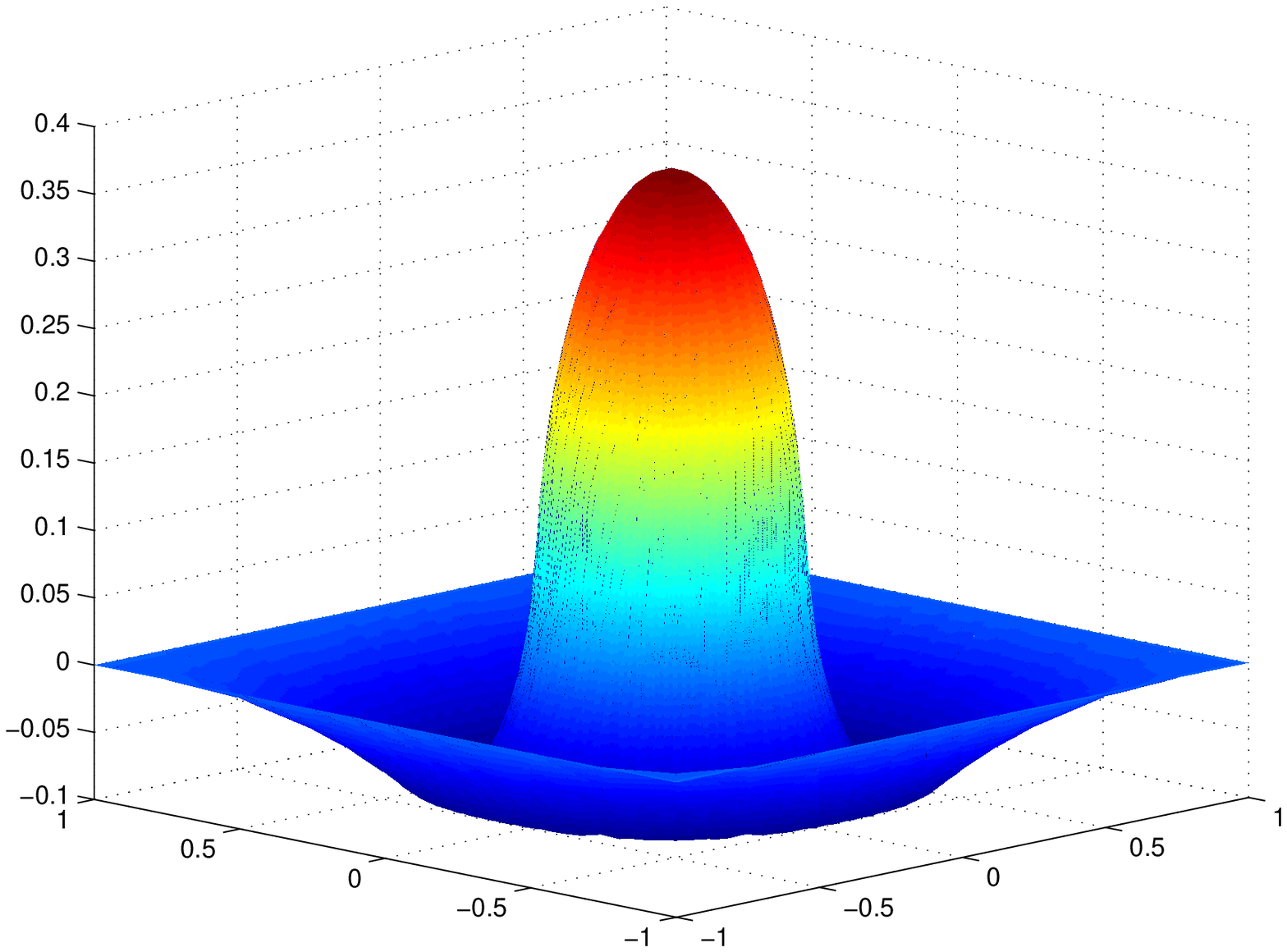}

%\hfil
%\includegraphics[width=.3\textwidth]{figuras/curvature/meshes-degree-1/mesh-00005}
\hfil
\includegraphics[width=.32\textwidth]{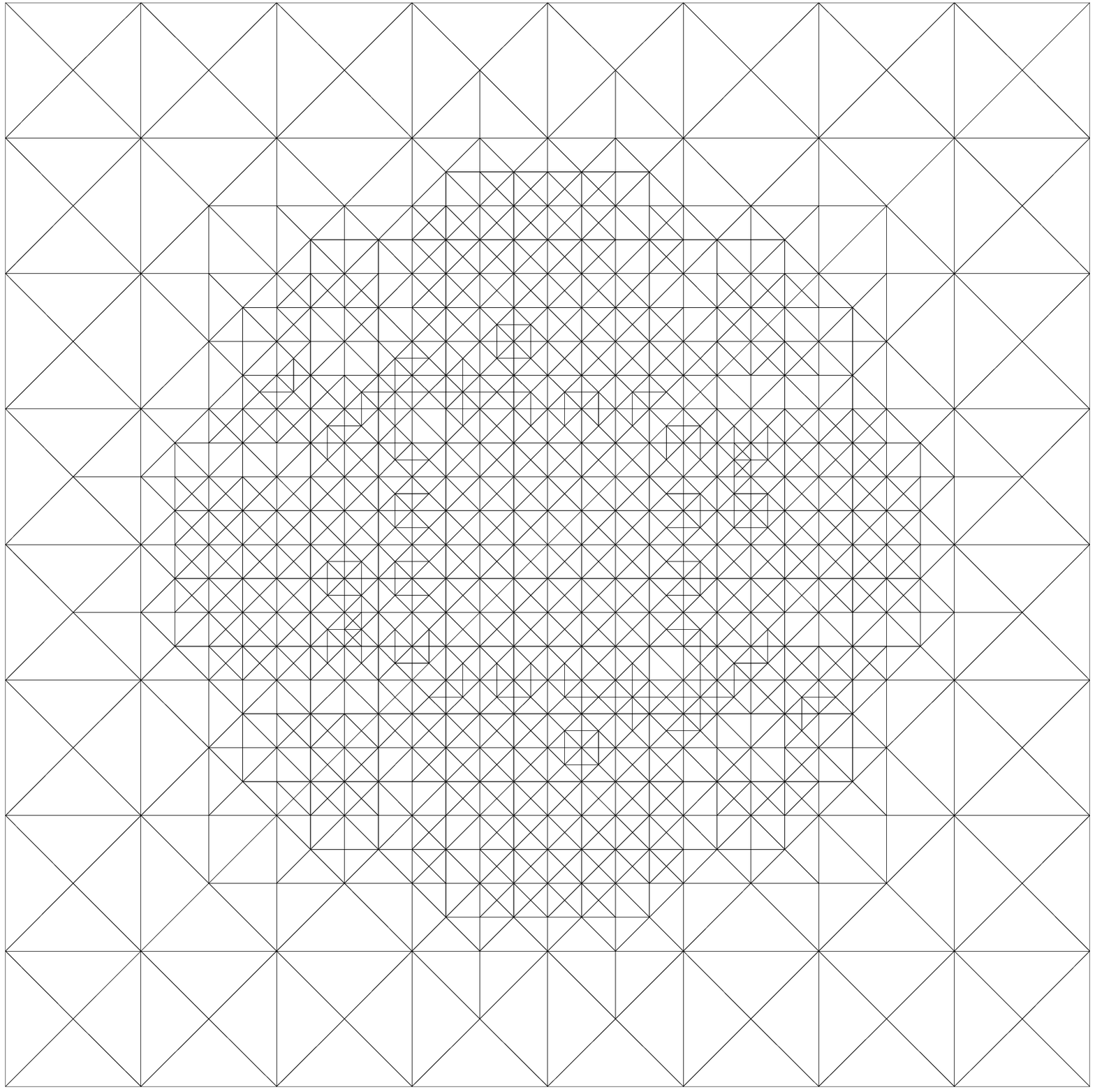}
\hfil
\includegraphics[width=.32\textwidth]{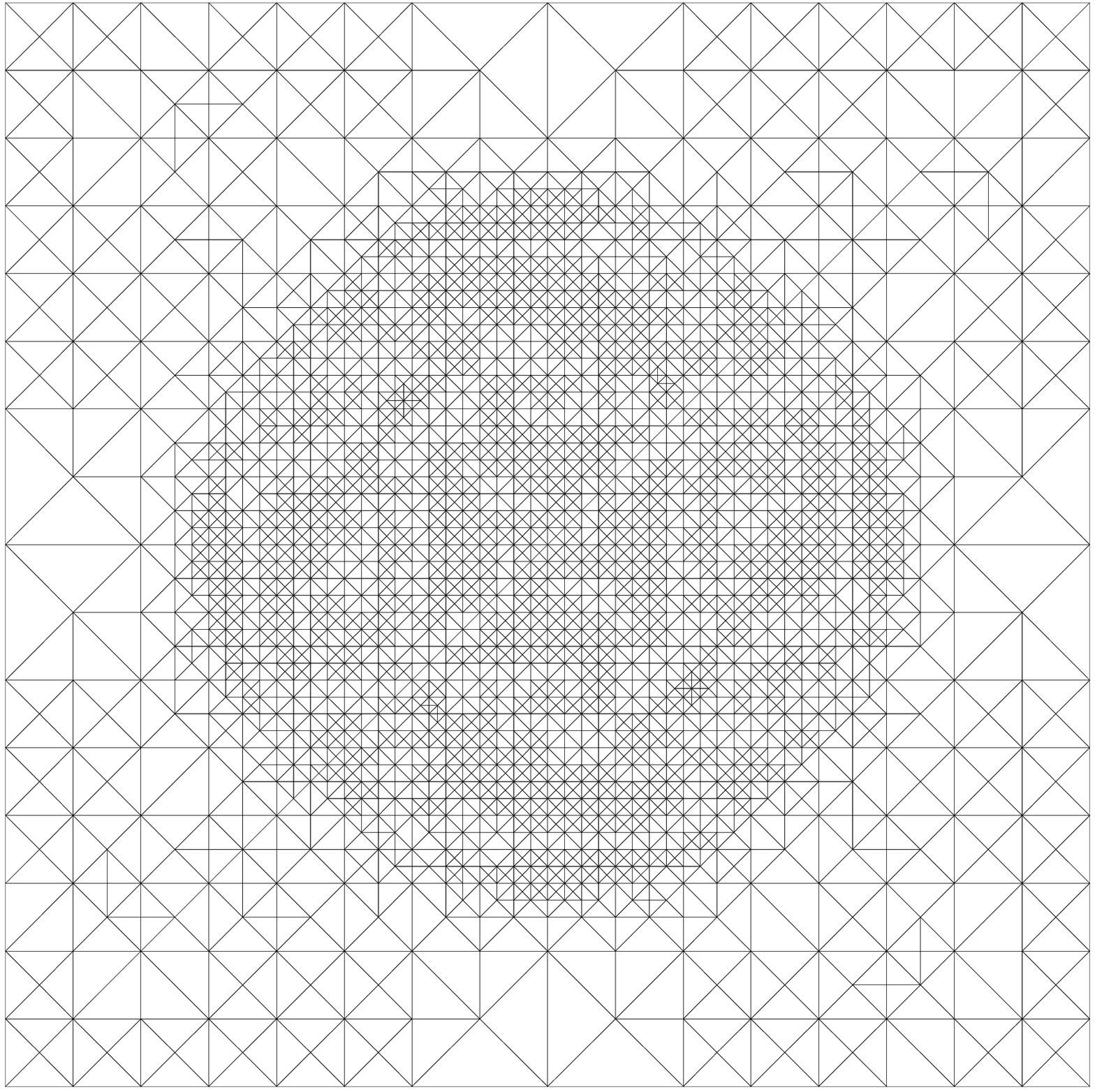}
\hfil
\includegraphics[width=.32\textwidth]{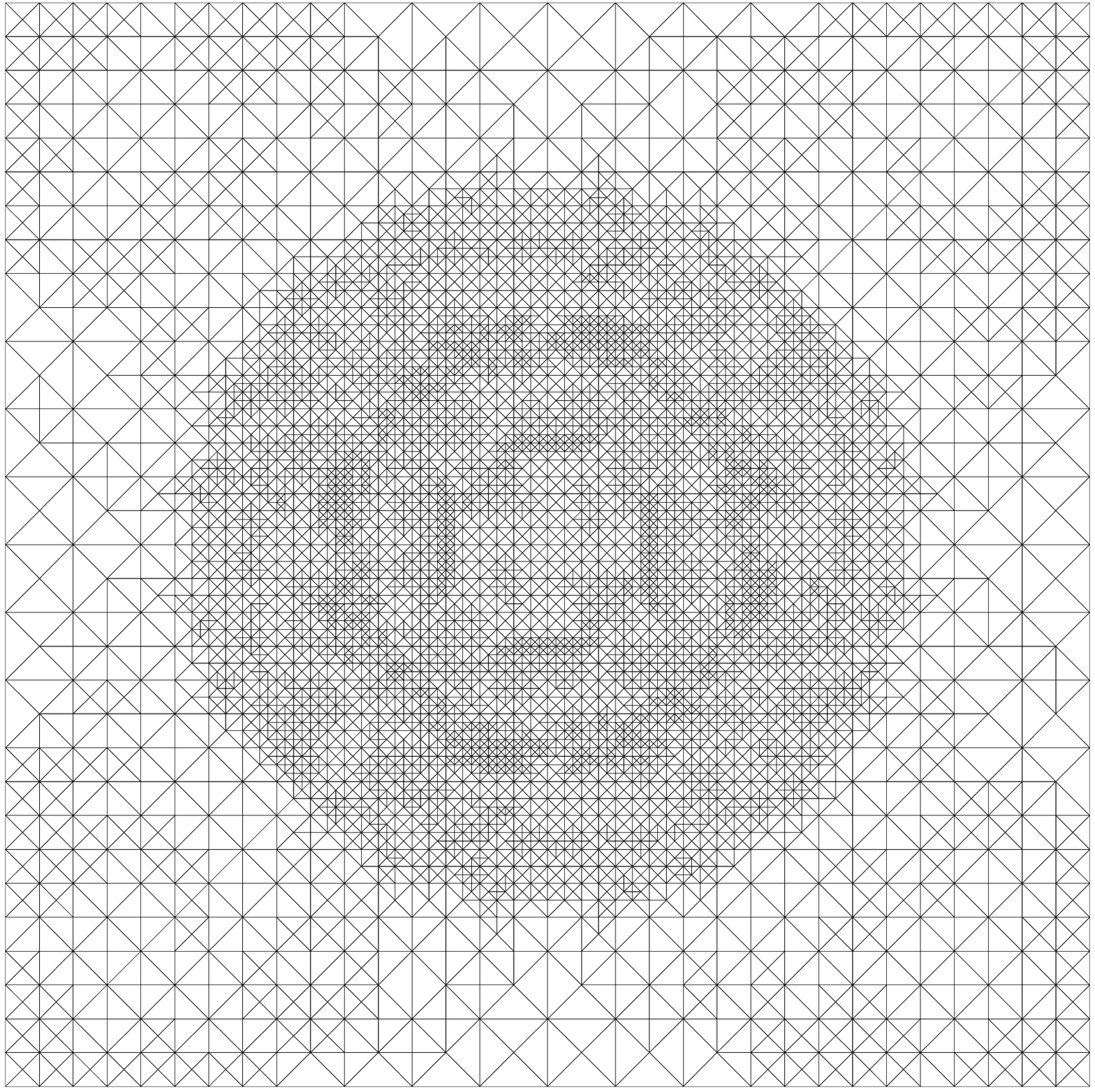}

%\hfil
%\includegraphics[width=.3\textwidth]{figuras/curvature/meshes-degree-2/mesh-00005}
\hfil
\includegraphics[width=.32\textwidth]{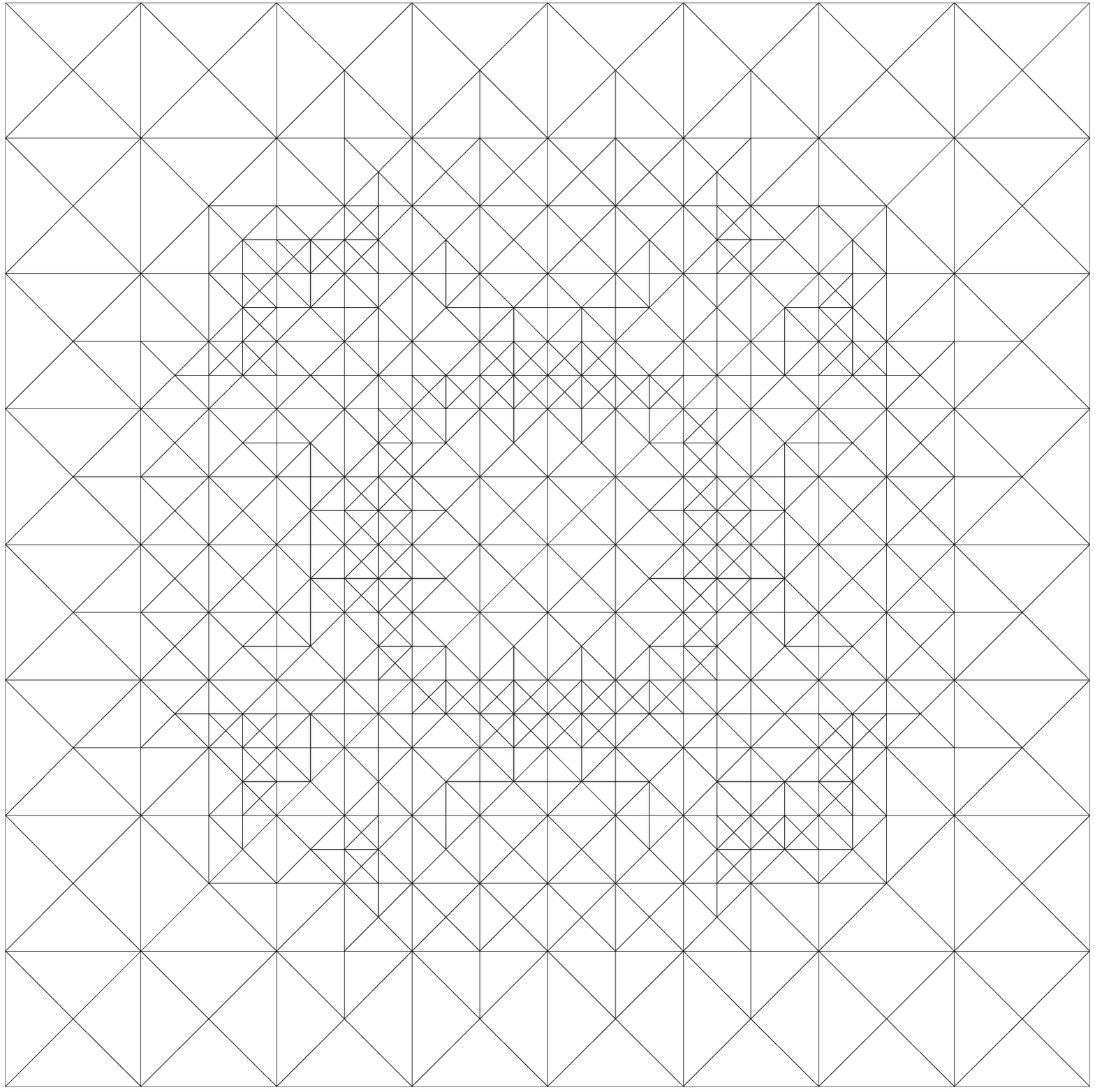}
\hfil
\includegraphics[width=.32\textwidth]{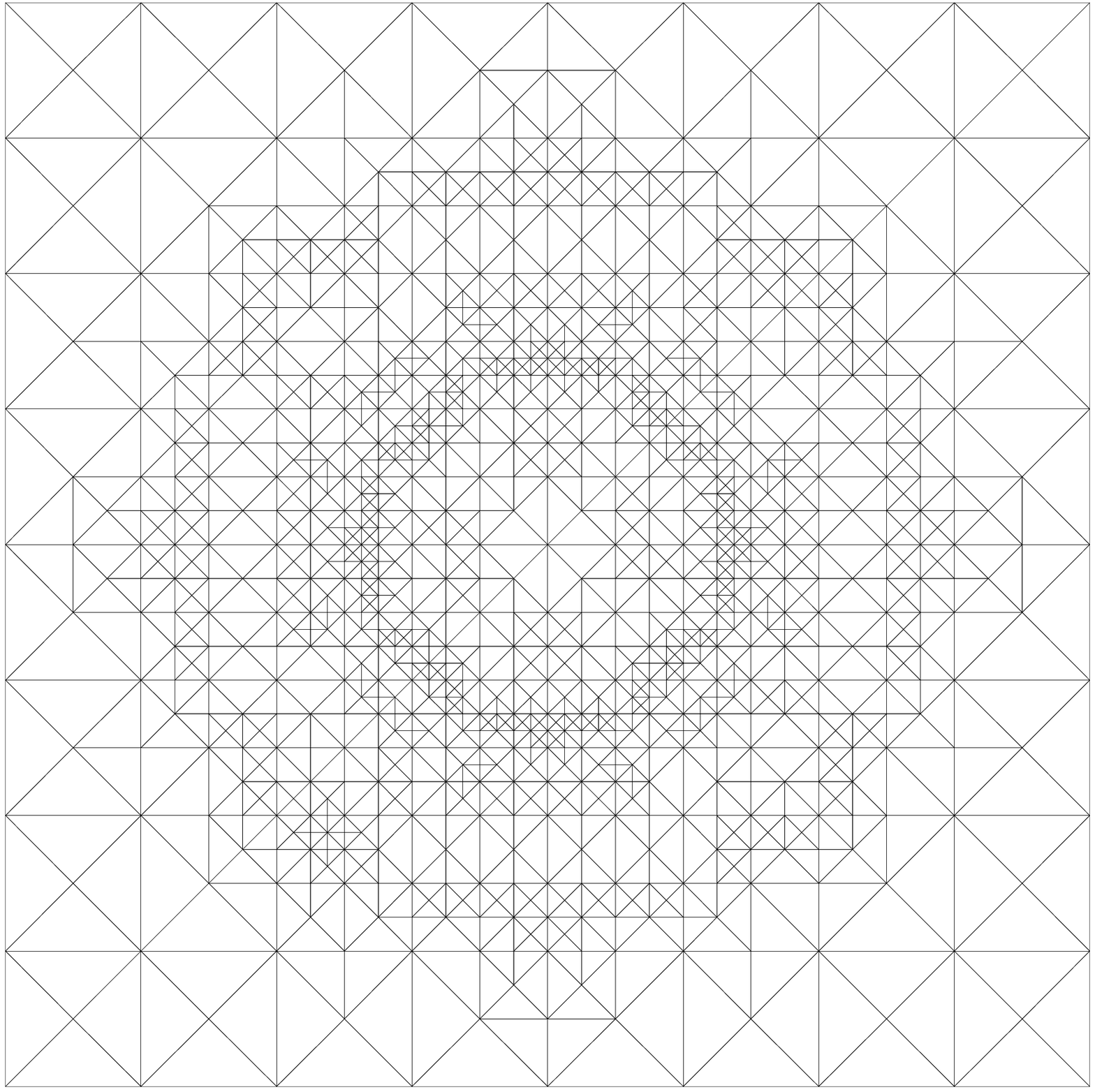}
\hfil
\includegraphics[width=.32\textwidth]{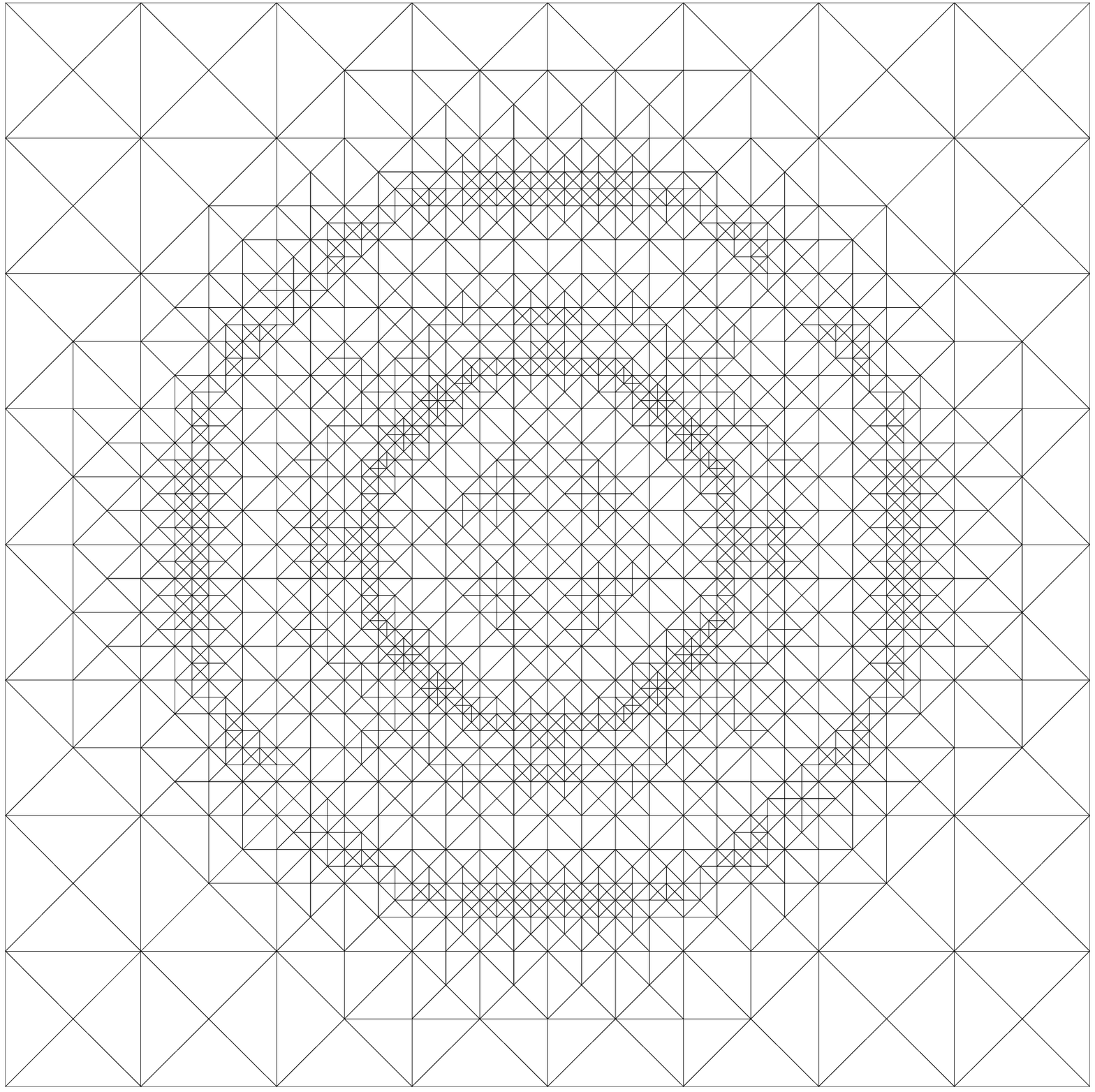}

%\hfil
%\includegraphics[width=.3\textwidth]{figuras/curvature/meshes-degree-3/mesh-00005}
\hfil
\includegraphics[width=.32\textwidth]{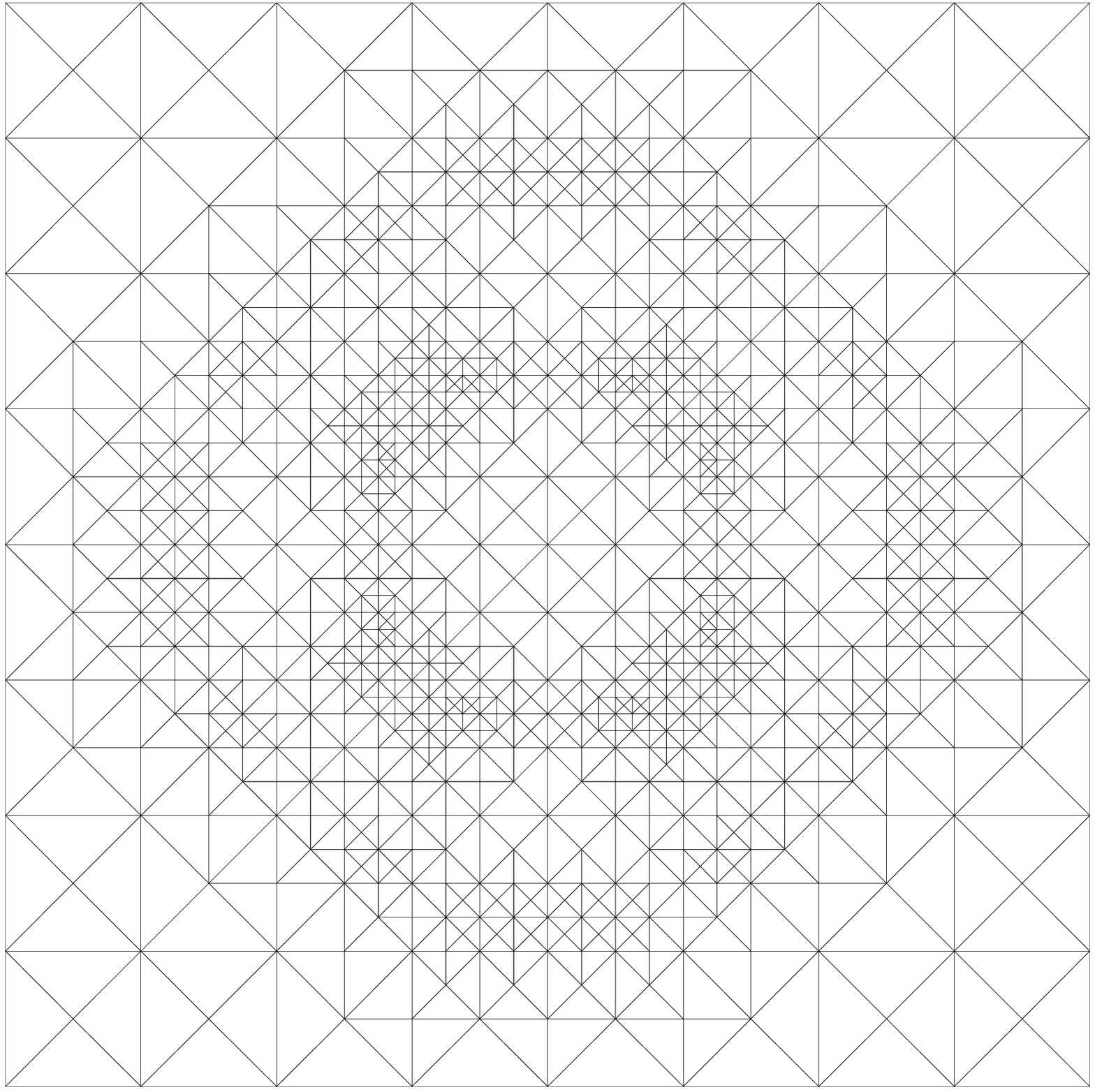}
\hfil
\includegraphics[width=.32\textwidth]{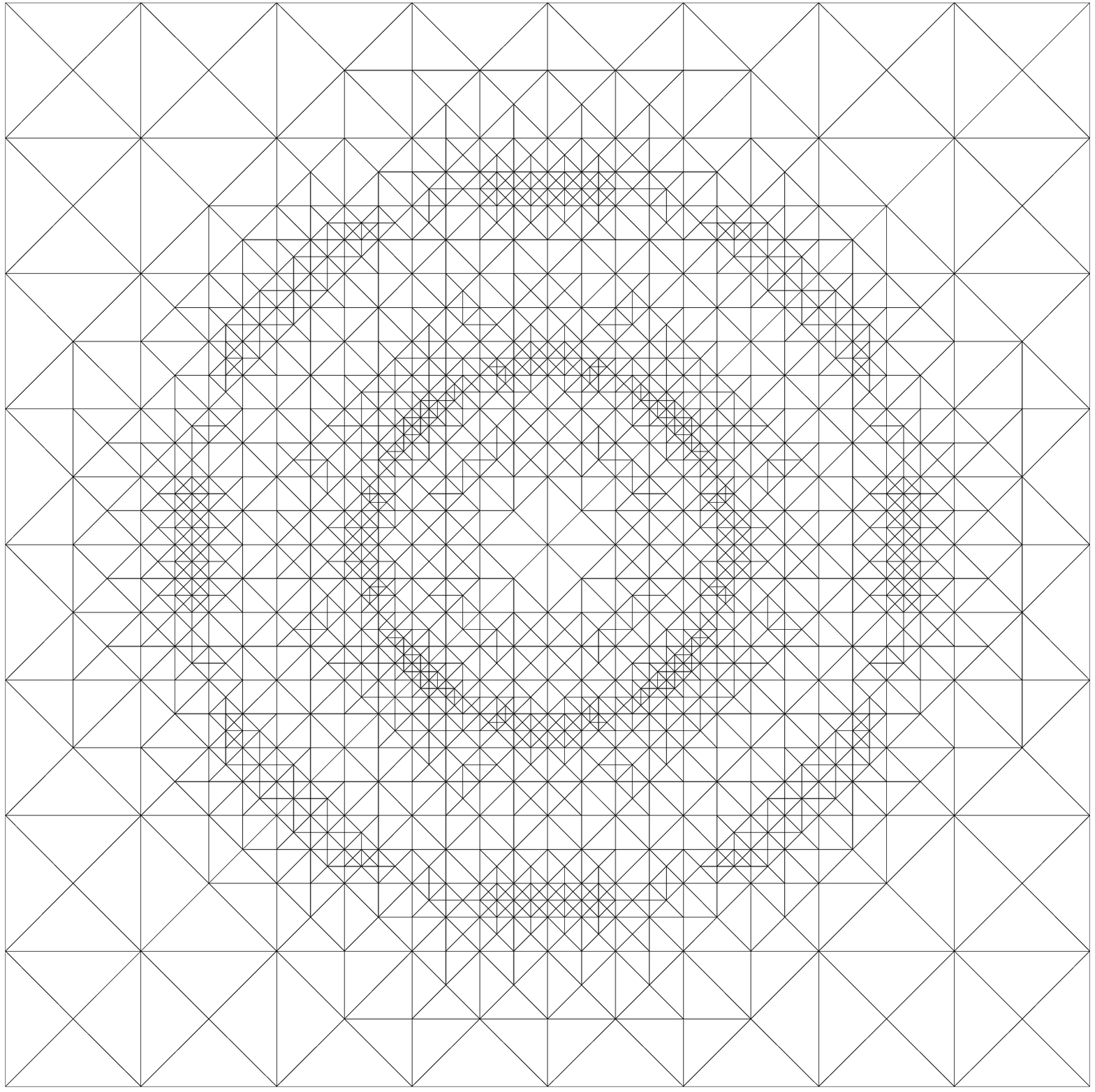}
\hfil
\includegraphics[width=.32\textwidth]{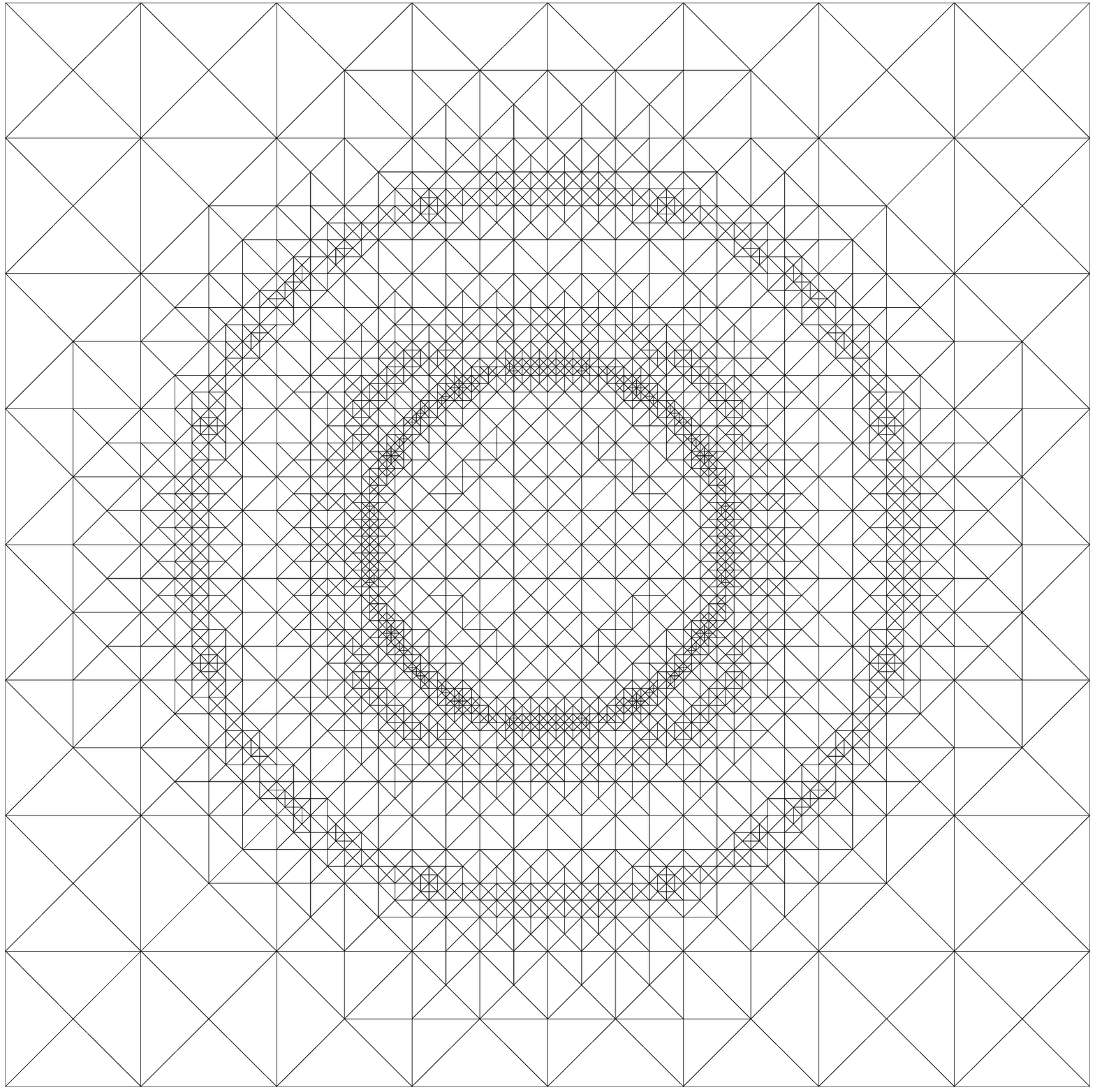}
\caption{\label{F:curvature}Adaptive meshes obtained when solving the prescribed mean curvature equation~\eqref{E:curvature} with polynomials of degree 1 (top), 2 (middle), 3 (bottom). The meshes correspond to iteration count 10 (left), 15 (middle) and 20 (right). It is worth observing the higher grading presented by the meshes corresponding to higher polynomial degree.}
\end{figure}

{\setlength{\parindent}{0pt}
\small

\medskip
\rule{.5\textwidth}{1.5pt}

\medskip
\textbf{Affiliations}
\begin{description}
\item[Eduardo M.\ Garau:] Consejo Nacional de Investigaciones Cient\'{\i}ficas y T\'{e}cnicas and Universidad Nacional del Litoral, Argentina, \url{egarau@santafe-conicet.gov.ar}. 

Partially supported by CONICET (Argentina) through Grant PIP 112-200801-02182, and Universidad Nacional del Litoral through Grant CAI+D PI 062-312.

Address: IMAL, G\"uemes 3450. S3000GLN Santa Fe, Argentina.

\item[Pedro Morin:] Consejo Nacional de Investigaciones Cient\'{\i}ficas y T\'{e}cnicas and Universidad Nacional del Litoral, Argentina, \url{pmorin@santafe-conicet.gov.ar}.

 Partially supported by CONICET (Argentina) through Grant PIP 112-200801-02182, and Universidad Nacional del Litoral through Grant CAI+D PI 062-312.

Address: IMAL, G\"uemes 3450. S3000GLN Santa Fe, Argentina.

\item[Carlos Zuppa:] Universidad Nacional de San Luis, Argentina,
\url{zuppa@unsl.edu.ar}.
 
Partially supported by Universidad Nacional de San Luis through Grant 22/F730-FCFMyN.

Address: Departamento de Matem\'atica, Facultad de Ciencias F\'\i sico Ma\-te\-m\'a\-ticas y Naturales, Universidad Nacional de San Luis, Chacabuco 918, D5700BWT San Luis, Argentina.
\end{description}
}

\end{document}